\documentclass[11pt,reqno]{amsart}
\usepackage {amssymb}
\usepackage {amsmath}
\usepackage {bbm}
\usepackage{amsthm}
\usepackage{mathtools}
\usepackage{graphicx}
\usepackage {amscd}
\usepackage {epic}
\usepackage{array}

\usepackage{mathrsfs}

\usepackage{etaremune}

\DeclareFontFamily{U}{dutchcal}{\hyphenchar\font=-1}
\DeclareFontShape{U}{dutchcal}{m}{n}{ <-> dutchcal-r}{}
\DeclareSymbolFont{dutchletters}{U}{dutchcal}{m}{n}

\DeclareMathSymbol{\mathdutchcal}{0}{dutchletters}{"41} 
\newcommand{\dcal}[1]{\text{\usefont{U}{dutchcal}{m}{n}#1}}

\usepackage[dvipsnames]{xcolor}
\usepackage[colorlinks,citecolor=OliveGreen,linkcolor=Mahogany,urlcolor=Plum,pagebackref]{hyperref}
\usepackage[alphabetic]{amsrefs}
\usepackage{cleveref}

\usepackage{enumerate}
\usepackage{tikz}
\usepackage{tikz-cd}
\usetikzlibrary{arrows.meta}
\usepackage{verbatim,color,geometry}
\usepackage[all]{xy}
\usepackage{enumitem}
\usepackage{bm}
\usepackage{longtable}
\usepackage{aliascnt}
\usetikzlibrary{calc}
\usepackage{textcomp}
\usepackage{tablefootnote}
\usepackage{float}

\newcommand{\mtf}[1]{\mathfrak{#1}}

\newcommand{\ove}[1]{\overline{#1}}
\newcommand{\wt}[1]{\widetilde{#1}}

\newcommand{\cC}{\dcal{C}}
\newcommand{\cD}{\dcal{D}}
\newcommand{\cE}{\dcal{E}}
\newcommand{\cF}{\dcal{F}}
\newcommand{\cG}{\dcal{G}}
\newcommand{\cH}{\dcal{H}}

\newcommand{\cL}{\dcal{L}}
\newcommand{\cM}{\dcal{M}}
\newcommand{\cN}{\dcal{N}}
\newcommand{\cO}{\dcal{O}}
\newcommand{\cP}{\dcal{P}}
\newcommand{\cQ}{\dcal{Q}}
\newcommand{\cR}{\dcal{R}}
\newcommand{\cS}{\dcal{S}}

\newcommand{\cU}{\dcal{U}}
\newcommand{\cV}{\dcal{V}}

\newcommand{\cX}{\dcal{X}}
\newcommand{\cY}{\dcal{Y}}
\newcommand{\cW}{\dcal{W}}

\newcommand{\tw}{\operatorname{tw}}

\newcommand\AAA{{\mathcal{A}}}
\newcommand\BB{{\mathcal{B}}}
\newcommand\CC{{\mathcal{C}}}

\newcommand\GG{{\mathcal{G}}}
\newcommand\HH{{\mathcal{H}}}

\newcommand\MM{{\mathcal{M}}}

\newcommand\PP{{\mathcal{P}}}

\newcommand\TT{{\mathcal{T}}}
\newcommand\UU{{\mathcal{U}}}
\newcommand\VV{{\mathcal{V}}}

\newcommand\XX{{\mathcal{X}}}

\newcommand{\bP}{\mathbb{P}}
\newcommand{\bC}{\mathbb{C}}
\newcommand{\bR}{\mathbb{R}}

\newcommand{\bQ}{\mathbb{Q}}
\newcommand{\bZ}{\mathbb{Z}}

\newcommand{\bN}{\mathbb{N}}

\newcommand{\red}{\mathrm{red}}

\newcommand{\Id}{\mathrm{Id}}

\newcommand{\KSBA}{\mathrm{KSBA}}

\newcommand{\Gr}{\mathrm{Gr}}

\newcommand{\Sym}{\mathrm{Sym}}

\newcommand{\sA}{\mathscr{A}}
\newcommand{\sB}{\mathscr{B}}

\newcommand{\sD}{\mathscr{D}}

\newcommand{\sP}{\mathscr{P}}

\newcommand{\sX}{\mathscr{X}}

\newcommand{\sZ}{\mathscr{Z}}

\DeclareMathOperator{\Aut}{Aut}

\DeclareMathOperator{\vol}{vol}

\DeclareMathOperator{\sta}{s}

\DeclareMathOperator{\Hilb}{Hilb}
\DeclareMathOperator{\lct}{lct}
\DeclareMathOperator{\Spec}{Spec}
\DeclareMathOperator{\spec}{Spec}

\DeclareMathOperator{\mult}{mult}

\DeclareMathOperator{\GL}{GL}
\DeclareMathOperator{\Cho}{Chow}

\DeclareMathOperator{\bfM}{{\textbf{M}}}
\DeclareMathOperator{\Pic}{Pic}

\DeclareMathOperator{\ord}{ord}

\DeclareMathOperator{\Proj}{Proj}

\DeclareMathOperator{\univ}{univ}
\DeclareMathOperator{\Supp}{Supp}

\DeclareMathAlphabet{\mathbbb}{U}{bbold}{m}{n}

\newcommand{\Chow}{\mathrm{Chow}}
\newcommand{\PGL}{\mathrm{PGL}}
\newcommand{\CM}{\mathrm{CM}}
\newcommand{\Hom}{\mathrm{Hom}}
\newcommand{\Bir}{\mathrm{B}}

\newcommand{\bmu}{\bm{\mu}}

\newcommand{\fM}{\mathfrak{M}}

\newcommand{\sC}{\mathscr{C}}

\newcommand{\slct}{\mathrm{slct}}

\newcommand{\Mat}{\mathrm{Mat}}

\newcommand{\sm}{\mathrm{sm}}
\newcommand{\rk}{\mathrm{rk}}

\newcommand{\PCY}{\cP\textbf{CY}}

\numberwithin{equation}{section}

\newtheorem{prop}{Proposition}[section]

\newcommand{\newaliastheorem}[3]{%
  \newaliascnt{#1}{prop}%
  \newtheorem{#1}[#1]{#2}%
  \aliascntresetthe{#1}%
  \crefname{#1}{#2}{#3}%
  \Crefname{#1}{#2}{#3}%
}

\newaliastheorem{thm}{Theorem}{Theorems}
\newaliastheorem{theorem}{Theorem}{Theorems}
\newaliastheorem{lem}{Lemma}{Lemmas}
\newaliastheorem{lemma}{Lemma}{Lemmas}
\newaliastheorem{cor}{Corollary}{Corollaries}
\newaliastheorem{corollary}{Corollary}{Corollaries}
\newaliastheorem{theodef}{Theorem-Definition}{Theorem-Definitions}
\newaliastheorem{prop-def}{Proposition-Definition}{Proposition-Definitions}
\newaliastheorem{convention}{Convention}{Conventions}
\newaliastheorem{conj}{Conjecture}{Conjectures}
\newaliastheorem{conjecture}{Conjecture}{Conjectures}

\theoremstyle{definition}
\newaliastheorem{defi}{Definition}{Definitions}
\newaliastheorem{defn}{Definition}{Definitions}
\newaliastheorem{exa}{Example}{Examples}
\newaliastheorem{exam}{Example}{Examples}
\newaliastheorem{expl}{Example}{Examples}
\newaliastheorem{rem}{Remark}{Remarks}
\newaliastheorem{rmk}{Remark}{Remarks}
\newaliastheorem{remark}{Remark}{Remarks}
\newaliastheorem{que}{Question}{Questions}

\newtheorem{Step}{Step}

\newtheorem{Step_dm}{Step}

\crefname{prop}{Proposition}{Propositions}
\Crefname{prop}{Proposition}{Propositions}

\title{Projective moduli of log Calabi--Yau fibrations over curves}

\author[G. Inchiostro]{Giovanni Inchiostro}
\address{Department of Mathematics, University of Washington, C-138 Padelford, Box 354350, Seattle, WA 98195, USA}
\email{ginchios@uw.edu}

\author[J. Zhao]{Junyan Zhao}
\address{Department of Mathematics, University of Maryland, William E. Kirwan Hall, 4176 Campus Dr, College Park, MD 20742, USA}
\email{jzhao81@umd.edu}

\date{\today}

\begin{document}

\begin{abstract}
  We introduce a new stability condition for log Calabi--Yau fibrations over curves. We prove that it gives rise to a proper Deligne--Mumford stack with a projective coarse moduli space, whose boundary still parametrizes flat fibrations over curves.
\end{abstract}

\maketitle
\tableofcontents

\section{Introduction}
Constructing projective moduli spaces for algebraic varieties is one of the central problems in algebraic geometry. By the minimal model program, every projective variety is conjecturally birational to a fibration whose building blocks are Fano varieties, Calabi--Yau varieties, and canonically polarized varieties. Of these three types, the moduli theory of Fano varieties and canonically polarized varieties has undergone remarkable progress in the past few decades, culminating in the construction of K-moduli spaces and KSBA moduli spaces, respectively; see e.g. \cites{Kol23,Xu25}. In contrast, unlike Fano or canonically polarized varieties, Calabi--Yau varieties do not carry a distinguished polarization, making the construction of moduli spaces substantially more subtle. For certain special classes, such as K3 surfaces, abelian varieties, and a few other families endowed with additional geometric structures, satisfactory moduli theories have been established. In general, however, one can obtain a well-behaved moduli problem only by equipping a Calabi--Yau variety with additional data, typically an ample divisor. This approach was first carried out in specific examples in \cites{Ale02,AE23}. More generally, \cite{KX20} constructed projective coarse moduli spaces for polarized Calabi--Yau pairs, assuming the boundedness result established later in \cite{Bir23}. This was subsequently generalized by \cite{Bir22} to any good minimal models.

\smallskip

The long-term goal is to develop a satisfactory moduli theory for log Calabi--Yau fibrations. These fibrations arise naturally from good minimal models of Kodaira dimension $1, \dots, n-1$, as well as Mori fiber spaces, which correspond to the case of negative Kodaira dimension, equipped with an appropriate boundary divisor. More precisely, for fixed numerical invariants, one would like to construct an algebraic stack admitting a \textit{projective} good moduli space, together with an open substack parametrizing log Calabi--Yau fibrations with mild singularities \((X,B)\to Z,\) and such that the boundary admits a meaningful modular interpretation. When restricted to smooth, or mildly singular, log Calabi--Yau fibrations, this moduli problem is comparatively well understood. The real difficulty lies in compactifying this moduli problem: one needs to identify the right class of degenerations to add at the boundary so that the resulting moduli space is proper, while still admitting a modular interpretation.

\smallskip

In this paper, we take a step in this direction by constructing a \textit{projective} coarse moduli space parametrizing log Calabi--Yau fibrations over curves.

\smallskip
\subsection{Stable log Calabi--Yau fibrations}
The object for which one would like to construct a moduli space is a morphism
\[
  f\colon (X, cD) \to C,
\]
where $X$ is a normal projective variety, $C$ is a smooth projective curve, $D$ is an $f$-ample integral divisor whose every component is horizontal, and $c\in [0,1)$ is a (unique) rational number such that $K_{X} + cD \sim_{\bQ} f^*L$ for some nef $\bQ$-line bundle $L$ on $C$ and that $(X,cD)$ has klt singularities. Typical examples include the following:
\begin{itemize}
  \item $X$ is a good minimal model of Kodaira dimension $1$, the morphism
    $f\colon X \to C$ is the ample model, and $D$ is any relatively ample divisor, with $c = 0$. This is the higher-dimensional analogue of minimal elliptic surfaces;
  \item $X$ is a Calabi--Yau variety admitting a fibration $f\colon X \to \bP^1$, and $D$ is any divisor that is $f$-ample. In dimension two this is exactly
    the case of elliptic K3 surfaces, and in higher dimensions there are many such
    examples: for instance, there exist K3-fibered or abelian-fibered Calabi--Yau threefolds;
  \item $f\colon X \to C$ is a Fano fibration, and $D$ is a divisor on $X$ that is
    proportional to $-K_X$ up to a pullback from $C$. The simplest examples of this
    form are Mori fiber spaces over curves (e.g. projective bundles) together with a
    horizontal divisor.
\end{itemize}
The morphisms described above are ``good'' objects, and they occupy an open substack of the moduli stack we are trying to construct; we call these log Calabi--Yau fibrations
\emph{non-degenerate} (cf. \Cref{defn:non-degenerate-type}). 

The core difficulty in constructing a projective moduli space is to impose ``correct'' geometric conditions on the objects we will parametrize, in such a way that non-degenerate log Calabi--Yau fibrations satisfy these conditions and, most importantly, that for a family of non-degenerate log Calabi--Yau fibrations over a punctured curve, there exists an extension by an object satisfying these conditions. The following definition is the one we would like to impose, and it is the main definition of this paper: it is this definition that gives rise to a projective coarse moduli space (cf. \Cref{thm:main-theorem}).

\begin{defn}\label{defn:main-defn}
A \emph{stable log Calabi--Yau fibration} $f:(\cX,\cD)\rightarrow \cC$ consists of
\begin{itemize}
    \item a twisted nodal curve $\cC$ (cf.~\cite[4.1.2]{AV02} and \Cref{def_twisted_curve});
    \item a flat projective morphism $f\colon \cX\rightarrow \cC$ from a demi-normal Deligne--Mumford stack $\cX$ with connected fibers;
        \item a $\bQ$-Cartier Weil divisor $\cD\subseteq \cX$,
\end{itemize}
with coarse space $f_X\colon(X,D)\to C$, satisfying the following conditions:
\begin{itemize}
    \item[(LCY)] there exist a positive integer $m$, a unique rational number $c\in [0,1)$, and a line bundle $\cL$ on $\cC$ such that
    \[
    m(K_{\cX}+c\cD)\sim f^*\cL;
    \]

    \item[(Qmap)] there exists an open substack $\cU\subseteq \cC$ whose complement consists of finitely many \emph{smooth} points such that
    \[(\cX,(c+\epsilon_0)\cD)|_{\cU}\ \longrightarrow \ \cU\]
    is KSBA-stable for some rational number $0<\epsilon_0\ll1$;

    \item[(Sing)] for every smooth point $p\in \cC$, the semi-log canonical threshold
    $\slct(\cX,c\cD;\cX_p)$ is positive;

    \item[(Stab)] for every sufficiently small rational number $0<\epsilon\ll1$, the $\bQ$-line bundle
    \[
    K_{X}+(c+\epsilon^2)D+\epsilon\,f_X^*\lambda_{\Chow,\epsilon_0}
    \]
    is ample, where $\lambda_{\Chow,\epsilon_0}$ denotes the descent of the \emph{Chow line bundle} (cf. \Cref{defn:Chow-line-bundle}) associated to the morphism $\big(\cX,(c+\epsilon_0)\cD\big)\to \cC$.
\end{itemize}
\end{defn}
It is worth noting that \textup{(LCY)} together with \textup{(Stab)} implies that $D$ is relatively ample over $C$. Moreover, the rational number $c\in[0,1)$ appearing in \textup{(LCY)} is uniquely determined, so we shall write a stable log Calabi--Yau fibration as $f\colon(\cX,c\cD)\to\cC$; we discuss these conditions in greater detail in \Cref{sec:stability-conditions}.

\smallskip

A stable log Calabi--Yau fibration is said to be \emph{of non-degenerate type} (cf.~\Cref{defn:non-degenerate-type}) if it admits a deformation to a non-degenerate log Calabi--Yau fibration. This condition should be viewed as a natural generalization of smoothability. We impose it primarily for technical reasons: all examples and applications considered in this paper are of non-degenerate type, and restricting to this class allows us to avoid several technical complications.

\smallskip

With this terminology in place, we can now state our main result.

\begin{theorem}\label{thm:main-theorem}
For any fixed numerical invariant $\Phi$ \textup{(cf.~\Cref{def_num_data})}, there exists a proper Deligne--Mumford stack $\overline{\MM}_{\Phi}$ of finite type over $\bC$ parametrizing stable log Calabi--Yau fibrations of non-degenerate type with numerical invariant $\Phi$. Moreover, $\overline{\MM}_{\Phi}$ admits a projective coarse moduli space $\overline{\mathfrak{M}}_{\Phi}$.
\end{theorem}

The main insight behind \Cref{thm:main-theorem} is that imposing the right notion of stability, namely \Cref{defn:main-defn}, is already enough to produce a well-behaved moduli theory. Once this definition was in place, the rest of the paper is devoted to verifying that it indeed does so; while this verification is often technical, the strategy at each step is guided by the geometric meaning of the stability conditions, as explained in \Cref{sec:stability-conditions} below.
\subsection{Stability conditions}\label{sec:stability-conditions}

The condition (LCY), which stands for ``Log Calabi--Yau'', is the most natural part of Definition \ref{defn:main-defn} and requires no further explanation. We therefore focus on the necessity and geometric meaning of the remaining three conditions (Qmap), (Sing), and (Stab), which stand for ``Quasimap'', ``Singularity'', and ``Stability'', respectively.

\subsubsection{Canonical bundle formula and Chow line bundles}
The conditions \textup{(Sing)} and \textup{(Stab)} in \Cref{defn:main-defn} are the key ingredients in the proof of \Cref{thm:main-theorem}; we explain their roles in more detail below.

Let $f\colon (X,cD)\to C$ be a non-degenerate log Calabi--Yau fibration. The canonical bundle formula gives
\[
K_{X/C}+cD \ \sim_{\bQ} \ f^*(B+\bfM),
\]
where
\[
B\ \coloneqq\ \sum_{p\in C}\bigl(1-\lct(X,cD;X_p)\bigr)[p]
\]
is the discriminant divisor and $\bfM$ is the moduli part; see e.g.\ \cite{Kol07}. The condition \textup{(Sing)} is equivalent to the pair $(C,B)$ being klt, and can also be phrased as requiring every log canonical center of $(X,cD)$ to dominate $C$. This condition is crucial for the separatedness of the moduli stack. Already in dimension two there are examples of this phenomenon: a punctured family of elliptic K3 surfaces with a section can admit more than one limit satisfying all conditions except (Sing), the limits being related by blowing up log canonical centers contained in fibers.
The condition \textup{(Stab)} implies that $K_X+cD$, and hence $K_C+B+\bfM$, is nef. When $K_C+B+\bfM$ is nef but not ample, for example when $X\to C=\bP^1$ is an elliptic K3 surface and $D$ is a multi-section with $c=0$, the condition \textup{(Stab)} is equivalent to the ampleness of the Chow $\bQ$-line bundle $\lambda_{\Chow,\epsilon_0}$.

If $f|_D\colon D\to C$ is a well-defined family of divisors, then $\lambda_{\Chow,\epsilon_0}$ agrees with the Chow--Mumford (CM) line bundle associated to the polarized family of pairs
\[
f\colon (X,cD;L)\to C,
\qquad L=K_X+(c+\epsilon_0)D.
\]
Introduced in \cite{Tia97}, the CM line bundle plays a fundamental role in the theory of K-stability, and is known to be ample on KSBA moduli spaces (cf.\ \cite[Theorem~1.1]{PX17}). Moreover, a key result of \cite[Theorem~3.3]{WX14} shows that its degree is minimized by the KSBA-stable limit. Consequently, for a fixed, sufficiently small rational number $\epsilon_0>0$ as in (Qmap) of \Cref{defn:main-defn}, the Chow line bundle $\lambda_{\Chow,\epsilon_0}$ is ample whenever the family $f\colon (X,(c+\epsilon_0)D)\to C$ either has maximal variation or contains a KSBA-unstable fiber. In particular, when $K_C+B+\bfM$ is nef but not ample and either condition holds, the divisor $K_X+(c+\epsilon^2)D+ f^*(\epsilon\lambda_{\Chow,\epsilon_0})$ is ample for all $0<\epsilon\ll 1$.

For our purposes, one of our key technical inputs is a generalization of the minimizing property of \cite[Theorem~3.3]{WX14} to the setting where $f|_D\colon D\to C$ need not be a well-defined family of divisors, for instance when $D$ contains components of fibers of $f\colon X\to C$ (cf.~\Cref{thm_generalize_wx}), which may be of independent interest. This generalization yields an equivalent, more concrete reformulation of (Stab):
\begin{enumerate}
    \item[(Stab$'$)] the line bundle $\cL$ in (LCY) is nef, and for every component $\cR$ of $\cC$ along which $\cL$ fails to be ample, the restricted family $(\cX,\cD)|_{\cR}\to \cR$ is not isotrivial, i.e.\ it has at least two non-isomorphic closed fibers.
\end{enumerate}
This minimizing property is also the key ingredient in proving the properness of the moduli stack.

\subsubsection{Stackiness of stable log Calabi--Yau fibrations and the quasimap condition}
A further essential ingredient of \Cref{defn:main-defn} is the quasimap condition (Qmap), which is intimately related to the stacky structure of $\cC$.

The condition (Qmap) implies that there exists a natural morphism from $\cU$ to the moduli stack $\PCY$ of polarized log Calabi--Yau pairs (cf.\ \cite{KX20} and \Cref{defn:moduli-stable-CY}). This morphism may be regarded as a rational map from $\cC$ to $\PCY$, or, if one has a suitable enlargement of $\PCY$, a quasimap in the sense of \cites{qmap1,qmap2}. This condition is useful because it guarantees that the geometry of $\big(\cX,(c+\epsilon)\cD\big)\to \cC$ over the nodes of $\cC$ is well behaved: it is a KSBA-stable family, and in particular the morphism $\cX\to \cC$ is flat. The price of this control is the presence of stackiness at certain nodes.

This stackiness is both mild and unavoidable. On the one hand, requiring the analogous condition in \Cref{defn:main-defn} with the stackiness removed imposes, in general, too strong a constraint: the resulting moduli stack would then fail to be proper. On the other hand, the stackiness occurring at the nodes is fully determined by the coarse space: for each branch $\cR$ of a given node, it is uniquely prescribed by the unique root stack that needs to be performed on the coarse space of $\cR$ in order to obtain a representable morphism from $\cR$ to $\PCY$; see \Cref{lem:uniquely-determined}.

\subsection{Comparison with existing moduli theories}

Before sketching the proof of \Cref{thm:main-theorem}, we compare the moduli stack constructed here with several related existing moduli theories, and to explain why none of them can be applied directly to our moduli problem.

\subsubsection{Moduli of maps}

Twisted stable maps are a useful tool for studying fibrations over curves. Morphisms from a twisted curve of fixed genus $g$ to a given proper Deligne--Mumford stack $\MM$, with fixed image class $\beta$ and satisfying an appropriate stability condition, form a projective coarse moduli space $\ove{\fM}_g(\MM;\beta)$ (cf. \cite{AV02}). A log Calabi--Yau fibration $f:(X,cD)\rightarrow C$, however, is in general not a locally stable family even when non-degenerate, so $C$ need not admit a morphism to the moduli stack $\MM\coloneqq \PCY$ parametrizing the fibers. One way to address this is to pass to the theory of quasimaps: enlarge $\MM$ to a stack $\wt{\MM}\supseteq\MM$ that also parametrizes the degenerate fibers of $f$, subject to the requirement that $\wt{\MM}$ admit a projective good moduli space. When $\MM$ satisfies the hypotheses needed for a space of stable quasimaps to exist, this quasimap space can be used to compactify the space of certain non-degenerate stable log Calabi--Yau fibrations; instances of this idea appear in \cite{twisted_map_2}. The difficulty with this approach is that constructing an enlargement $\MM\subseteq \wt{\MM}$ satisfying the required hypotheses is itself highly non-trivial. We therefore take a different route: we adapt the strategy used to construct spaces of stable quasimaps, but replace the existence of a target stack $\wt{\MM}$ with tools from birational geometry.

\subsubsection{Moduli of stable minimal models}

A second relevant theory is the moduli of stable minimal models constructed in \cite{Bir22} and discussed further in \Cref{sec:moduli-of-stable-minimal-models}; its construction is guided by the same philosophy as KSBA moduli theory. Its objects are good minimal models with ample model $(X,B)\rightarrow Z$, together with a divisor $A$ such that $(X,B+\epsilon A)$ is KSBA-stable for $0<\epsilon\ll1$; here $Z$ is determined by the pair $(X,B)$ rather than being part of the data, whereas $A$ is part of the data. This theory is designed around a different collection of examples than the ones we have in mind, and several cases of interest to us fall outside its scope for this reason. For instance, it does not directly apply to the case of an elliptic K3 surface $X\rightarrow \bP^1$ with $B=\emptyset$ and $A$ a section, since here the ample model of $X$ is a point rather than $\bP^1$ and $K_X+A\sim A$ fails to be ample. Likewise, when $(X,B)$ has Kodaira dimension $1$, the base $Z$ must be a nodal curve, and questions such as the flatness of $X\rightarrow Z$, its equi-dimensionality over the smooth points of $Z$, and the behavior of $(X,B)\rightarrow Z$ over the nodes of $Z$ lie outside the scope of that theory as currently developed. 

At the same time, one significant advantage of the moduli of stable minimal models is that $Z$ may have arbitrary dimension, and it is natural to ask whether our theory can likewise be extended to fibrations over higher-dimensional bases. So far, however, we lack the tools for a straightforward generalization to the higher-dimensional case. For example, the theory of maps from higher-dimensional varieties to a fixed moduli stack, which we rely on heavily in this paper for the case of curves, remains poorly understood.

\subsubsection{Moduli of K-stable Calabi--Yau fibrations over curves}

More recently, a different approach to the moduli of Calabi--Yau fibrations over curves has been developed using K-stability; see \cites{Haha25b,HaHa25,Ha25,Hat26}. There, they construct a coarse moduli space parametrizing polarized klt-trivial fibrations $f\colon (X,A)\rightarrow C$ over curves that are \emph{uniformly adiabatically K-stable}, a notion introduced in \cite{Ha24} that detects the existence of a certain cscK metric. Under additional assumptions, Hashizume and Hattori prove in \cites{Haha25b,Hat26} that this moduli space is quasi-projective. Compared with what one would ultimately like to have, the main property that remains to be established is properness: the resulting stack is known to be separated, but properness has not yet been established. The history of K-moduli of Fano varieties is instructive here: even in that more developed setting, establishing properness proved to be considerably more difficult than establishing the other foundational properties (cf. \cite{LXZ22}). This reflects a feature of K-stability itself: properness calls for filling in every one-parameter degeneration with a K-polystable limit, and verifying K-stability of an object tends to be hard. We would also like to point out that the use of quasimaps in the aforementioned work is essential, particularly in the proof of quasi-projectivity. For us instead, the proof of projectivity follows from an application of Koll\'ar's ampleness lemma.


\medskip The relation between the moduli space constructed in this paper and these other theories, as well as applications and examples, will be explored in future work. In examples where more than one of these theories applies, it would be interesting to construct morphisms between the resulting moduli
spaces and compare their boundaries. We would also like to mention that, in the case of surfaces, there is already a substantial body of work \cites{la2002explicit, ABtwisted, inchiostro_elliptic,ascher_invariance_pluri,AB_del_pezzo,ABBDILW23,AB_k3,BL24,ISZ25,IZ25,BFHIMZ24}; these results provide valuable foundational examples that inform the general theory developed in this paper.

\subsection{Outline of the paper}\label{sec:Sketch of proofs}

\subsubsection{Strategy of the proofs} 
The proofs of properness, separatedness, and boundedness of the stack are not independent of each other; the arguments are intertwined. We start with proving the universal closedness of the stack, for which we verify the valuative criterion. Given a stable log Calabi--Yau fibration $(X_{\eta},cD_{\eta})\to C_{\eta}$ over the generic point $\eta$ of a DVR $R$, to get an extension to the whole DVR, we first construct a filling with good geometric structure, and then perform stable reduction by running a minimal model program. To get a good filling, we apply the properness of the moduli of twisted stable maps. However, since $(X_{\eta},(c+\epsilon) D_{\eta})\to C_{\eta}$ is not necessarily locally stable, one cannot view it as a morphism $C_{\eta}\to \PCY$. Instead, we perform a root stack $\cC^{\sta}_{\eta}\rightarrow C_{\eta}$ to get a morphism $\cC^{\sta}_{\eta}\rightarrow \PCY$, which induces a new family $(\cX^{\sta}_{\eta},(c+\epsilon) \cD^{\sta}_{\eta})\to \cC^{\sta}_{\eta}$. Applying properness of the moduli of \emph{pointed} twisted stable maps, one obtains an extension to $\Spec R$ for this new family. This, however, is not an extension of the original family, and producing one requires the most technical part of the argument: a delicate sequence of birational modifications that eventually yields an extension of the original family over $R$. The stable reduction step is analogous to \cite[Proposition 3.2]{ISZ25}, with some generalization of the techniques to adapt to our setup. To verify that the limit is stable (cf.~\Cref{defn:main-defn}), the Conditions (LCY), (Qmap) and (Sing) more or less follow from the construction, while the Condition (Stab) is the place requiring a genuinely new ingredient: on those components $\Gamma$ of the central fiber of $C\rightarrow \Spec R$ on which $K_C+B+\bfM$ is nef but not ample, the required ampleness must instead come from the positivity of the Chow line bundle. This positivity follows from our generalization (cf.~\Cref{thm_generalize_wx}) of the minimizing property of \cite[Theorem~3.3]{WX14}: this stability forces the family over $\Gamma$ either to have maximal variation or to contain a fiber with bad singularities, and both possibilities contribute positively to the degree of the Chow line bundle on $\Gamma$.

To verify the valuative criterion of separatedness, for any two extensions $(\cX_i,\cD_i)\rightarrow \cC_i\rightarrow \Spec R$ with isomorphic generic fibers, we first show that $\cC_1$ and $\cC_2$ are isomorphic, where the condition (Sing) plays an essential role. Then the isomorphism of the two families is straightforward: over the generic points of the central fiber of $\cC_i$, the isomorphism follows from the separatedness of the KSBA-moduli, since the fibers over these points are both KSBA-stable for $i=1,2$; one then extends this isomorphism, defined away from a locus of codimension at least two, using the standard $S_2$-property of the section rings. This, however, does not prove separatedness, since one still needs to show that this isomorphism uniquely extends the given isomorphism of the generic fiber. To this end, we may assume that the two isomorphic extensions are the ones constructed in the proof of universal closedness, and show that, for this family over $\Spec R$, any automorphism of the generic fiber extends uniquely to the whole family, by proving this at each step of the universal closedness argument.

To prove boundedness, one first proves the boundedness of stable log Calabi--Yau fibrations with a fixed numerical invariant $\Phi$ and with the base curve having a bounded number of irreducible components. As a corollary of the argument used to prove universal closedness, one can bound the number of irreducible components appearing in any stable log Calabi--Yau fibration limit, which completes the proof of boundedness. As for the proof of algebraicity of the stack and the projectivity of the coarse moduli space, the main new ingredient throughout the whole proof is that one needs to keep track of the morphism $(X,cD)\rightarrow C$ instead of only $(X,cD)$ and $C$, since the morphism is part of the data; this is also the main difference between the moduli stack constructed in this paper and the moduli space in \cite{Bir22}. For the proof of algebraicity, we incorporate the morphism into the moduli problem by using an appropriate Hom-stack. For the proof of projectivity, we refine the structure groups of certain vector bundles so that an ampleness lemma (cf.~\Cref{thm:ampleness-lemma}) can be applied. To this end, we slightly generalize Koll\'{a}r's ampleness lemma (cf.~\cites{kol90,KP17}) in Appendix \ref{appendix:Ampleness Lemma} by removing a technical assumption, so that it applies to our setting; this generalization may have further applications. The essential reason why \cite{kol90,KP17} cannot be applied directly is that the morphism $(X,cD)\rightarrow C$ is not determined by $(X,cD)$ alone.

\subsubsection{Organization of the paper}

In \Cref{sec:prelim}, we collect the necessary preliminaries, with particular emphasis on the moduli theory of stable minimal models developed in \cite{KX20,Bir22}, as well as the Chow and CM line bundles. We also generalize the minimizing property established in \cite{WX14}. In \Cref{sec:stable-lCY}, we introduce the notion of stable log Calabi--Yau fibrations and establish their basic properties.

In \Cref{Valuative criterion}, we verify the valuative criteria for universal closedness and separatedness. The argument concerns families over DVRs and is independent of the boundedness and algebraicity of the moduli stacks. Compared with the later sections, this section contains a number of new ideas and better illustrates the stability condition that we impose. We also establish a key property of all stable limits of non-degenerate log Calabi--Yau fibrations, which plays a role in proving the boundedness of stable log Calabi--Yau fibrations of non-degenerate type in \Cref{sec:Boundedness of stable log Calabi--Yau fibrations}.

In \Cref{sec:Representability by a Deligne--Mumford stack}, we prove that the moduli stack is algebraic and Deligne--Mumford; its properness then follows immediately from the results of the previous two sections. Finally, in \Cref{sec:projectivity}, we prove that the coarse moduli space is projective.

\subsection*{Acknowledgments}
We thank Dori Bejleri, Giulio Codogni, Philip Engel, Stefano Filipazzi, Yuchen Liu, Andrea Maffei, Zsolt Patakfalvi, Roberto Svaldi, Xiaowei Wang, Lingyao Xie, Chenyang Xu, and Ziquan Zhuang for helpful conversations, and Caucher Birkar, Kenta Hashizume, and Masafumi Hattori for comments on an earlier version of this manuscript. We are also grateful to J\'{a}nos Koll\'{a}r and S\'{a}ndor Kov\'{a}cs for helpful conversations regarding the ampleness lemma. Part of this work was completed during a visit of GI to the Brin Mathematics Research Center at the University of Maryland. GI was supported by NSF grant DMS-2502104. JZ was supported by an AMS-Simons Travel Grant.

\subsection*{Conventions} We work over the field $\bC$ of complex numbers; we expect all the results to hold for any algebraically closed field $k$ of characteristic 0. 

\begin{enumerate}
   \item Throughout this paper, we denote algebraic stacks (resp.~schemes) over $\bC$ by calligraphic (resp.~Roman) letters, e.g.~$\CC, \XX$ (resp.~$C, X$). Families of algebraic stacks (resp.~schemes) over a base scheme $S$ are denoted by Dutch calligraphic (resp.~script) letters, e.g.~$\cC, \cX$ (resp.~$\sC, \sX$).

\item Let $(R, \eta, \xi)$ be a triple consisting of a discrete valuation ring (abbr. DVR), together with its generic point $\eta$ and special point $\xi$ of $\Spec R$. Throughout this paper, we refer to such a triple $(R, \eta, \xi)$ simply as a DVR, and we denote a uniformizer of $R$ by $\varpi$.
\item For a family of objects $\sX\rightarrow \Spec R$ over a DVR $(R,\eta,\xi)$, we will denote by $\sX_\eta$ (resp. $\sX_{\xi}$) the generic (resp. special) fiber.
\item For an algebraic stack $\XX$, we denote by $\XX^{\red}$ its associated reduced substack.
\end{enumerate}

\section{Preliminaries}\label{sec:prelim}

In this section we recall a few results on twisted stable maps, moduli of varieties of log-general type and stable minimal models, and finally the CM and Chow divisor.

\subsection{Moduli of twisted stable maps}
In this subsection, we introduce twisted stable maps, a key tool needed for the construction of our moduli space. We first report two definitions of \cite{AV02}.
\begin{defn}[Twisted curves; {\cite[Definition 4.1.2]{AV02}}]\label{def_twisted_curve}
    Let $S$ be a Noetherian scheme. A \emph{twisted curve} over $S$ is a proper flat morphism \[\cC \ \longrightarrow \ S\] from a Deligne--Mumford stack $\cC$, which is \'etale locally a nodal curve over $S$, satisfying that the coarse moduli space $\cC\rightarrow C$ is an isomorphism over the smooth locus of $C\to S$. 
    \end{defn}
\begin{defn}[Twisted pointed curves; {\cite[Definition 4.1.2]{AV02}}]
    A \textit{twisted $n$-pointed curve} over $S$ is the data of a pair \[(\cC;\ \sigma_1,\ldots,\sigma_n) \ \longrightarrow \ S,\text{ where }\]\begin{enumerate}
        \item $\cC\to S$ is a flat and proper morphism from a Deligne--Mumford stack, which is \'etale locally a nodal curve over $S$,
        \item $\sigma_i\to \cC$ is a closed embedding contained in the smooth locus of $\cC\to S$,
        \item $\sigma_i\to S$ is a (possibly trivial) gerbe for every $i$, and
        \item $\sigma_i$ is disjoint from $\sigma_j$ for $i\neq j$.
    \end{enumerate}Moreover, we require that the coarse moduli space $\cC\to C$ is an isomorphism away from the nodal locus of $C\to S$ and away from the gerbes $\sigma_i$.
\end{defn}
\begin{remark}
    A twisted $n$-pointed curve over a smooth scheme $S$ is the data of a twisted curve $\cC'\to S$, together with $n$ sections of its coarse moduli space $s_1,\ldots,s_n$ (i.e. the coarse moduli spaces of the gerbes $\sigma_i$), and a root stack of order $d_i$ at $s_i$. This follows since the only possible stacky structure for $\sigma_i$ is the one of a root stack, from \cite{geraschenko2017bottom}.
\end{remark}

The formation of the coarse moduli space commutes with base change over $S$ by \cite[Lemma 2.3.3]{AV02}. In particular, each fiber of $C\to S$ is the coarse moduli space of a twisted curve over a point, and when $S$ is a point one can check that $C$ is a nodal curve (cf. \cite[Proposition 4.1.1]{AV02}). One can then think of a twisted curve as a nodal curve equipped with additional data at the nodes.

The structure of a twisted curve over a node of $C$ can also be described explicitly, via the versal deformation space of a twisted node, which is well understood (cf. \cite[Proposition 2.2]{olsson2007log}). Over a field of characteristic $0$, a node of a twisted curve is, \'etale locally over its coarse space, isomorphic to
\[
\big[(\spec k[x,y]/xy)/\bmu_n\big]
\]
with the action given by
\[
\zeta * x = \zeta x \qquad \text{and} \qquad \zeta * y = \zeta^k y
\]
for some $k$ coprime to $n$. One is often interested in the case $k=-1$, i.e.\ in twisted curves that are \'etale locally of the form
\[
\big[(\spec k[x,y]/xy)/\bmu_n\big], \qquad \text{with } \qquad \zeta * x = \zeta x \ \text{ and } \ \zeta * y = \zeta^{-1} y.
\]
Such nodes are called \emph{balanced}, and they are the only twisted nodes admitting a smoothing, since the smoothing parameter $xy$ is $\bmu_n$-invariant. As all twisted curves appearing in our moduli problems will be balanced, we will not make this distinction explicit in what follows.

\smallskip

Let $\cX$ be a Deligne--Mumford stack with projective coarse moduli space $X$, and fix an ample class $H$ on $X$.

\begin{defn}
     A \emph{twisted stable map} $\phi\colon \cC\to \cX$ is a representable morphism from a twisted curve $\cC$ such that, denoting by $f\colon C\to X$ the induced morphism on coarse moduli spaces, the class $
    \omega_C\otimes f^*H^{\otimes 3}$ is ample.
\end{defn}
The following is the most important result on twisted stable maps that will be needed.

\begin{theorem}[\textup{cf. \cite[Theorem 1.4.1]{AV02}}]
    Fix two positive integers $d,g\in \bN$. Then there is a proper Deligne--Mumford stack parametrizing twisted stable maps $\phi\colon \cC\to \cX$ such that the induced map on coarse spaces $f\colon C\to X$ has degree $d$ with respect to $H$, and such that $C$ has genus $g$.
\end{theorem}

\subsection{Moduli space of stable minimal models}\label{sec:moduli-of-stable-minimal-models}
In this subsection, we review the moduli theory of stable minimal models and stable Calabi--Yau pairs (cf.~\cite{KX20,Bir22}). We will work only over reduced schemes over $\bC$. The relevant moduli theory is the KSBA-moduli theory, which is well known; we refer the reader to \cite{Kol23}, especially \cite[Chapter~4]{Kol23}, for more details on locally stable and KSBA-stable families.

Local stability \cite[Definition--Theorem~4.7]{Kol23} is invariant under \'etale covers. More precisely, let $(X,D)\to S$ be a well-defined family of pairs as in \cite[Definition--Theorem~4.7.1]{Kol23}, let $\iota\colon U\to X$ be an \'etale cover, and set $D_U \coloneqq \iota^{-1}(D)$. Then $(X,D)\to S$ is locally stable if and only if $(U,D_U)\to S$ is locally stable. In particular, the definition of local stability and well-defined family of pairs extends verbatim to morphisms $(\cX,c\cD)\to S$, where $\cX$ is a Deligne--Mumford stack equipped with a flat morphism $\cX\to S$ whose fibers have dimension $n$, $\cD\hookrightarrow \cX$ is a closed substack, flat over $S$, whose fibers have dimension $n-1$, and $c\in \bQ$. 

\begin{defn}[\textup{\cite[Definition 1.1]{Bir22}}]\label{defn:stable-mimimal-model}
    A \emph{stable minimal model} $(X,B;A)$ over a field $k$ of characteristic $0$ consists of a pair $(X,B)$ and an effective $\bR$-divisor $A$, where
    \begin{itemize}
        \item $(X,B)$ is a projective connected slc pair,
        \item $K_X+B$ is semi-ample defining a contraction $f\colon X \rightarrow Z$,
        \item $A$ is ample over $Z$, and
        \item $(X, B + tA)$ is slc for some $t > 0$.
    \end{itemize}
As $Z$ is uniquely determined by $(X,B;A)$, we also denote the stable minimal model by $(X, B;A)\rightarrow Z$. If $Z$ is a point, i.e. $K_X+B$ is $\bQ$-linearly trivial, then we call $(X,B;A)$ a \emph{stable Calabi--Yau pair}.
\end{defn}
One should think of the coefficient $t>0$ as being sufficiently close to $0$. Two special cases that we will use most frequently later are the following:
\begin{itemize}
    \item $B=\emptyset$, $K_X\sim_{\bQ}0$, and $A$ is ample; and
    \item $B=A$ is ample and $K_X+B\sim_{\bQ}0$.
\end{itemize}

\begin{defn}
    Let $S$ be a reduced $\bC$-scheme. A \emph{family of stable minimal models} over $S$ consists of a projective morphism $X \rightarrow S$ of schemes and $\bQ$-divisors $B, A$ on $X$ such that
    \begin{itemize}
        \item $(X, B+tA)\rightarrow S$ is a locally stable family for any sufficiently small rational $t>0$, and
        \item $(X_s, B_s;A_s)$ is a stable minimal model over $k(s)$ for each $s\in S$.
    \end{itemize}
Here $B_s, A_s$ are the divisorial pullbacks of $B, A$ to the fiber $X_s$.
\end{defn}

\begin{defn}\label{defn:data-set}
A \emph{data set} for stable minimal models, denoted by
$\Xi=(n,d,v,\sigma(t))$, consists of
\begin{itemize}
    \item two positive integers $n,d$;
    \item a rational number $v > 0$; and
    \item a polynomial $\sigma(t)\in \bQ[t]$.
\end{itemize}
A \emph{$\Xi$-stable minimal model} is a stable minimal model
$(X,B;A)\to Z$ satisfying the following conditions:
\begin{itemize}
    \item $\dim X=n$;
    \item the coefficients of $B$ and $A$ belong to $\{0,\frac{1}{d},\frac{2}{d},\ldots,1\}$;
    \item $\vol(A|_F)=v$ for a general fiber $F$ of $X\to Z$ over each irreducible component of $Z$; and
    \item $(K_X+B+tA)^n=\sigma(t)$ for all $0\le t\ll 1$.
\end{itemize}
\end{defn}

The preceding definition gives rise to the following result.

\begin{thm}[\textup{cf. \cite[Theorem 1.14]{Bir22}}]
\label{thm:stable-minimal-model}
For a fixed data set $\Xi$, the moduli functor of
$\Xi$-stable minimal models is represented by a proper
Deligne--Mumford stack of finite type over $\bC$, and its coarse
moduli space is projective. 
\end{thm}

We denote the moduli stack by $\MM^{\Bir}_{\Xi}$ and its coarse moduli space by $\fM^{\Bir}_{\Xi}$. A key input in the proof of the above theorem is the boundedness of stable minimal models (cf. \cite[Theorem 1.12]{Bir22}). In particular, there exists a rational number $\epsilon_0>0$, depending only on $\Xi$, such that for every family of stable minimal models $(X,B;A)\rightarrow S$ parametrized by $\MM^{\Bir}_{\Xi}$, the corresponding family $(X,B+tA)\rightarrow S$ is KSBA-stable (cf. \cite[Definition–Theorem 4.7]{Kol23}) for all $0<t\leq \epsilon_0$. Therefore, there is a natural morphism
    \[
    i\colon \MM^{\Bir}_{\Xi}\ \longrightarrow \ \MM^{\KSBA}_{n, \Phi', \sigma(\epsilon_0)},
    \]
    inducing $j\colon \mathfrak{M}^{\Bir}_{\Xi}\to \mathfrak{M}^{\KSBA}_{n, \Phi', \sigma(\epsilon_0)}$ on coarse moduli spaces. Here $\MM^{\KSBA}_{n, \Phi', \sigma(\epsilon_0)}$ denotes the moduli stack of KSBA-stable marked pairs $(X,D)$ of \cite[Section 8.2]{Kol23}, where $X$ has dimension $n$, the coefficients of $D$ lie in $\Phi'$, which is the smallest subset of the form $\{0,\frac{1}{\delta},\frac{2}{\delta},\ldots,1\}$ containing $\{0,\epsilon_0, 2\epsilon_0, \ldots ,\lfloor\tfrac{1}{\epsilon_0}\rfloor\epsilon_0\}$ and $\{0,\frac{1}{d},\frac{2}{d},\ldots,1\}$; and $\vol(K_X + D)=\sigma(\epsilon_0)$ is the value at $\epsilon_0$ of the polynomial $\sigma(t)$ associated to $\Xi$.
    
    Let $\pi\colon(\sX,\sB;\sA)\to \MM^{\Bir}_\Xi$ be the universal family. Then $\pi\colon(\sX,\sB+\epsilon_0\sA)\rightarrow \MM^{\Bir}_\Xi$ is a KSBA-stable family. Let $\lambda_{\CM}$ be the CM line bundle (cf. \Cref{defn:Chow-line-bundle}) on $\MM^{\Bir}_\Xi$ associated to this KSBA-stable family.

\begin{lemma}
   A positive multiple of $\lambda_{\CM}$ descends to an ample line bundle on $\mathfrak{M}^{\Bir}_\Xi$.
\end{lemma}

\begin{proof}
   The morphism $i$ is proper, representable, and quasi-finite; hence it is finite, and so is $j$. The desired statement then follows from the functoriality of the CM line bundle, together with the fact that a sufficiently high multiple of the CM line bundle on $\MM^{\KSBA}_{n, \Phi', \sigma(\epsilon_0)}$ descends to an ample line bundle on $\mathfrak{M}^{\KSBA}_{n, \Phi', \sigma(\epsilon_0)}$ by \cite[Theorem 1.1]{PX17}.
\end{proof}

A specific case of the moduli spaces $\MM^{\Bir}_{\Xi}$, namely the case when $(X,B;A)$ is a stable Calabi--Yau pair, is studied in \cite{KX20}, where the authors prove the projectivity of each irreducible component of the coarse moduli space; the required boundedness, however, is due to \cite{Bir22}.

\begin{defn}\label{defn:moduli-stable-CY}
    For $\Xi$ as above, we denote by $\PCY_{\Xi}$ the moduli stack of \Cref{thm:stable-minimal-model} in the case when $Z$ is a point, and by $\textbf{PCY}_\Xi$ its coarse moduli space. When $\Xi$ is clear from the context, we simply write $\PCY$ (resp.\ $\textbf{PCY}$) for $\PCY_{\Xi}$ (resp.\ $\textbf{PCY}_\Xi$). We adopt the analogous conventions for $\MM^{\Bir}_{\Xi}$ and $\mathfrak{M}^{\Bir}_{\Xi}$. Similarly, when $n,d$ and $v$ are clear from the context, we will simply write $\MM^{\Bir}_{\sigma(t)}$.
\end{defn}

\subsection{Chow and CM line bundles}
In this subsection we recall the construction of the Chow line bundle and the CM line bundle associated to a family, and we generalize the minimizing property of KSBA-stable limits with respect to the CM-degree established in \cite[Theorem~3.3]{WX14}.

The CM line bundle, introduced in \cite{Tia97} and reformulated in \cite{PT09}, is expected to be positive on the moduli space of K\"{a}hler--Einstein varieties; see also \cite{FR06,XZ20,CP21} for further details. It is defined for any family of polarized pairs $(X,D;\cL)\rightarrow S$ over an arbitrary base scheme $S$ of finite type, subject to the requirement that $D\rightarrow S$ be a family of well-defined divisors. It is possible to generalize the CM line bundle to the setting where $D\rightarrow S$ is not a family of well-defined divisors; however, doing so may result in a loss of geometric meaning. For our purposes, we instead use the Chow line bundle, which is isomorphic to the CM line bundle for KSBA-stable families.

\begin{thm}[\textup{cf.~\cite[Theorem~4]{KM76}}]
    Let $f\colon X\to T$ be a flat and proper morphism with fibers of pure dimension $n$, and let $L$ be a line bundle on $X$ which is ample over $T$. Then there exist $m_0>0$ and unique line bundles $\cL_0,\dots,\cL_{n+1}$ on $T$ such that, for every $m\ge m_0$,
    \begin{equation}\label{eq:KM-expansion}
    \det f_*(L^{\otimes m}) \ \simeq\ \bigotimes_{i=0}^{n+1}\cL_i^{\otimes\binom{m}{i}}.
    \end{equation}
    We call \eqref{eq:KM-expansion} the \emph{Knudsen--Mumford expansion} of $\det f_*(L^{\otimes m})$.
\end{thm}

By cohomology and base change, together with the uniqueness of the $\cL_i$, the formation of $\cL_i$ is functorial. Moreover, under mild assumptions on the singularities of $X$, one has
\begin{equation}\label{eq:first-chern-class}
    c_1(\cL_{n+1}) \ = \ f_*\big(c_1(L)^{n+1}\big)
\end{equation}
by \cite[Appendix~A]{CP21}, specifically \cite[Lemmas~A.2 and~A.4]{CP21}. We now specialize this construction to the case of interest.

\begin{defn}
A \emph{weak KSBA-stable family} $f\colon(X,D)\to T$ consists of a projective variety $X$ satisfying Serre's condition $S_2$ and the condition $G_1$ (Gorenstein in codimension $1$), an effective $\bQ$-divisor $D$ on $X$, and a flat morphism $f\colon X\to T$ to a Gorenstein reduced scheme $T$, such that:
\begin{enumerate}
    \item every fiber of $f$ satisfies Serre's condition $S_1$;
    \item $K_{X/T}+D$ is an $f$-ample $\bQ$-Cartier divisor; and
    \item there exists a dense open subset $U\subseteq T$ such that the restriction $f|_U\colon(X,D)|_U \to U$ is a KSBA-stable family.
\end{enumerate}
\end{defn}

We emphasize that $D\to T$ is not assumed to be flat, nor even a Mumford divisor: $D$ may contain components of some fibers of $f$.

\begin{defn}\label{defn:Chow-line-bundle}
    Let $f\colon (X,D)\to T$ be a weak KSBA-stable family, let $r>0$ be an integer such that $r(K_{X/T}+D)$ is Cartier, and set $L\coloneqq\cO_X\big(r(K_{X/T}+D)\big)$. We define the \emph{Chow line bundle} associated to $f\colon(X,D)\rightarrow T$ to be the $\bQ$-line bundle
    \[
    \lambda_{\Cho,f,L} \ \coloneqq \ \frac{1}{r^{n+1}}\,\cL_{n+1},
    \]
    where 
     $\cL_{n+1}$ is as in the Knudsen--Mumford expansion \eqref{eq:KM-expansion}. When $f$ or $L$ is clear from the context, we suppress it from the notation. If, moreover, $f\colon(X,D)\rightarrow T$ is a locally stable (and hence KSBA-stable) family, then $\lambda_{\Cho}$ agrees with the Chow--Mumford (abbr. CM) line bundle $\lambda_{\CM,f,D}$ associated to $f$.
\end{defn}

\begin{remark}
    The Chow line bundle is only a $\bQ$-line bundle, but it follows from \eqref{eq:KM-expansion} that it is independent of the choice of $r$: if $\lambda_1$ and $\lambda_2$ denote the Chow $\bQ$-line bundles defined using $r_1$ and $r_2$ respectively, then there exists $N>0$ such that $\lambda_1^{\otimes N}$ and $\lambda_2^{\otimes N}$ are line bundles differing by a torsion line bundle; in particular, after replacing $N$ by a sufficiently divisible multiple, they are isomorphic.
\end{remark}

Let $f\colon (X,D)\rightarrow (0\in C)$ be a weak KSBA-stable family over a smooth projective pointed curve such that $D$ is $\bQ$-Cartier. Let $\nu\colon X^\nu\rightarrow X$ be the normalization of $X$ and $f^\nu\colon X^\nu\rightarrow  C$ be the composition. Then $\nu^*K_X= K_{X^\nu}+E$ for some \textit{effective} divisor $E$ by \cite[Equation (5.7.1)]{Kol13}. Hence $\nu\colon (X^\nu,\nu^*D+E)\rightarrow (X,D)$ is a crepant morphism, and using \cite[Lemma A.4]{CP21} one has \begin{equation}\label{eq:normalization-Chow}
    \deg \lambda_{\Cho,f^\nu,\nu^*D+E} \ =\  \deg \lambda_{\Cho,f,D}.
\end{equation}

\begin{prop}
    Let $f\colon (X,D)\rightarrow (0\in C)$ be a weak KSBA-stable family over a smooth projective pointed curve such that $D$ is $\bQ$-Cartier. Let $C'\rightarrow C$ be a degree $d$ finite cover, $f'\colon X'\rightarrow C'$ be the pullbacks of $f$, and $\phi\colon X'\rightarrow X$ the induced morphism. Then \[\deg \lambda_{\Cho,f',D'} \ = \  d\cdot \deg \lambda_{\Cho,f,D},\] where $D'\coloneqq \phi^*D$.
\end{prop}

\begin{proof}
    Notice that $K_{X'/C'}\sim_{\bQ} \phi^*K_{X/C}$, and hence $K_{X'/C'}+D' \sim \phi^*(K_{X/C}+D)$. Then the equality follows immediately from \Cref{eq:first-chern-class}. 
\end{proof}

The following result generalizes \cite[Theorem 3.3]{WX14}; its proof is also a mild generalization of the one in \cite{WX14}.

\begin{thm}\label{thm_generalize_wx}
Let $f^{\sta}\colon (X^{\sta},D^{\sta})\to C$ and $f\colon (X,D)\to C$ be two weak KSBA-stable families of pairs over a pointed smooth projective curve $(0\in C)$. Let $r>0$ be an integer such that
\[
\cL \coloneqq r(K_{X/C}+D),
\qquad
\cL^{\sta} \coloneqq r(K_{X^{\sta}/C}+D^{\sta})
\]
are both Cartier. Assume that
\[
(X,D;\cL)|_{C^\circ}
\simeq
(X^{\sta},D^{\sta};\cL^{\sta})|_{C^\circ},
\]
where $C^\circ \coloneqq C\setminus \{0\}$, and that $(X^{\sta},D^{\sta} + X_0^{\sta})$ is log canonical. Then
\[
\deg \lambda_{\Cho,f^{\sta},\cL^{\sta}}
\ \le\ 
\deg \lambda_{\Cho,f,\cL}.
\]
Moreover, the equality holds if and only if $f$ and $f^{\sta}$ are isomorphic over $C$.
\end{thm}

\begin{proof}
We first assume that the generic fiber $X|_{C^{\circ}}$ is normal. Let $\nu:X^\nu\rightarrow X$ be the normalization of $X$, and $D^\nu\coloneqq \nu^*D$. One can write \[\nu^*(K_X+D) \ \sim_{\bQ} \ K_{X^{\nu}}+D^\nu+M\] for some effective $\bQ$-divisor $M$. Let $Y$ be the normalization of the graph of the birational map $X \dashrightarrow X^{\sta}$. Then $Y\rightarrow X$ factors through a morphism $p^\nu:Y\rightarrow X^\nu$:
\[\begin{tikzcd}[ampersand replacement=\&]
	\&\& Y \& \\
	X \& {X^\nu} \&\& {X^{\sta}}
	\arrow["p"', from=1-3, to=2-1]
	\arrow["{p^\nu}", from=1-3, to=2-2]
	\arrow["{p^{\sta}}", from=1-3, to=2-4]
	\arrow["\nu"', from=2-2, to=2-1]
\end{tikzcd}\] One can write
\[
(p^{\sta})^*(K_{X^{\sta}/C}+D^{\sta}) + E
\sim_{\bQ}
p^*(K_{X/C}+D)
\]
for some (not necessarily effective) $\bQ$-divisor $E$ supported on $Y_0$. Since $(X^{\sta},D^{\sta}+X_0^{\sta})$ is log canonical, we have
\[
K_{Y/C}+\wt D
\sim_{\bQ}
(p^{\sta})^*(K_{X^{\sta}/C}+D^{\sta}) + G,
\]
where $\wt D \coloneqq (p^{\sta})^{-1}_*D^{\sta}$ and $G$ is a $p^{\sta}$-exceptional $\bQ$-divisor, which is effective from \cite[Equation (5.7.1)]{Kol13}. Moreover, there exists an effective $\bQ$-divisor $H$ supported on $X_0$ such that we have the following equality for the following two codimension-one cycles
\[
p_*(K_{Y/C}+\wt D)
=
K_{X/C}+D-H.
\] Consider the function
\begin{equation}\nonumber \begin{split} F(t)\ & \coloneqq \ (n+1)\left(\big((p^{\sta})^*\cL^{\sta}+tE\big)^n.\ (p^{\sta})^*\cL^{\sta}\right) - n\big((p^{\sta})^*\cL^{\sta}+tE\big)^{n+1} \\ \ & =\ \big((p^{\sta})^*\cL^{\sta}+tE\big)^{n+1} - t(n+1)\left(\big((p^{\sta})^*\cL^{\sta}+tE\big)^n.E\right) \end{split} \end{equation} By definition,
\[
F(0)
\ =\ 
(\cL^{\sta})^{n+1}
\ =\ 
r^{n+1}\deg \lambda_{\Cho,f^{\sta},\cL^{\sta}}.
\] On the other hand,
\begin{align}\label{equations_from_WX}
F(r)
&=
(n+1)\bigl((p^*\cL)^n \cdot (p^{\sta})^*\cL^{\sta}\bigr)
-
n(\cL)^{n+1} \\
&=
r(n+1)\bigl((p^*\cL)^n \cdot (K_{Y/C}+\wt D-G)\bigr)
-
n(\cL)^{n+1} \\
&\le
r(n+1)\bigl((p^*\cL)^n \cdot (K_{Y/C}+\wt D)\bigr)
-
n(\cL)^{n+1} \\
&=
r(n+1)\bigl(\cL^n \cdot (K_{X/C}+D-H)\bigr)
-
n(\cL)^{n+1} \\
&\le
r(n+1)\bigl(\cL^n \cdot (K_{X/C}+D)\bigr)
-
n(\cL)^{n+1} \\
&=
r^{n+1}\deg \lambda_{\Cho,f,\cL}.
\end{align} For any $t_0\in (0,r)$, we compute
\[
F'(t_0)
\ =\ 
-t_0n(n+1)
\bigl(((p^{\sta})^*\cL^{\sta}+t_0E)^{n-1}\cdot E^2\bigr).
\] Notice that $(p^{\sta})^*\cL^{\sta}+t_0E$ is a positive linear combination of $p^*\cL$ and $(p^{\sta})^*\cL^{\sta}$. In particular, it is ample for every $t_0\in (0,r)$ since $Y$ is the normalization of the graph. Therefore, the Hodge index theorem implies that $F'(t_0)\ge 0$. Hence
\[
\deg \lambda_{\Cho,f^{\sta},\cL^{\sta}}
\ \le\ 
\deg \lambda_{\Cho,f,\cL}.
\] Finally, note that $F'(t_0)>0$ unless $E$ is a multiple of $Y_0$. However, if $E=e\cdot Y_0$ for some $e\neq0$, then $e>0$ by \Cref{equations_from_WX} and hence \[F(r)\ =\ r^{n+1}\deg \lambda_{\Cho,f,\cL}-er^{n+1}(n+1)\big((K_X+D)|^n_{X_0}\big)\ <\ r^{n+1}\deg \lambda_{\Cho,f,\cL}.\] Therefore, if equality holds, then \(E=0\), and
\[
(p^{\sta})^*\cL^{\sta}
\sim_{\bQ}
p^*\cL.
\]
Consequently,
\[
X^{\sta}
\simeq
\Proj_C R\bigl(Y/C,(p^{\sta})^*\cL^{\sta}\bigr)
\simeq
\Proj_C R(Y/C,p^*\cL)
\simeq
X^\nu,
\]
and \(D^{\sta}=D^\nu+\Xi\) for some $\bQ$-divisor $\Xi$ whose support is contained in $X_0^{\nu}$. Since \(D^{\sta}\to C\) is a family of Mumford divisors and \(\Supp \Xi\subseteq X^\nu_0\), it follows that \(\Xi=0\). Hence \(D^{\sta}=D^\nu\), and \(X\) is regular in codimension~\(1\). Furthermore, because the fibers of \(X\to C\) satisfy \(S_1\), the total space \(X\) satisfies \(S_2\). Since \(X\) is both \(S_2\) and regular in codimension~\(1\), it is normal. Therefore, one has $(X,D)\simeq (X^{\sta},D^{\sta})$.

If the generic fiber of $f$ is not normal, let $\nu\colon X^\nu\to X$ and $\nu^{\sta}\colon X^{\sta,\nu}\to X^{\sta}$ be the normalizations. Let $D^\nu\coloneqq \nu^*D$ and let $\Gamma^\nu$ denote the conductor divisor on $X^\nu$. Then
\[
K_{X^\nu}+D^\nu+\Gamma^\nu
\ \sim_{\bQ}\ 
\nu^*(K_X+D),
\]
and similarly,
\[
K_{X^{\sta,\nu}}+D^{\sta,\nu}+\Gamma^{\sta,\nu}
\sim_{\bQ}
(\nu^{\sta})^*(K_{X^{\sta}}+D^{\sta}).
\] Therefore, by \Cref{eq:normalization-Chow} and the case where the generic fiber is normal, we obtain
\[
\deg \lambda_{\Cho,f^{\sta},\cL^{\sta}}
\le
\deg \lambda_{\Cho,f,\cL}.
\]
Moreover, equality holds if and only if
\[
(X^\nu,D^\nu+\Gamma^\nu)
\ \simeq\ 
(X^{\sta,\nu},D^{\sta,\nu}+\Gamma^{\sta,\nu}).
\] In this case, the conductor divisor $\Gamma^\nu$ is horizontal. It follows that $(X,D)$ and $(X^{\sta},D^{\sta})$ are isomorphic in codimension one. Since both $X$ and $X^{\sta}$ satisfy Serre's condition $S_2$, they are obtained as the relative $\Proj$ of the same graded algebra. Hence $(X,D)\simeq (X^{\sta},D^{\sta})$.
\end{proof}

\begin{remark}\label{rem:local-nature-of-minimizing}
The inequality in \Cref{thm_generalize_wx} is local in nature. More precisely, one may compare the Chow degrees of a weak KSBA-stable family and a KSBA-stable family over an open smooth pointed curve $(0\in C)$, since, after a finite base change, the curve $C$ can be compactified and KSBA-stable reduction can be performed away from the marked point $0\in C$.
\end{remark}

Let $f\colon (X,D)\to C$ be a weak KSBA-stable family over a smooth projective curve $C$ defined over a possibly non-closed field $K$ of characteristic $0$, and let $g\colon C\to \mathfrak{M}^{\KSBA}$ be the induced morphism to the KSBA-moduli space parametrizing the general fibers of $f$. Let $N>0$ be a sufficiently divisible integer such that the $N$-th tensor power of the Chow line bundle (or equivalently, of the CM line bundle) on $\MM^{\KSBA}$ descends to an ample line bundle $L_{\Cho}^{\otimes N}$ on $\mathfrak{M}^{\KSBA}$ (cf. \cite[Theorem 1.1]{PX17}). The locus of $C$ over which $f$ fails to be KSBA-stable is closed; after possibly replacing $K$ by a finite extension $K'$, it is a finite union of $K'$-rational points $p_1,\dots,p_k$.

\begin{cor}\label{cor_degree_of_chow_minimized_by_pull_back_from_KSBA}
There exist positive integers $a_1,\dots,a_k$ such that
\[
\lambda_{\Cho,f,D}^{\otimes N} \ \simeq\ g^*L_{\Cho}^{\otimes N}\otimes \cO_C\big(a_1p_1+\cdots+a_kp_k\big).
\]
\end{cor}

\begin{proof}
It follows from \cite[Theorem 3.1]{BV24} and \cite[Proposition~1.6]{bejleri2025root} that there exists a root stack $\cC$ over $C$ together with a KSBA-stable family $
\phi\colon (\cX,\cD)\to \cC$ extending the generic fiber of $f$. For $N\gg 0$ sufficiently divisible, the line bundle
\(\lambda_{\Cho,\phi,\cD}^{\otimes N}\) descends to a line bundle on $C$. By the universal property of the Chow line bundle, this descended line bundle is naturally isomorphic to $g^*L_{\Cho}^{\otimes N}$. Then the desired property follows immediately from \Cref{thm_generalize_wx} and \Cref{rem:local-nature-of-minimizing}.
\end{proof}

\section{Stable log Calabi--Yau fibration over curves}\label{sec:stable-lCY}

In this section, we introduce the main object of this paper, namely stable log Calabi--Yau fibrations. We then study their basic properties in preparation for the subsequent sections.

\subsection{Families of stable log Calabi--Yau fibrations}

The following is the main definition of this paper, which is the relative version of \Cref{defn:main-defn}.

\begin{defn}\label{def_CY_fibration} Let $B$ be a reduced scheme of finite type over $\bC$. A \textit{stable log Calabi--Yau fibration} over $B$, denoted by \[f\colon (\cX,c\cD)\ \longrightarrow \  \cC\ \longrightarrow \ B,\] consists of 
    \begin{itemize}
        \item a twisted curve $\cC\to B$;
        \item a (representable) flat projective morphism $\pi\colon \cX\rightarrow \cC$ with connected fibers such that for every $b\in B$, $\cX_b$ is a demi-normal Deligne--Mumford stack;
        \item a closed substack $\cD\subseteq \cX$ of pure codimension $1$ whose associated Weil divisor is $\bQ$-Cartier, flat over $B$ and with $S_1$ fibers
    \end{itemize} with coarse space $\phi:(X,cD)\stackrel{\psi}{\rightarrow} C\rightarrow B$ satisfying the following conditions:
    \begin{enumerate}
    \item[(LCY)] for a sufficiently divisible $m$, one has \[\cO_{\cX}\big(m(K_{\cX/B}+c\cD)\big)\ \simeq \  \pi^*\cL\] for some line bundle $\cL$ on $\cC$;
    \item[(Qmap)] there is an open substack $\cU\subseteq \cC$ such that for every geometric point $b\in B$, the support of $\cC_b\smallsetminus \cU_b$ is a (possibly empty) set of smooth points $\{x_1,\ldots,x_n\}\subseteq \cC_b$; and such that there exists an $\epsilon_0>0$ such that \[\big(\cX, (c+\epsilon_0)\cD\big)|_\cU\ \longrightarrow \ \cU\] is KSBA-stable;
        \item[(Sing)] for every geometric point $b\in B$ and for every smooth point $p\in \cC_b$, the semi-log canonical threshold of the fiber of $(\cX_b,c\cD_b)\to \cC_b$ over $p$ is strictly positive;

        \item[(Stab)] the $\bQ$-line bundle
        \[
        K_{X/B} + (c+\epsilon^2)D + \psi^* (\epsilon \Lambda_{\Cho})
        \]
        is ample over $B$ for every $0<\epsilon\ll 1$, where $\Lambda_{\Cho}$ is the descent to $C$ of the Chow line bundle associated to the family $\big(\cX,(c+\epsilon_0)\cD\big)\rightarrow \cC$.
    \end{enumerate}   
\end{defn}

The labels (Sing) and (Stab) stand for ``singularity condition" and ``stability condition", (Qmap) stands for ``quasimap condition", and (LCY) stands for ``log Calabi--Yau".

\begin{remark}\label{rem:stable-LCY}
\begin{enumerate}
    \item In condition (Stab) of \Cref{def_CY_fibration}, the exponent $2$ in $\epsilon^2$ is not essential. It can be replaced by any positive continuous function $g(\epsilon)$ such that $g(\epsilon)=o(\epsilon)$ as $\epsilon\to 0$.
    \item If the base scheme $B$ is not specified, then by default we assume that $B=\Spec \bC$. In this case, we simply refer to $(\cX,c\cD)\to \cC$ as a stable log Calabi--Yau fibration.
    \item Since $c$ is uniquely determined by the morphism $(\cX,\cD)\to \cC$, we will sometimes simply refer to $(\cX,\cD)\to \cC$, or  $\bigl(\cX,(c+\epsilon)\cD\bigr)\to \cC$ for $0<\epsilon\ll 1$, as a stable log Calabi--Yau fibration.
    \item The morphism $(\cX,c\cD)\to B$ is locally stable from \cite[Definition-Theorem 4.7]{Kol23}, as $K_{\cX/B} + c\cD$ is $\bQ$-Cartier, $\cD$ is flat over $B$ and the fibers of $(\cX,c\cD)\to B$ are semi-log canonical.
    \item An $f$-fiber over $p\in \cU$, denoted by $(\cX_p,c\cD_p)$ can be viewed as a polarized Calabi--Yau pair $(\cX_p,c\cD_p;\cD_p)$. By the universal property of $\PCY$ (cf. \Cref{defn:moduli-stable-CY}), the family of polarized Calabi--Yau pairs $(\cX,c\cD;\cD)|_{\cU}\to \cU$ induces a natural morphism $\cU\to \PCY$. We will use this morphism throughout the paper. As explained after \Cref{thm:stable-minimal-model}, we do not distinguish between the moduli space of polarized Calabi--Yau pairs and the corresponding KSBA moduli space. Accordingly, we simply say that the morphism $\cU\to \PCY$ is induced by the family
$\big(\cX,(c+\epsilon)\cD\big)|_{\cU}\to \cU$.
\item Although it is now standard to use K-flat Mumford divisors when defining families of pairs, we do not do so here because $\cX$ is a stack and the notion of K-flatness requires a \textit{projective} family. Instead, we impose the stronger condition that $\cD$ is flat with $S_1$ fibers. The only issue to verify is that this condition is preserved under taking limits, so that the valuative criterion for properness remains valid. More precisely, let $(R,\eta,\xi)$ be a DVR and let $(\cX,c\cD)\to \Spec R$ be a family. Then by \cite[Proposition~5.1]{FI24}, the scheme-theoretic closure of $\cD_\eta$ in $\cX$ is flat over $R$ and has $S_1$ fibers. Indeed, the properties of being flat and having $S_1$ fibers can be checked \'etale locally on $\cX$, where \cite[Proposition~5.1]{FI24} applies.
\end{enumerate}
\end{remark}

\begin{defn}
Let $f\colon (\cX,c\cD)\xrightarrow{\pi}\cC\xrightarrow{g} B$ and $f'\colon (\cX',c\cD')\xrightarrow{\pi'}\cC'\xrightarrow{g'} B'$
be two stable log Calabi--Yau fibrations over reduced schemes of finite type. A morphism
\[\big[f\colon(\cX,c\cD)\xrightarrow{\pi}\cC\xrightarrow{g} B\big]\ \longrightarrow \ \big[f'\colon(\cX',c\cD')\xrightarrow{\pi'}\cC'\xrightarrow{g'} B'\big]\]
consists of morphisms
\[
\phi\colon B\to B',\qquad
\sigma\colon \cC\to \cC',\qquad
\tau\colon \cX\to \cX'
\]
such that $\tau^*\cD'=\cD$ and such that the diagram
\[
\xymatrix{
\cX \ar[d]_{\tau} \ar[r]^{\pi}
& \cC \ar[d]_{\sigma} \ar[r]^{g}
& B \ar[d]^{\phi} \\
\cX' \ar[r]^{\pi'}
& \cC' \ar[r]^{g'}
& B'
}
\]
commutes and both squares are Cartesian.
\end{defn}

\begin{defn}[Admissible quadruple]\label{def_num_data}
    Let $\Phi\coloneqq \big(g,c, v(t),V(t_1,t_2)\big)$ be a quadruple consisting of
    \begin{itemize}
        \item an integer $g\in \mathbb{N}$ and a rational number $c\in [0,1)\cap \bQ$, and
        \item two rational polynomials $v(t)\in \bQ[t]$ and $V(t_1,t_2)\in \bQ[t_1,t_2]$.
    \end{itemize}
    A stable log Calabi--Yau fibration $(\cX,c\cD)\to \cC$ \emph{has numerical invariant $\Phi$} if
    \begin{enumerate}
    \item $\cC$ has genus $g$;
    \item the fiber $(\cX_p,c\cD_p)$ over a general point $p\in \cC$ satisfies \[ \left(K_{\cX_p} + (c+t)\cD_p \right)^{\dim \cX_p}\ =\ v(t) \] and \[\left(K_{\cX/\cC}+(c+t_1)\cD  + t_2\cX_p\right)^{\dim \cX} \ =\ V(t_1,t_2) .\]
        
    \end{enumerate}
    For a given quadruple $\Phi$, if there exists a stable log Calabi--Yau fibration with numerical invariant $\Phi$, then we call $\Phi$ an \emph{admissible quadruple}. 
\end{defn}

\begin{remark}
Given an admissible quadruple $\Phi$, it follows from \cite[Theorem 1.14]{Bir22} and \cite[Theorem 2]{KX20} that there exists a proper Deligne--Mumford stack $\PCY_{\Phi}$ whose closed points parametrize polarized Calabi--Yau pairs $(Z,(c+t)D_Z)$ of dimension $n-1$ and volume $v(t)$. Since the choice of $\Phi$ is fixed throughout this paper, we omit it from the notation and simply write $\PCY$ instead of $\PCY_{\Phi}$.
\end{remark}

\begin{defn}\label{defn:epsilon-coefficient}
    Given an admissible quadruple $\Phi$, there is an $\epsilon_0>0$ such that every pair $(Z,cD_Z + tD_Z)$ of $\PCY$ is KSBA-stable for $0<t\le \epsilon_0$. We call $\epsilon_0$ an \textit{$\epsilon$-coefficient for} $\Phi$.
 \end{defn} 

\begin{defn}[Index of a quadruple]\label{def_index_num_data}
Let $\Phi$ be an admissible quadruple with $\epsilon$-coefficient $\epsilon_0$, and let
\[
\big(\sZ^{\univ},(c+\epsilon_0)\sD^{\univ}\big)\xrightarrow{\pi}\PCY
\coloneqq \PCY_{\Phi}
\]
be the universal family. As before, we do not distinguish the polarized Calabi--Yau pair $(Z,cD_Z;D_Z)$ from the corresponding KSBA-stable pair $(Z,(c+\epsilon_0)D_Z)$. Define $N$ to be the smallest positive integer such that:
\begin{itemize}
    \item $N\bigl(K_{\sZ^{\univ}/\PCY} +(c+\epsilon_0)\sD^{\univ}\bigr)$ is Cartier;
    \item the $N$-th power of the Chow $\bQ$-line bundle associated to $\pi$ descends to a very ample line bundle on the coarse moduli space $\bf{PCY}$ of $\PCY$; and
    \item for every pair $(Z,(c+\epsilon_0)D)\in\PCY$, the order of $\Aut(Z,(c+\epsilon_0)D)$ divides $N$.
\end{itemize}
We call $N$ the \emph{index} of $\Phi$.
\end{defn}

We now record several examples of stable log Calabi--Yau fibrations that the reader may keep in mind.

\begin{exa}[Elliptic surfaces]
Let $f\colon X\to C$ be a smooth minimal elliptic surface of non-negative Kodaira dimension with integral fibers, and let $D$ be a multisection of $f$. Then $f\colon (X,D)\to C$ is a stable log Calabi--Yau fibration. A particularly important case is when $X$ is an elliptic K3 surface.
\end{exa}

\begin{exa}[Abelian fibrations]
The previous example admits a higher-dimensional analogue. Let $f\colon X\to C$ be a flat and proper morphism with irreducible fibers whose generic fiber is a principally polarized abelian variety, and let $D\subseteq X$ be a divisor whose restriction to each fiber is a theta divisor. If $(X,\epsilon D)\to C$ has maximal variation for all sufficiently small $\epsilon>0$, then $f\colon (X,D)\to C$ is a stable log Calabi--Yau fibration.
\end{exa}

\begin{exa}[Minimal models of Kodaira dimension one]
Let $X$ be a good minimal model of Kodaira dimension $1$, and let $f\colon X\to C$ be its ample model morphism. If $D\subseteq X$ is any $f$-ample $\bQ$-Cartier divisor, then $f\colon (X,D)\to C$ is a stable log Calabi--Yau fibration.
\end{exa}

\begin{exa}[Mori fiber spaces]
Let $\bP\cE\to C$ be a projective bundle over $C$ with fibers isomorphic to $\bP^n$, and let $D\subseteq \bP\cE$ be a divisor, flat over $C$, whose restriction to each fiber has degree $d>n+1$ and is of maximal variation. If the pair $\left(\bP\cE,\frac{n+1}{d}D\right)$ has mild singularities, then $\left(\bP\cE,\frac{n+1}{d}D\right)\to C$ is a stable log Calabi--Yau fibration.
\end{exa}

For technical reasons, and because it encompasses most applications and all of the examples discussed above, we will primarily focus on stable log Calabi--Yau fibrations that admit a degeneration to a mildly singular stable log Calabi--Yau fibration. This motivates the following definition.
\begin{defn}[Non-degenerate (type)]\label{defn:non-degenerate-type}
A stable log Calabi--Yau fibration $(\cX,c\cD)\to \cC$ is called
\emph{non-degenerate} if $\cX$ is normal, which implies that $\cC$ is a smooth scheme, and
$\cD$ has no vertical components. We say that $(\cX,c\cD)\to \cC$ is of
\emph{non-degenerate type} if it can be deformed to a non-degenerate stable log Calabi--Yau fibration. More precisely, there exist a DVR $(R,\eta,\xi)$ and a
stable log Calabi--Yau fibration
\[
(\cX_R,c\cD_R) \ \longrightarrow \  \cC_R \ \longrightarrow \  \Spec R
\]
such that
\begin{enumerate}
    \item the special fiber
    \((\cX_\xi,c\cD_\xi)\to \cC_\xi\) is isomorphic to
    \((\cX,c\cD)\to \cC\),
    and
    \item the geometric generic fiber
    \((\cX_{\bar\eta},c\cD_{\bar\eta})\to \cC_{\bar\eta}
    \) is non-degenerate.
\end{enumerate}
\end{defn}

\subsection{Properties of stable log Calabi--Yau fibrations}

In this subsection, we establish several preliminary results on stable log Calabi--Yau fibrations that will be used throughout the paper.

\begin{lem}\label{lem:uniquely-determined}
    Given a fibration $(X,cD)\to C$, there exist only finitely many stable log Calabi--Yau fibrations $(\cX,c\cD)\to \cC$ with coarse space $(X,cD)\to C$.
\end{lem}
\begin{proof}
    We may assume that such a stable log Calabi--Yau fibration $(\cX,c\cD)\to \cC$ exists. Since the stackiness of $(\cX,c\cD)\to \cC$ is concentrated only over the twisted nodes of $\cC$, it suffices to show that, over a node of $C$, the stable log Calabi--Yau fibration is determined by its coarse space up to finitely many choices. For the normalization $R$ of every irreducible component of $C$, there is an induced rational map $R\dashrightarrow \PCY$ to the moduli of stable Calabi--Yau pairs parametrizing fibers, which may fail to be defined over a node $p\in C$. By \cite[Theorem 3.1]{BV24}, there is a unique root stack $\cR\to R$ and a unique representable extension over $p$ of the rational map $R\dashrightarrow  \PCY$. It follows that there exists a unique twisted curve $\cC$ with coarse moduli space $C$ and, by the finiteness of the automorphism group of stable Calabi--Yau pairs, finitely many KSBA-stable families
    \[
    (\cX,(c+\epsilon)\cD)\to \cC
    \]
    defined over a neighborhood of $p\in \cC$, whose coarse space is $(X,(c+\epsilon)D)\to C$.
\end{proof}

\begin{prop}\label{prop_descent_KX_D}
    Let $(\cX,c\cD)\stackrel{\pi}{\rightarrow} \cC\to B$ be a stable log Calabi--Yau fibration over a reduced
    scheme $B$ of finite type, and let $(X,cD)\to C\to B$ be its coarse space with $\rho\colon \cX\to X$ the coarse space morphism. Then
    \begin{enumerate}
        \item[\textup{(1)}] for every sufficiently divisible $m>0$, the divisors $mK_{\cX/B}$
        and $m\cD$ descend to $mK_{X/B}$ and $mD$ respectively, that is
        \[
        \cO_{\cX}(mK_{\cX/B})\simeq \rho^*\cO_X(mK_{X/B})
        \qquad\text{and}\qquad
        \cO_{\cX}(m\cD)\simeq \rho^*\cO_X(mD);
        \]
        
        \item[\textup{(2)}]  $K_{X/B} + cD$ is $\bQ$-linearly equivalent to the pull-back of a line bundle on $C$.
    \end{enumerate}
\end{prop}

\begin{remark}\label{rem:weaken-the-conditions}
    In fact, as we will see in the proof below, the assumptions can be considerably weakened: the fibration $\pi$ need not satisfy condition (Stab) in \Cref{def_CY_fibration}, and, in (1), it need not satisfy conditions (LCY) and (Sing) either.
\end{remark}

\begin{proof}[Proof of \Cref{prop_descent_KX_D}]To prove (1), consider $\cU\subseteq \cX$, the largest open substack of $\cX$ with coarse moduli space $U\subseteq X$, such that the fibers of $\cU\to B$ and $U\to B$ are Gorenstein. Denoting by $\rho_\cU\colon \cU\to U$ the coarse moduli space map, it suffices to prove that
\[
\rho_{\cU}^*\omega_{U/B}\ \simeq \ \omega_{\cU/B},
\]
since, for every $b\in B$, the fibers $\cU_b$ and $U_b$ contain all codimension-one points of $\cX_b$ and $X_b$, respectively. In particular, we may replace $\cX$ by the locus
where the fibers of $\cX\to B$ and $X\to B$ are Gorenstein; equivalently, we
may replace the properness assumption on $\cX\to B$ with the assumption that
both $\cX\to B$ and $X\to B$ have Gorenstein fibers.
We first treat the statement on $K_{\cX/B}$.
Pick $m>0$ divisible enough that $\cO_\cX(mK_{\cX/B})$ is a line bundle
descending to $X$. By Grothendieck duality
\cite[Theorem 1.17]{nironi2008grothendieck}, there is a natural map
\[
\rho_*\omega_{\cX/B}\to \omega_{X/B}.
\]
It is straightforward to check that this map commutes with base change, since we have
replaced $\cX$ and $X$ by their Gorenstein loci over $B$. It thus suffices to
show that $\rho_*\omega_{\cX/B}\to \omega_{X/B}$ is an isomorphism in the case
$B=\spec k$.

For $B=\spec k$, the result follows from a local computation: by
\cite[Proposition 2.2]{olsson2007log}, a twisted node $n\in \cX$ is locally
$[\spec(k[x,y]/xy)/\bmu_n]$, with $\bmu_n$ acting by $\zeta * x = \zeta x$ and
$\zeta * y = \zeta^{r} y$ for some $r$ coprime to $n$. Hence, if we take the
reduced fiber over $n$ and let $\xi$ be the generic point of one of its
irreducible components, then $\cX$ is étale-locally of the form
$[\spec(K(\xi)[x,y]/xy)/\bmu_n]$, and the coarse moduli space map is induced
by
\[
\spec(K(\xi)[x,y]/xy)\to \spec(K(\xi)[x^n,y^n]/x^ny^n).
\]
A generator of the dualizing sheaf on $\spec(K(\xi)[x,y]/xy)$ is
$\frac{\operatorname{d}x\wedge\operatorname{d}y}{xy}$, on which $\bmu_n$ acts
trivially, while a generator of the dualizing sheaf of $\spec(K(\xi)[x^n,y^n]/x^ny^n)$ is
$\frac{\operatorname{d}x^n\wedge\operatorname{d}y^n}{x^ny^n}$. Since
$\frac{\operatorname{d}x\wedge\operatorname{d}y}{xy}$ is $\bmu_n$-invariant,
it descends, and one checks that it descends to a nonzero multiple of the
latter generator.

The statement on $\cO_{\cX}(m\cD)$ is simpler: again it suffices to check
that, for $m$ divisible enough, $\cO_{\cX}(m\cD)$ descends to $\cO_X(mD)$ in
codimension one of each fiber. This holds away from the twisted nodes, since
the coarse moduli space map is an isomorphism on that locus. It also holds at the
generic points of the fibers over the twisted nodes, since, by condition
(Qmap) of \Cref{def_CY_fibration}, $\cD$ does not contain those points.

For (2), descending the relation $m(K_{\cX/B}+c\cD)\sim\pi^*\cL$ of condition
(LCY) along $\rho$ via (1) shows that a sufficiently divisible multiple of
$K_{X/B}+cD$ is pulled back from a line bundle on $C$, so $(X,cD)\to C$ is a
log Calabi--Yau fibration.
\end{proof}

\begin{lemma}\label{lemma_ampleness_of_omega_C_plus_lambda}
    Let $(\cX,c\cD)\to \cC$ be a stable log Calabi--Yau fibration, and let $\cC\to C$ be its
    coarse space. Then $(\omega_\cC\otimes \lambda_{\Cho}^{\otimes n})^{\otimes M}$ descends
    to an ample line bundle on $C$ for any sufficiently divisible $M,n>0$.
\end{lemma}

\begin{proof}
    By \Cref{cor_degree_of_chow_minimized_by_pull_back_from_KSBA} and the ampleness of the Chow line bundle (or equivalently, the CM line bundle) $L_{\Cho}$ on the KSBA-moduli space of stable pairs (cf. \cite[Theorem 1.1]{PX17}), the line bundle
    $\lambda_{\Cho}$ has non-negative degree on every irreducible component of $\cC$, and its
    degree on a component $\cR$ vanishes if and only if every fiber of $\big(\cX,(c+\epsilon)\cD\big)|_{\cR}\rightarrow \cR$ is KSBA-stable. Since $\omega_\cC$ is ample on every component of $\cC$ whose coarse space is
    neither a rational tail nor a rational bridge, it suffices to show that $\lambda_{\Cho}$
    has strictly positive degree on each component $\cR$ whose coarse space is a rational tail
    or a rational bridge.

    Fix such a component $\cR$. Then $\cR$ is a root stack of $\bP^1$ at at most two points, so
    there is a (possibly ramified) cover $\bP^1\to\cR$; let $(Y,(c+\epsilon_0)D_Y)$ be the
    pullback family. Suppose for contradiction that $\deg_\cR\lambda_{\Cho}=0$. Then every
    fiber of $(Y,(c+\epsilon_0)D_Y)\to\bP^1$ is KSBA-stable, so the induced
    classifying map $\bP^1\to\PCY$ has image of dimension zero; as $\bP^1$ is simply
    connected, this map is constant and the family $(Y,(c+\epsilon_0)D_Y)$ is isotrivial,
    hence a product after a finite \'etale cover. This contradicts condition~(Stab) of
    \Cref{def_CY_fibration}, which forces $\lambda_{\Cho}$ to have positive degree on any
    such rational component. Therefore $\deg_\cR\lambda_{\Cho}>0$, and the claim follows.
\end{proof}

Let $\cC\rightarrow B$ be a flat family of proper twisted curves over an integral base scheme $B$ of finite type, let $\pi\colon\cX\to \cC$ be a projective surjective morphism with connected pure-dimensional fibers, and $\cD\subseteq \cX$ a $\bQ$-Cartier divisor flat over $B$. Let $c\in \bQ$ and assume that $(\cX,c\cD)\rightarrow B$ is a locally stable family. We define
\[
U\coloneqq\big\{
b\in B \ | \
\operatorname{slct}(\cX_b,c\cD_b;\cX_p)>0
\text{ for every }
p\in \cC_b^{\sm}
\big\}.
\]

\begin{lemma}\label{lemma_slct_positive_is_open}
The subset $U\subseteq B$ is constructible. Moreover, if $\pi$ is flat and the fibers of the morphism $(\cX,c\cD)\to \cC$ over the nodal locus of $\cC\to B$ are semi-log canonical, then $U$ is open, and in this case, there is a $\delta>0$ such that for every $b\in U$, $\slct(\cX_b,c\cD_b;\cX_p)>\delta$ for every $p\in \cC_b^{\sm}$.
\end{lemma}

\begin{proof}
We first prove that $U$ is constructible. This part of the proof only involves
the morphism $\pi$ over the smooth locus of $\cC$, where $\pi$ is schematic.
Hence, we replace all Deligne--Mumford stacks by their coarse moduli spaces and
use roman letters for the resulting objects.

Working one connected component at a time, we may assume that $B$ is connected
and reduced. By shrinking $B$ and arguing by Noetherian induction, we may
further assume that $B$ is smooth with generic point $\eta$, and that there
exists a simultaneous normalization and a simultaneous log resolution
\[
(Y,D_Y)\xrightarrow{\psi}(X^\nu,D^\nu+\Delta)
\xrightarrow{\nu}(X,cD),
\]
where $\Delta$ is the conductor divisor and $D_Y$ is the divisor for which
$\psi$ is crepant. Write $D_Y=P-N$, where $P,N\ge 0$ have no common components. Let
$\Gamma_1,\dots,\Gamma_r$ be the strata of $\lfloor P\rfloor$, and let $Z_i\subseteq C$ be the image of $\Gamma_i$ under the morphism $Y\to C$.
Since $Y\to C$ is proper, each $Z_i$ is closed. Suppose that, for some $i$, the generic fiber $(Z_i)_\eta$ is contained in
$C_\eta^{\sm}$ and has codimension one in $C_\eta$. After shrinking $B$, the
same holds fiberwise. Then there exists a log canonical center dominating a
divisor contained in the smooth locus of $C\to B$, and hence
$\slct(X_b,cD_b;(X_b)_p)=0$ for every point $p$ on this divisor. Therefore, the condition defining $U$ fails on $B$. We may therefore assume that, for every $i$, the generic fiber $(Z_i)_\eta$
either dominates an irreducible component of $C_\eta$ or is contained in the
singular locus of $C_\eta$. After shrinking $B$, the same property holds
fiberwise. Consequently, no log canonical center of
$(Y_b,D_{Y_b})$ dominates a divisor contained in $C_b^{\sm}$.
Equivalently, $\slct(X_b,cD_b;(X_b)_p)>0$ for the generic point $p$ of every
prime divisor contained in $C_b^{\sm}$. This proves that $U$ is constructible.

We now prove the second assertion by proving that $U$ is stable under generalization. Let $(R,\eta,\xi)$ be a DVR together with a morphism $\mathfrak t\colon \Spec R\to B$ such that
$\mathfrak t(\xi)\in U$, let \[(\cX_R,c\cD_R)\xrightarrow{\pi_R}\cC_R\to \Spec R\] be the induced family. After a finite base change, let $p\in \cC_\eta$ be the generic point of a prime divisor contained in $\cC_\eta^{\sm}$, and let $\overline p\subseteq \cC_R$ denote its closure. Since we will not use the properness of $\cC_R\to\Spec R$, we may replace $\cC_R$ by an \'etale neighborhood of the closed point of $\overline p$, which we continue to denote by $\cC_R$, and assume that it is a scheme. In other words, we may assume that $\cC_R\to\Spec R$ is a family of non-proper nodal curves.

If $\overline p$ meets the nodes of $\cC_\xi$, we may perform a sequence of blowups $\cC'_R\to\cC_R$ supported over the nodal locus so that the strict transform of $\overline p$ is contained in the smooth locus of
$\cC'_R\to\Spec R$. Set $\cX'_R:=\cX_R\times_{\cC_R}\cC'_R$. Then, from the assumption that the fibers of $(\cX,c\cD)\to \cC$ are slc over the nodal locus of $\cC$, the induced morphism $\pi'_R\colon(\cX'_R,c\cD'_R)\to\cC'_R$ remains locally stable over the exceptional divisors of $\cC_R'\to \cC_R$. Since $\mathfrak t(\xi)\in U$, every prime divisor contained in the smooth locus of $\cC'_\xi$ has positive semi-log canonical threshold. Hence, for sufficiently small $\epsilon>0$, the pair \[\big(\cX'_R,c\cD'_R+\epsilon(\pi'_R)^*\overline p+\cX'_\xi\big)\] is semi-log canonical in a neighborhood of the special fiber. Restricting to the generic fiber shows that
$(\cX_\eta,c\cD_\eta+\epsilon\cX_p)$ is semi-log canonical, and therefore $\slct(\cX_\eta,c\cD_\eta;\cX_p)>0$. Since $p$ was arbitrary, we conclude that $\eta\in U$. Thus $U$ is stable under generalization. Since $U$ is constructible and stable under generalization, it follows that $U$ is open.

The statement regarding $\delta$ is proved as follows. As above, since the
statement only involves the morphism $\pi$ over the smooth locus of $\cC$,
and $\cC$ is a scheme, we may replace all Deligne--Mumford stacks by their
coarse moduli spaces, and use roman letters for the resulting objects. We
proceed by Noetherian induction, so we may assume that $B$ is irreducible.
Since the sublocus $Z\subseteq C$ over which the fiber of $(X,cD)\to C$ is
not locally stable is closed and quasi-finite over $B$, after a
stratification and a finite étale base change we may assume that
$Z\to B$ is a union of sections, and that there exist a simultaneous
normalization and a simultaneous log resolution
\[
(Y,D_Y)\ \stackrel{\psi}{\longrightarrow}\ (X^\nu,D_{X^{\nu}})
\ \stackrel{\nu}{\longrightarrow}\ (X,cD+\pi^{-1}(Z))\ \longrightarrow \ (C,Z).
\]
It then follows that the minimum of the semi-log canonical thresholds of the
fibers $(X_b,cD_b)$ is locally constant as a function of $b\in B$, and the
desired statement follows.
\end{proof}

\begin{prop}\label{lem:stable-lcy-implies-slc}
Let $(\cX,c\cD)\to \cC\to B$ be a stable log Calabi--Yau fibration over a
reduced scheme $B$ of finite type, and let $(X,cD)\to C\to B$ be its coarse
space. Then there exists $\epsilon_1>0$ such that the morphism
$(X,(c+\epsilon)D)\to B$ is locally stable for every rational number
$0\le\epsilon\le \epsilon_1$.
\end{prop}
\begin{proof}
By condition (LCY) and the assumption that $\cD$ is $\bQ$-Cartier, for every
rational number $\epsilon$ and every sufficiently divisible integer $m$
(depending on $\epsilon$), the sheaf
$\cO_{\cX}\bigl(m(K_{\cX/B}+(c+\epsilon)\cD)\bigr)$ is a line bundle. By
\Cref{prop_descent_KX_D}, this line bundle descends to
\[
\cO_X\bigl(m(K_{X/B}+(c+\epsilon)D)\bigr).
\]
By conditions (Sing) and (Qmap), together with the quasi-compactness of $B$,
the last part of \Cref{lemma_slct_positive_is_open} shows that there exists
$\epsilon_1>0$ such that $(\cX_b,c\cD_b)$ has semi-log canonical threshold at least $\epsilon_1$ with respect to every fiber of $(\cX_b,c\cD_b)\to \cC_b$. In particular, for every rational number $0\le\epsilon\le\epsilon_1$, the pair $(\cX_b,(c+\epsilon)\cD_b)$ is
semi-log canonical for every $b\in B$. The desired conclusion now follows
from the definition of locally stable families
\cite[Definition--Theorem 4.7]{Kol23}.
\end{proof}

\subsection{KSBA-stable models}\label{sec:KSBA-model}

Let $f\colon (\cX,c\cD)\to \cC$ be a stable log Calabi--Yau fibration with numerical invariant
\(\Phi=\big(g,c,v(t),V(t_1,t_2)\big)\), and let \(p_1,\ldots,p_m\in \cC\) be the (smooth) points over which $f:\big(\cX,(c+\epsilon)\cD\big)\rightarrow \cC$ is not KSBA-stable for $0<\epsilon\ll1$. Then there is a morphism
\[
\cC\setminus \{p_1,\ldots,p_m\}
\ \longrightarrow\ 
\PCY
\]
to the moduli stack of stable Calabi--Yau pairs parametrizing the fibers. Since \(\PCY\) is proper, after replacing \(\cC\) with a suitable root stack $(\cC^{\sta};\sigma_1,...,\sigma_m)\to (\cC;p_1,...,p_m)$ along the points \(p_i\), this morphism extends to a KSBA-stable family
\[
\big(\cX^{\sta},(c+\epsilon)\cD^{\sta}\big) \ \longrightarrow \ \cC^{\sta} .
\]

\begin{defn}[KSBA model]\label{defn:KSBA-model}
The family \((\cX^{\sta},(c+\epsilon)\cD^{\sta})\to\cC^{\sta}\) is called the \emph{stacky KSBA model} of \(f\), and its coarse space \((X^{\sta},(c+\epsilon)D^{\sta})\to C^{\sta}\) is called the \emph{KSBA model} of \(f\). 
\end{defn}

\begin{lemma}\label{lemma_tsm_family_is_stable}
With the notation as above, the morphism
\(\phi\colon (\cC^{\sta};\sigma_1,\ldots,\sigma_m)\to \PCY\)
induced by \(\pi^{\sta}\) is a twisted stable $m$-pointed map.
\end{lemma}

\begin{proof}
    It suffices to show that for every irreducible component $\cR^{\sta}\subseteq \cC^{\sta}$ such that $\phi|_{\cR^{\sta}}$ has $0$-dimensional image, the line bundle $\omega_{\cC^{\sta}}(\sigma_1+\cdots+\sigma_m)|_{\cR^{\sta}}$ is ample. Let $\cR$ be the corresponding irreducible component of $\cC$.

    Suppose $\omega_{\cC^{\sta}}|_{\cR^{\sta}}$ is nef but no marking $\sigma_i$ lies on $\cR^{\sta}$. Then $\cC$ and $\cC^{\sta}$ are isomorphic in a neighborhood of $\cR^{\sta}=\cR$. In particular, $\pi$ and $\pi^{\sta}$ are isomorphic over $\cR$, so the map \[(\cX,(c+\epsilon)\cD)|_{\cR}\to \cR\] has nef but non-ample Chow line bundle. By \Cref{cor_degree_of_chow_minimized_by_pull_back_from_KSBA}, the family over $\cR$ is therefore isotrivial, i.e., all closed fibers are isomorphic. By condition (Stab), the coarse space of
    $(\cX|_{\cR},(c+\epsilon)\cD|_{\cR} + \cF)$ is KSBA-stable, where $\cF$ denotes the fibers over the nodes of $\cC$ lying on $\cR$. This implies that $\omega_\cC|_\cR$ is ample.

    Suppose now that $\omega_{\cC^{\sta}}|_{\cR^{\sta}}$ is not nef. Then either the coarse moduli space of $\cR^{\sta}$ is a rational tail, or $\cC=\cR=\bP^1$.  Assume first that the coarse moduli space of $\cR^{\sta}$ is a rational tail. If all fibers over $\cR$ are KSBA-stable with coefficient $c+\epsilon$, i.e., there is no marking $\sigma_i$ on $\cR^{\sta}$, one can proceed as above to obtain a contradiction. Hence there must exist KSBA-unstable fibers over $\cR$, and we need to show that there are at least two of them; we argue by contradiction.
    
    Consider the cover $\bP^1\to \cR$ whose preimage of the stacky point of $\cR$ is a single point. Let $(Y,cD_Y + F)$ be the pulled-back family over $\bP^1$, and let $p_0\in \bP^1$ be the point over which the fiber is $F$.
    Consider the normalization of $Y$, with connected components $\{Y_i\}_{i=1}^k$ and conductor divisors $\Delta_i\subseteq Y_i$. Let $Z_i$ be the Stein factorization of the morphism $Y_i\to \bP^1$. We obtain the following commutative diagram
    \[
    \xymatrix{\coprod_{i=1}^k(Y_i,cD_i + F_i + \Delta_i) \ar[rr]\ar[d]_{\coprod \pi_i}& & (Y,cD+F)\ar[d] \\\coprod_{i=1}^k Z_i\ar[rr]^-{\coprod \rho_i} && \bP^1.}
    \]
    Now, take a smooth point $q\in \bP^1$.
    Since the map $\bP^1\to \textbf{PCY}$ is constant, if the fiber over $q\in \bP^1$ is KSBA-stable, then in an \'etale neighborhood of $q$ the family $(Y,cD)\to \bP^1$ is a product. In particular, $q$ is not a ramification point of $\rho_i$ for any $i$. 
   Since the fiber over $p_0$ is KSBA-stable when replacing the coefficient $c$ with $c+\epsilon$, the maps $\rho_i$ are unramified over $p_0$. On the other hand, they cannot be ramified elsewhere either, since otherwise we would have at least two markings on $\bP^1$ as any map from a smooth curve to $\bP^1$ cannot be ramified at a single point. Hence each $\rho_i$ is \'etale, as $\bP^1$ is simply connected, each $\rho_i$ is an isomorphism.

    The fibrations $\pi_i$ are still log Calabi--Yau, and since the source is now normal, we may apply the canonical bundle formula. Since by assumption the map to $\textbf{PCY}$ is constant, the moduli part is trivial on the generic fiber, hence trivial. By condition (Sing) in \Cref{def_CY_fibration}, the boundary part has coefficients strictly less than $1$. Thus, in order for $K_{Z_i}+B_i+\mathbf{M}_i$ to be nef, $\Supp B_i$ must contain at least two points. Hence there are at least two points marked with $\sigma_i$, as desired.

    Assume finally that $\cR=\cC=\bP^1$, and consider the normalization as before:
     \[
    \xymatrix{\coprod_{i=1}^k(X_i,cD_i + \Delta_i) \ar[rr]\ar[d]_{\coprod \pi_i}& & (X,cD)\ar[d] \\\coprod_{i=1}^k Z_i\ar[rr]^-{\coprod \rho_i} && \bP^1.}
    \]
    Note that $F_i$ and $F$ are omitted here, since $\cR$ has no nodes. Now each map $\rho_i$ is either an isomorphism or ramified at no more than two points. Proceeding as above, we see that the boundary part of $\pi_i$ has coefficients strictly less than $1$, and that the moduli part is trivial. Hence $\pi_i$ has at least three unstable fibers, and $Z_i\to \bP^1$ is either an isomorphism or ramified at two points. Either way, there are three marked points on $\bP^1$.
\end{proof}

The following is analogous to \cite[Theorem 8.13]{AV_complete_moduli} and \cite[Theorem 4.2]{ABtwisted}, and the same proof as in \textit{loc. cit.} applies.

\begin{lemma}\label{lemma_singularities_tsp}
 Let $(\cC,\sigma_1,\ldots,\sigma_m)\rightarrow B$ be a twisted curve over a connected smooth scheme \(B\), and $(\cC,\sigma_1,\ldots,\sigma_m)\to \PCY$ be a twisted stable $m$-pointed map. Let
\[
\pi\colon (\cX,\cD)\longrightarrow (\cC,\sigma_1,\ldots,\sigma_m)\longrightarrow B
\]
be the pullback of the universal family. Let \(\cX_i\) be the fiber of \(\pi\) over \(\sigma_i\), and let \(X_i\) denote the coarse space of \(\cX_i\). Then the induced morphism of the coarse spaces \(
p\colon \bigl(X,D+\sum_i X_i\bigr)\to B\) is a locally stable family.
\end{lemma}

\begin{remark}
    Let $\Sigma_i\subseteq C$ be the coarse space of $\sigma_i$. With the notation above, $X_i$ is \emph{not} the fiber over $\Sigma_i$. Indeed, $X_i$ is reduced since it is the coarse moduli space of the reduced stack $\cX_i$. In other words, $X_i$ is the reduced structure of the fiber over $\Sigma_i$. This reflects the fact that the fibers of a KSBA model need not be reduced, so a KSBA model is not necessarily locally stable. Nonetheless, we see in \Cref{lemma_singularities_tsp} that every log canonical center of $(X,(c+\epsilon)D)$ intersects the generic fiber of $X\rightarrow C$. Similarly, if we denote by $\cC\to C$ the coarse moduli space of $\cC$, since $\pi_*\cO_\cX = \cO_\cC$ and since $\cX\to X$ and $\cC\to C$ are coarse moduli spaces, we have that $p_*\cO_X=\cO_C$, where $p\colon X\to C$ is the map induced by $\pi$.
\end{remark}

\section{The valuative criterion of properness}\label{Valuative criterion}

In this section, we establish the valuative criteria for universal closedness and separatedness of the moduli stack of stable log Calabi--Yau fibrations, whose construction will be carried out in \Cref{sec:Representability by a Deligne--Mumford stack}. Combined with the boundedness of stable log Calabi--Yau fibrations of non-degenerate type proved in \Cref{sec:Boundedness of stable log Calabi--Yau fibrations}, these results imply the properness of the stack.

We choose to address the valuative criteria first for two main reasons. First, the existence and uniqueness of stable limits illustrate both the significance and the elegance of the stability condition for log Calabi--Yau fibrations. Second, the arguments yield a useful property of stable limits (cf. \Cref{cor:number-components-bounded}), which will play a role in the proof of boundedness (cf. \Cref{thm_bounding_fam}).

\smallskip

We begin by proving universal closedness (cf. \Cref{thm:valuative-criterion}) and then show that any automorphism of the generic fiber extends to the family over the DVR constructed there (cf. \Cref{thm:extension-of-automorphism}). This extension result will be used in the final step of the proof of separatedness (cf. \Cref{thm:separatedness}).

\subsection{Existence of stable limits}\label{sec:Existence-of-stable-limits}

\begin{thm}[Universal closedness]\label{thm:valuative-criterion}
    Let $(R,\eta,\xi)$ be a DVR, let $C_\eta\to \eta$ be a smooth projective curve, and let \[f_\eta\colon (X_\eta,cD_\eta )\ \longrightarrow \ C_\eta\] be a non-degenerate stable log Calabi--Yau fibration over $\eta$. Then, up to replacing $R$ with a finite extension, there is a family of twisted curves $\cC\to \Spec R$ and a stable log Calabi--Yau fibration
    \[
    f\colon (\cX,c\cD) \ \longrightarrow \  \cC
    \]
    extending $f_\eta$. 
\end{thm}

\smallskip

The proof is organized as follows. We begin by constructing a family
\[
  (X^{\sta},cD^{\sta}) \ \longrightarrow \ C^{\sta}\ \longrightarrow \ \spec R
\]
which need not agree with the original family $(X_\eta,cD_\eta)$ over the
generic fiber. It is only birational to it, but has the advantage of
extending over all of $\spec R$. The most technical part of the proof consists of performing a sequence of birational modifications to $(X^{\sta},cD^{\sta})$ to produce a new family $(W^{\sta},T_{W^{\sta}} + cD_{W^{\sta}})$, together
with a divisor $G^{\sta}$, such that the canonical model $(Y,cD_Y)$ of
\[
  (W^{\sta},T_{W^{\sta}} + (c+\epsilon)D_{W^{\sta}} + G^{\sta})
\]
over $C^{\sta}$ exists and, over $C_\eta$, recovers the original family
$(X_\eta,cD_\eta)$. 
Once this is achieved, we have extended the original family to a family
\[
  (Y,cD_Y)\longrightarrow C\longrightarrow \spec(R).
\]
Since this family will be a Calabi--Yau fibration, we may consider its boundary
part $B$ and moduli part $\mathbf{M}$ in $C^{\sta}$; following \cite{ISZ25},
we then obtain a weak canonical model of $(Y,cD_Y+Y_0)$ by first taking a
minimal model of the generalized pair $(C^{\sta},B + \mathbf{M})$, as in \Cref{eq:diagram-of-MMP}. Finally, we show that this minimal model is a stable log Calabi--Yau fibration. The following diagrams may help the reader visualize the proof of \Cref{thm:valuative-criterion}.

\smallskip

\begin{figure}[ht]
\centering
\begin{tikzcd}[ampersand replacement=\&]
	{(X_{\eta},cD_{\eta})} \& {(Y_{\eta},cD_{Y_{\eta}})} \& {(Y,cD_Y)} \& \\
	{(W_{\eta},cD_{W_{\eta}})} \& {(W^{\sta}_{\eta},T_{W^{\sta}_{\eta}}+cD_{W_{\eta}^{\sta}})} \& {(W^{\sta},T_{W^{\sta}}+cD_{W^{\sta}})} \& {(Z,T_Z+cD_Z)} \\
	{(X_{\eta},cD_{\eta})} \& {(X^{\sta}_{\eta},cD^{\sta}_{\eta})} \& {(X^{\sta},cD^{\sta})} \\
	\& {C_{\eta}} \& C^{\sta}
	\arrow["\sim", no head, from=1-1, to=1-2]
	\arrow[hook, from=1-2, to=1-3]
	\arrow[from=2-1, to=1-1]
	\arrow[from=2-1, to=3-1]
	\arrow[dashed, from=2-2, to=1-2]
	\arrow[hook, from=2-2, to=2-3]
	\arrow[from=2-2, to=3-2]
	\arrow[dashed, from=2-3, to=1-3]
	\arrow[from=2-3, to=2-4]
	\arrow[from=2-3, to=3-3]
	\arrow[from=2-4, to=3-3]
	\arrow[dashed, <->, from=3-1, to=3-2]
	\arrow[from=3-1, to=4-2]
	\arrow[hook, from=3-2, to=3-3]
	\arrow[from=3-2, to=4-2]
	\arrow[from=3-3, to=4-3]
	\arrow[hook, from=4-2, to=4-3]
\end{tikzcd}
\caption{The construction of birational modifications.}
\label{fig:stable-reduction}
\end{figure}
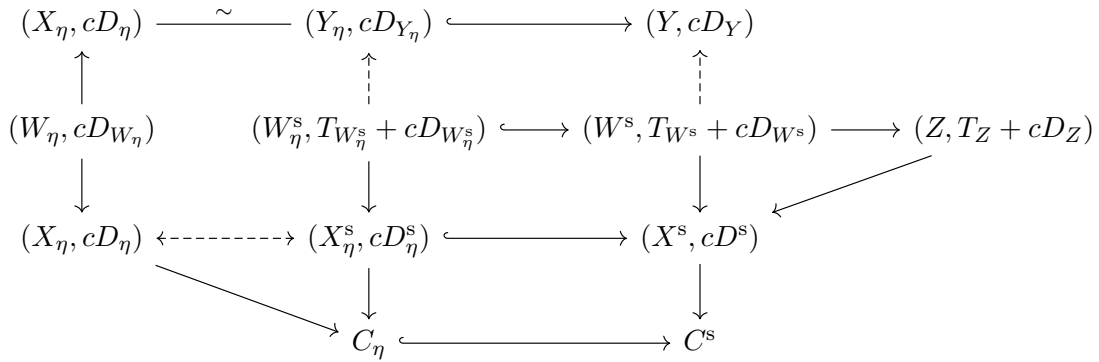

\begin{figure}[ht]
\centering
\begin{tikzcd}
	{\big(Y,cD_{Y}\big) = \big(Y^{(0)},cD_{Y^{(0)}}\big)} & {\big(Y^{(1)},cD_{Y^{(1)}}\big)} & \cdots & {\big(Y^{(k)},cD_{Y^{(k)}}\big) = \big(X,cD\big)} \\
	{\big(C^{\sta}, B^{(0)}+\bfM^{(0)}\big)} & {\big(C^{(1)}, B^{(1)}+\bfM^{(1)}\big)} & \cdots & {\big(C^{(k)}, B^{(k)}+\bfM^{(k)}\big).}
	\arrow[ dotted, from=1-1, to=1-2]
	\arrow["{\phi^{(0)}}", from=1-1, to=2-1]
	\arrow[ dotted, from=1-2, to=1-3]
	\arrow["{\phi^{(1)}}",from=1-2, to=2-2]
	\arrow[ dotted, from=1-3, to=1-4]
	\arrow[dotted, from=1-3, to=2-3]
	\arrow["{\phi^{(k)}}",  from=1-4, to=2-4]
	\arrow[ from=2-1, to=2-2]
	\arrow[ from=2-2, to=2-3]
	\arrow[ from=2-3, to=2-4]
\end{tikzcd}
\caption{The MMP of the generalized pairs with birational transformations of $(Y,cD_Y)$.}
\label{eq:diagram-of-MMP}
\end{figure}

\begin{proof} Let $\PCY$ denote the moduli stack of polarized Calabi--Yau pairs parametrizing general $f_{\eta}$ fibers with coarse moduli space $\bf{PCY}$; see \Cref{rem:stable-LCY}(5).
\begin{Step}\label{step_1_properness}
Construct the family $\big(X^{\sta},cD^{\sta} + \sum (F_i^{\sta})^{\red}\big)\to C^{\sta}\to \spec R$.   
\end{Step}

\smallskip

The locus in $C_\eta$ where the fibers of $f_{\eta}:\big(X_{\eta},(c+\epsilon)D_{\eta}\big)\rightarrow C_{\eta}$ are not semi-log canonical for $0<\epsilon\ll1$ is closed, and up to enlarging $R$ it consists of $n$ disjoint sections of $C_\eta\to \eta$, which we denote by $p_{j,\eta}\in C_{\eta}$ for $j=1,...,n$. Let \[\begin{tikzcd}[ampersand replacement=\&]
	{\big(\cX^{\sta}_\eta,(c+\epsilon)\cD^{\sta}_{\eta}\big)} \& {(\cC^{\sta}_\eta, \mtf{p}_{1,\eta},\ldots,\mtf{p}_{n,\eta})} \& \eta
	\arrow["{f^{\sta}_{\eta}}", from=1-1, to=1-2]
	\arrow["{\pi^{\sta}_{\eta}}", from=1-2, to=1-3]
\end{tikzcd}\] be the stacky KSBA-model of $f_\eta$, where $\mtf{p}_j\to \eta$ are (possibly trivial) gerbes. From \Cref{lemma_tsm_family_is_stable}, this induces a pointed twisted stable map
    \[
    (\cC^{\sta}_\eta, \mtf{p}_{1,\eta},\ldots,\mtf{p}_{n,\eta}) \ \longrightarrow \  \PCY
    .\] By \cite[Theorem 1.4.1]{AV02}, this map can be extended to a twisted stable map
    \[
    (\cC^{\sta}, \mtf{p}_{1},\ldots,\mtf{p}_{n}) \ \longrightarrow \ \PCY
    \]
    over $\spec R$, which, by the universality of $\PCY$, corresponds to a family \[\textstyle \big(\cX^{\sta}, (c+\epsilon)\cD^{\sta} + \sum_{j=1}^n\cF^{\ \sta}_j \big) \ \longrightarrow \  (\cC^{\sta}, \mtf{p}_1,\ldots,\mtf{p}_n),\]
    where we denote by $\mtf{p}_j$ the closure of the gerbes $\mtf{p}_{j,\eta}$ and by $\cF^{\ \sta}_j$ the fibers over $\mtf{p}_j$. We denote by $p_j$ the coarse spaces of $\mtf{p}_j$.
    Consider then its coarse moduli space 
    \[\textstyle \big(X^{\sta},(c+\epsilon)D^{\sta} + \sum_{j=1}^n(F^{\sta}_j)^{\red}\big)\ \longrightarrow \ (C^{\sta}, p_1,\ldots,p_n).\]
Here $F^{\sta}_j$ denotes the scheme theoretic fiber of $X^{\sta}\rightarrow C^{\sta}$ over $p_j$, which might be non-reduced; while the coarse space of $\cF^{\ \sta}_j$ is always reduced. In particular, the coarse moduli space of $\cF_j^{\ \sta}$ is $(F_j^{\sta})^{\red}$.

As the next birational modifications we will perform are all local over $C^{\sta}$, we fix an index $j$ and denote $\mtf{p}_j$ (resp. $p_j,\cF^{\ \sta}_j,F^{\sta}_j$) by $\mtf{p}$ (resp. $p,\cF^{\ \sta},F^{\sta}$).

\begin{lem}\label{step_2_properness}
    The $\bQ$-divisor $K_{X^{\sta}}+cD^{\sta}+(F^{\sta})^{\red}$ is  $\bQ$-linearly trivial over $C^{\sta}$.
\end{lem} 

\begin{proof}
Consider the coarse moduli space map 
\[\textstyle \rho:\big(\cX^{\sta}, (c+\epsilon)\cD^{\sta} + \cF^{\ \sta} \big) \ \longrightarrow \  \big(X^{\sta},(c+\epsilon)D^{\sta} + (F^{\sta})^{\red}\big),\] which can be ramified only over $(F^{\sta})^{\red}$ in codimension one, and hence $\rho$ is crepant by Riemann--Hurwitz. By the definition of the moduli stack $\PCY$, there is a line bundle $\cM$ on $\cC^{\sta}$ such that 
    \begin{equation}\label{eq:log-CY-condition}
        \pi^*\cM \ \simeq \  m(K_{\cX^{\sta}}+ c\cD^{\sta} + \cF^{\ \sta})
    \end{equation}
    for $m>0$ divisible enough. Up to replacing $\cM$ with a multiple of it, we can assume it is the pullback of a line bundle $M$ on $C^{\sta}$. Since $\rho_*\cO_{\cX^{\sta}} = \cO_{X^{\sta}}$, the pull-back map 
    \[
    \Pic(X^{\sta})\to \Pic(\cX^{\sta})
    \]
    is injective, so $m(K_{X^{\sta}}+cD^{\sta} + (F^{\sta})^{\red})\simeq \pi^*m(K_{\cX^{\sta}}+ c\cD^{\sta} + \cF^{\ \sta})$ agrees with the pull-back of $M$.
\end{proof}

\smallskip

\begin{Step}[Modification on $X_{\eta}$]\label{step_3_properness}
Construct a birational morphism $g_\eta\colon W_{\eta}\to X_\eta$
extracting the irreducible components of the fiber $F^{\sta}_{\eta}$,
together with a divisor $G_\eta$ on $W_\eta$, such that the canonical model
of $(W_{\eta},(c+\epsilon)D_{W_\eta} + G_{\eta})$ over $C_\eta$ agrees with that of $(X_\eta,(c+\epsilon)D_\eta)$ for every
$0\le \epsilon \ll 1$.
\end{Step}

\begin{lem}\label{lem:negative-discrepancy}
    The discrepancy of $(X^{\sta}_\eta,cD_{\eta}^{\sta}+(F^{\sta}_{\eta})^{\red})$ with respect to every divisor over the fiber $F_{\eta}\subseteq X_\eta$ is negative.
\end{lem}
\begin{proof}
    Let $\alpha^{\sta}$ be the log canonical threshold of $(X^{\sta}_\eta,cD^{\sta}_\eta)$ along the fiber $F_{\eta}^{\sta}$ over $p_{\eta}$, which is positive by \Cref{def_CY_fibration} and \Cref{lemma_singularities_tsp}. Note that the pairs $(X_\eta,cD_\eta+\alpha F_{\eta})$ and $(X^{\sta}_\eta,cD_{\eta}^{\sta}+\alpha^{\sta}F^{\sta}_{\eta})$ are crepant birational to one another over a neighborhood of $p_{\eta}\in C_{\eta}$; this follows immediately from the canonical bundle formula, see \cite[Theorem 8.5.1]{Kol07}. Consequently, the discrepancy of $(X_\eta,cD_\eta+\alpha F_{\eta})$ with respect to every component of $F^{\sta}_{\eta}$ is negative, and likewise the discrepancy of $(X^{\sta}_\eta,cD^{\sta}_\eta+\alpha^{\sta} F^{\sta}_{\eta})$ with respect to every component of $F_{\eta}$ is negative. Since $\alpha^{\sta} F^{\sta}_\eta \le (F^{\sta}_\eta)^{\red}$, the discrepancy of $(X^{\sta}_\eta,cD_{\eta}^{\sta}+(F^{\sta}_{\eta})^{\red})$ with respect to every component of $F_{\eta}$ is negative as well.
\end{proof}
By \cite[Theorem 1]{moraga2020extracting}, there exists a birational morphism
\[
g_{\eta}\ \colon\  W_{\eta} \ \longrightarrow \  X_\eta
\]
extracting precisely the components of $F^{\sta}_{\eta}$ that are not contained in $F_{\eta}$. Note that applying \cite[Theorem 1]{moraga2020extracting} may require replacing the fraction field $K(R)$ with a finite extension, since the results in \textit{loc.\ cit.} are stated over an algebraically closed field; this causes no issue, as we are free to replace $R$ by a finite extension. By construction, the $g_\eta$-exceptional divisors have strictly negative discrepancies with respect to the pair $(X_\eta,cD_\eta + \alpha F_\eta)$, where $\alpha$ is the log canonical threshold with respect to the fiber $F_{\eta}$ over $p_{\eta}$.

\smallskip

Fix an $\epsilon_0>0$ such that, for every rational number $0<\epsilon \le\epsilon_0$, by \Cref{step_2_properness} the $\bQ$-divisors
\begin{equation}\label{eq:log-can}
    K_{X^{\sta}}+(c+2\epsilon)D^{\sta} + \sum_{i=1}^n(F^{\sta}_i)^{\red}\sim_{\bQ,C^{\sta}}2\epsilon D^{\sta}
\qquad\text{and}\qquad
K_{X_\eta} + (c+2\epsilon)D_\eta\sim_{\bQ,C_{\eta}}2\epsilon D_{\eta}
\end{equation} are ample over $C^{\sta}$ and over $C_\eta$, respectively, and the corresponding pairs are log canonical. Then, for any rational number $0<\alpha'<\alpha$ sufficiently close to $\alpha$ and any $0\le\epsilon\le \epsilon_0$, $(X_\eta,(c+\epsilon)D_\eta +\alpha' F_{\eta})$ has negative discrepancy with respect to every $g_{\eta}$-exceptional divisor. Fix such a rational number $\alpha'\in(0,\alpha)$, and let $G_{\eta}$ be the $g_{\eta}$-exceptional $\bQ$-divisor determined by
    \begin{equation}\label{eq:defining-eq}
        K_{W_{\eta}} + (c+\epsilon_0)D_{W_{\eta}} + (g_{\eta})_*^{-1}(\alpha'F_{\eta}) + G_{\eta} \ \sim_\bQ\  g_{\eta}^*\big(K_{X_\eta}+ (c+\epsilon_0)D_\eta +\alpha' F_{\eta}\big).
    \end{equation}

\begin{lem}\label{lem:step_G_eta}
    The $\bQ$-divisor $G_{\eta}$ is effective, is supported on the proper transform of $F^{\sta}_{\eta}$ on $W_{\eta}$, and satisfies $\lfloor G_{\eta}\rfloor=0$. Moreover, the log canonical model of
    \[
    \big(W_{\eta}, (c+\epsilon)D_{W_\eta} + G_{\eta}\big)
    \]
    agrees with that of $\big(X_\eta,(c+\epsilon)D_\eta\big)$ over $C_\eta$ for every $0\le \epsilon \le \epsilon_0$.
\end{lem}
\begin{proof}
    The claims on $G_{\eta}$ follow immediately from the defining relation \eqref{eq:defining-eq}, together with the facts that $\alpha$ is the log canonical threshold and $\alpha'<\alpha$. It then follows from \cite[Corollary 3.53]{KM98} that the canonical models of $(W_{\eta}, (c+\epsilon)D_{W_\eta} + G_{\eta})$ and $(X_\eta,(c+\epsilon)D_\eta)$ over $C_\eta$ are isomorphic for every $0\le \epsilon \le \epsilon_0$.
\end{proof}

\begin{Step}[Modification of $X^{\sta}$]\label{step_5_properness} Construct a birational model $W^{\sta}\rightarrow X^{\sta}$ together with boundary divisors such that the relative log canonical model over $C^{\sta}$ exists, and its generic fiber is isomorphic to $\big(X_\eta,(c+\epsilon)D_\eta\big)$.
\end{Step}

By \cite[Theorem 2.10]{Has24}, one can take a crepant $\bQ$-factorial dlt modification
\[
h\ \colon\  \big(Z,(c+\epsilon_0)D_Z+T_Z\big)\ \longrightarrow \ \big(X^{\sta},(c+\epsilon_0)D^{\sta}\big),
\]
where $D_Z$ is the strict transform of $D^{\sta}$ and $T_Z$ is the sum of the $h$-exceptional divisors with coefficient $1$. Since there are no log canonical centers contained in the fibers of
\(X^{\sta}\to C^{\sta}\) by \Cref{lemma_singularities_tsp} and $(X^{\sta}, (c+2\epsilon_0)D^{\sta})$ is log canonical by (\ref{eq:log-can}), every component of $T_Z$ is horizontal, i.e. dominates $C^{\sta}$. Let $D_Z$ (resp. $F_Z^{\red})$ be the proper transform of $D^{\sta}$ (resp. $(F^{\sta})^{\red}$).
\begin{lem}\label{lem:crepant-birational}
    For every $0\leq \lambda\leq 1$ and every $ \epsilon\ge 0$, the morphism
\[
\big(Z,T_Z+(c+\epsilon)D_Z+\lambda F_Z^{\red}+Z_\xi\big)\longrightarrow \big(X^{\sta},(c+\epsilon)D^{\sta}+\lambda(F^{\sta})^{\red}+X^{\sta}_\xi\big) 
\]
is crepant birational.
\end{lem} 

\begin{proof}   
As $h:\big(Z,(c+\epsilon_0)D_Z+T_Z\big)\ \longrightarrow \ \big(X^{\sta},(c+\epsilon_0)D^{\sta}\big)$ is crepant, the morphism
\[
h\ \colon\  \big(Z,T_Z+(c+\epsilon)D_Z+Z_\xi\big)\ \longrightarrow \ \big(X^{\sta},(c+\epsilon)D^{\sta}+X^{\sta}_\xi\big)
\]
is also crepant, and hence the pair $\big(Z,T_Z+(c+\epsilon)D_Z+Z_\xi\big)$
is log canonical for every $0\le \epsilon \le \epsilon_0$.
Let $F_Z$ denote the fiber of $Z\to C^{\sta}$ over $p$. By \Cref{lemma_singularities_tsp}, the divisor $(F^{\sta})^{\red}$ is $\bQ$-Cartier. Since every component of $T_Z$ is horizontal, the pull-back of $(F^{\sta})^{\red}$ coincides with $F_Z^{\red}$; the same conclusion holds for $D^{\sta}$. Therefore, for every $\lambda$ and $\epsilon $, 
\begin{equation}\label{eq:crepant-birational}
   \big(Z,T_Z+(c+\epsilon)D_Z+\lambda F_Z^{\red}+Z_\xi\big)\longrightarrow \big(X^{\sta},(c+\epsilon)D^{\sta}+\lambda(F^{\sta})^{\red}+X^{\sta}_\xi\big) 
\end{equation}
is crepant birational.
\end{proof}

Let $\alpha^{\sta}$ be the log canonical threshold of $(X^{\sta}_\eta,cD^{\sta}_\eta)$ with respect to the fiber $F_{\eta}^{\sta}$ over $p_{\eta}$ as in \Cref{lem:negative-discrepancy}. As $\alpha^{\sta} (F_Z)_{\eta} \le (F_Z)^{\red}_{\eta}$ and every component of $F_Z$ dominates $\spec R$, one also has \begin{equation}\label{eq_alpa}
    \alpha^{\sta}F_Z \le (F_Z)^{\red}.
\end{equation}
From \Cref{lem:crepant-birational} and the proof of \Cref{lem:negative-discrepancy}, the pair \[\big(Z_\eta,T_{Z,\eta} + (c+\epsilon)D_{Z,\eta} + \alpha^{\sta} (F_Z)_{\eta}\big)\] has negative discrepancy with respect to every component of $F_{\eta}\subseteq X_\eta$, for every $\epsilon> 0$. Then by \Cref{eq_alpa} the same applies also for the pair \[(Z,T_Z + (c+\epsilon)D_Z + F_Z^{\red}+Z_\xi).\] Thus by \cite[Theorem 1]{moraga2020extracting} there exists a birational contraction
\[
h^{\sta}\ \colon \ W^{\sta} \ \longrightarrow \ Z
\]
extracting precisely the components of $F_\eta$ that are not components of $F_Z$, such that all extracted divisors are $\bQ$-Cartier. Denote by $G^{\sta}\subseteq W^{\sta}$ the proper transform in $W^{\sta}$ of the divisor $G_\eta$ defined in \Cref{eq:defining-eq}, and denote by $F_{W^{\sta}}$ the fiber of $W^{\sta}\to C^{\sta}$ over $p$.

\begin{lem}\label{lem:two-statements}
The following statements hold:
\begin{enumerate}
    \item[\textup{(1)}]\label{pt_1}
    The support of $G^{\sta}$ is contained in the support of $(h^{\sta})^{-1}_*F_Z$.

    \item[\textup{(2)}]\label{pt_2_step_6}
    The pair
    \[
    \big(W^{\sta},
    (h_*^{\sta})^{-1}(T_Z+(c+\epsilon)D_Z)
    +G^{\sta}
    +\epsilon'F_{W^{\sta}}
    +W^{\sta}_\xi\big)
    \]
    is log canonical for any $0\le \epsilon\le \epsilon_0$ and any sufficiently small $0<\epsilon'\ll 1$.
\end{enumerate}
\end{lem}

\begin{proof}
Since $Z$ is $\bQ$-factorial and every $h^{\sta}$-exceptional divisor is $\bQ$-Cartier, the variety $W^{\sta}$ is also $\bQ$-factorial. Statement~\textup{(1)} follows immediately from the definition of $G_\eta$; see \Cref{eq:defining-eq}. Indeed, each component of $G_\eta$ corresponds to a component of $F^{\sta}_\eta$, whose pushforward to $Z_\eta$ is a component of $(F_Z)_\eta$. Let $T_{W^{\sta}}$ and $D_{W^{\sta}}$ denote the strict transforms of $T_Z$ and $D_Z$ on $W^{\sta}$, respectively. Then there exists an \emph{effective} divisor $\Gamma$, supported on the whole $h^{\sta}$-exceptional locus, such that
\begin{equation}\label{eq_W^s-->Z_crepant}
h^{\sta}\colon
\big(
W^{\sta},
T_{W^{\sta}}
+(c+\epsilon_0)D_{W^{\sta}}
+(h_*^{\sta})^{-1}F_Z^{\red}
+\Gamma
+W^{\sta}_\xi
\big)
\longrightarrow
\big(
Z,
T_Z+(c+\epsilon_0)D_Z+F_Z^{\red}+Z_\xi
\big)
\end{equation}
is a crepant birational morphism. Moreover, we have
\begin{equation}\label{eq_support_G}
\Supp\big((h_*^{\sta})^{-1}F_Z^{\red}+\Gamma\big)
\ =\ \Supp F_{W^{\sta}}.
\end{equation} To prove \textup{(2)}, note that \Cref{lem:step_G_eta} shows that
$\lfloor G_\eta\rfloor=0$. Since $G^{\sta}$ is the closure of $G_\eta$, we have
$\lfloor G^{\sta}\rfloor=0$, and hence
$G^{\sta}<(h_*^{\sta})^{-1}F_Z^{\red}$.
Therefore, \Cref{eq_W^s-->Z_crepant,eq_support_G} imply that the pair
\[
\big(
W^{\sta},
(h_*^{\sta})^{-1}(T_Z+(c+\epsilon)D_Z)
+G^{\sta}
+\epsilon'F_{W^{\sta}}
+W^{\sta}_\xi
\big)
\]
is log canonical for every $0\le \epsilon\le \epsilon_0$ and every sufficiently small $0<\epsilon'\ll1$.
\end{proof}

\begin{lem}\label{Step_7_properness}
For any $0\le \epsilon \le \epsilon_0$, the relative log canonical model of
\[
\big(W^{\sta},
T_{W^{\sta}}+(c+\epsilon)D_{W^{\sta}}+G^{\sta}+W^{\sta}_\xi\big)
\]
over $C^{\sta}$ exists, and its generic fiber is isomorphic to the relative log canonical model of the pair
$(X_\eta,(c+\epsilon)D_\eta)$ over $C_\eta$. In particular, both $T_{W^{\sta}}$ and
$G^{\sta}$ are contracted by the log canonical model morphism.
\end{lem}

\begin{proof}
Since the support of the effective divisor
$(h_*^{\sta})^{-1}F_Z^{\red}+\Gamma-G^{\sta}$ coincides with the fiber of
$W^{\sta}\to C^{\sta}$ over $p$, and the pair
\[
\big(
W^{\sta},
T_{W^{\sta}}+(c+\epsilon_0)D_{W^{\sta}}
+(h_*^{\sta})^{-1}F_Z^{\red}
+\Gamma
+W^{\sta}_\xi
\big)
\]
constructed in \Cref{eq_W^s-->Z_crepant} is log canonical, the pair
$\big(W^{\sta},T_{W^{\sta}}+(c+\epsilon_0)D_{W^{\sta}}+G^{\sta}+W^{\sta}_\xi\big)$
has no log canonical centers lying over $p$. Thus every log canonical center
of $(W^{\sta},T_{W^{\sta}}+(c+\epsilon_0)D_{W^{\sta}}+G^{\sta}+W^{\sta}_\xi)$
dominates $C^{\sta}$. Applying \cite[Theorem 1.1]{HX13} together with
\cite[Theorem 2.10]{Has24}, the pair
\[
\big(
W^{\sta},
T_{W^{\sta}}+(c+\epsilon)D_{W^{\sta}}
+G^{\sta}
+W^{\sta}_\xi
\big)
\]
admits a relative log canonical model over $C^{\sta}$. Since the log canonical
model is obtained as the relative Proj of an $S_2$-algebra, its generic fiber
coincides with the log canonical model of the generic fiber. Similarly, after extracting from $W_\eta$ the (horizontal) log canonical centers which are irreducible components of $T_{W^{\sta}_\eta}$, from how the birational contractions $W^{\sta}\to X^{\sta}$ and $W_\eta\to X_\eta$ are constructed, the two pairs $(W^{\sta}_\eta
,T_{W^{\sta}_\eta}
+(c+\epsilon)D_{W^{\sta}_\eta}
+G^{\sta}_\eta)$ and $(W_\eta, (c+\epsilon)D_{W_\eta}
+G_\eta)$ are isomorphic in codimension one. Since the log canonical
model is obtained as the relative Proj of an $S_2$-algebra, the pairs 
\[
(W^{\sta}_\eta
,T_{W^{\sta}_\eta}
+(c+\epsilon)D_{W^{\sta}_\eta}
+G^{\sta}_\eta)\text{ and }(W_\eta, (c+\epsilon)D_{W_\eta}
+G_\eta)
\]have the same log canonical model over $C^{\sta}$, for every $0\le \epsilon\le\epsilon_0$. The log canonical model over $C^{\sta}$ for the pair $(W_\eta, (c+\epsilon)D_{W_\eta}
+G_\eta)$ is precisely the log canonical model of
$(X_\eta,(c+\epsilon)D_\eta)$ over $C_\eta$ by \Cref{lem:step_G_eta}. In particular,
both $T_{W^{\sta}}$ and $G^{\sta}$ are contracted by the relative log
canonical model morphism.
\end{proof}

We denote by
\[
\big(Y,(c+\epsilon)D_Y+Y_\xi\big)
\]
the relative log canonical model constructed in \Cref{Step_7_properness}. See \Cref{fig:stable-reduction} for a summary of the birational modifications involved in the construction.

\begin{Step}
    Study the geometry of $\big(Y,(c+\epsilon)D_{Y}+Y_\xi\big)\rightarrow C^{\sta}$, as a preparation for running MMP in the next step.
\end{Step}
Notice that $Y$ and $X^{\sta}$ are isomorphic away from the fiber over $p$, and that the map $Y\to C^{\sta}$ has equidimensional fibers.

\begin{lem}\label{lem:relative-log-can-model}
The relative log canonical model of $\big(W^{\sta},T_{W^{\sta}}+cD_{W^{\sta}} + G^{\sta}+W^{\sta}_\xi\big)$ over $C^{\sta}$ is $C^{\sta}$.
\end{lem}
\begin{proof}
Since $W^{\sta}_\xi$ is pulled back from $C^{\sta}$, it suffices to prove the
statement for the pair $\big(W^{\sta},T_{W^{\sta}}+cD_{W^{\sta}} + G^{\sta}\big)$.
Let $U\coloneqq C_\eta^{\sta}\cup(C^{\sta}\setminus\{p\})\subseteq C^{\sta}$
be the resulting open subscheme, and let $\Delta$ denote the restriction of
$T_{W^{\sta}}+cD_{W^{\sta}} + G^{\sta}$ to $W^{\sta}_U$. Consider the
Cartesian square
\[
\begin{tikzcd}
W^{\sta}_U \arrow[r, "i"] \arrow[d, "\alpha_U"'] & W^{\sta} \arrow[d, "\alpha"] \\
U \arrow[r, "j"'] & C^{\sta}.
\end{tikzcd}
\]
For $m>0$ sufficiently divisible, since $W^{\sta}_U$ contains all
codimension-one points of $W^{\sta}$, there is an isomorphism
\[
i_*\cO_{W^{\sta}_U}\big(m(K_{W^{\sta}_U} +\Delta)\big)\ \simeq\
\cO_{W^{\sta}}\big(m(K_{W^{\sta}}+ cD_{W^{\sta}} + G^{\sta})\big).
\]
Now observe that $\alpha_U$ is the canonical model of $(W^{\sta}_U,\Delta)$
over $U$. Indeed, the canonical model commutes with restriction to open
subsets of $C^{\sta}$, so it suffices to check this over $U_\eta=C^{\sta}_\eta$
and over $C^{\sta}\setminus\{p\}$ separately: the former holds by
\Cref{Step_7_properness}, while the latter holds because, over this locus,
our pair is crepant birational to $(X^{\sta},cD^{\sta}+(F^{\sta})^\red)\big|_{C^{\sta}\setminus\{p\}}$,
by \Cref{step_2_properness}. In particular, there is a line bundle $\cM_U$ on $U$ such that
\[
(\alpha_U)_*\cO_{W^{\sta}_U}\big(md(K_{W^{\sta}_U} + \Delta)\big)\ \simeq\ \cM_U^{\otimes d}
\]
for every $d\in \bN$ and every sufficiently divisible $m>0$. Since $p$ is a
smooth point of $C^{\sta}$, the line bundle $\cM_U$ extends to a line bundle
$\cM$ on $C^{\sta}$. We then obtain isomorphisms
\begin{align*}
\alpha_*\cO_{W^{\sta}}\big(md(K_{W^{\sta}}+ cD_{W^{\sta}} + G^{\sta})\big)
&\simeq \alpha_*i_* \cO_{W^{\sta}_U}\big(md(K_{W^{\sta}_U} + \Delta)\big) \\
&\simeq j_* (\alpha_U)_*\cO_{W^{\sta}_U}\big(md(K_{W^{\sta}_U} + \Delta)\big) \\
&\simeq j_*\cM_U^{\otimes d} \ =\ \cM^{\otimes d},
\end{align*}
and therefore
\[
\operatorname{Proj}_{C^{\sta}}\Big(\bigoplus_d \alpha_*\cO_{W^{\sta}}\big(md(K_{W^{\sta}}+ cD_{W^{\sta}} + G^{\sta})\big)\Big)
= \operatorname{Proj}_{C^{\sta}}\big(\bigoplus_d \cM^{\otimes d}\big) = C^{\sta}.
\]
\end{proof}

We next show that the pair $
\big(Y,(c+\epsilon)D_Y+Y_\xi\big)$
is independent of $\epsilon$.

\begin{lem}\label{lem:log-CY-over-Cs}
The divisor $D_Y$ is $\bQ$-Cartier and relatively ample over $C^{\sta}$. Moreover, the relative log canonical model of
$(Y,cD_Y+Y_\xi)$ over $C^{\sta}$ is $C^{\sta}$.
\end{lem}

\begin{proof}
Consider a relative good minimal model
\[ \mu\ :\ 
\big(W^{\min},T_{W^{\min}} + cD_{W^{\min}} + G^{\min} + W_\xi^{\min}\big)\ \longrightarrow \ C^{\sta}
\]
of $\big(W^{\sta}, T_{W^{\sta}} +cD_{W^{\sta}} +G^{\sta} +W^{\sta}_\xi\big)$
over $C^{\sta}$. Since the pair
$\big(W^{\sta}, T_{W^{\sta}} +(c+\epsilon)D_{W^{\sta}} +G^{\sta} +W^{\sta}_\xi\big)$
is lc, so is the pair
\[
\big(W^{\min},T_{W^{\min}} + (c+\epsilon)D_{W^{\min}} + G^{\min} + W_\xi^{\min}\big)
\] for $0\le \epsilon \ll 1$, and it is immediate to check that it admits a good minimal model over
$C^{\sta}$. Moreover, for $0\le \epsilon \ll 1$, a relative weak canonical model for
the pair
$\big(W^{\min},T_{W^{\min}} + (c+\epsilon)D_{W^{\min}} + G^{\min} + W_\xi^{\min}\big)$
will also be a relative weak canonical model for the original pair
$\big(W^{\sta}, T_{W^{\sta}} +(c+\epsilon)D_{W^{\sta}} +G^{\sta} +W^{\sta}_\xi\big)$
over $C^{\sta}$. In particular, since
$\big(W^{\min},T_{W^{\min}} + cD_{W^{\min}} + G^{\min} + W_\xi^{\min}\big)$
is a \textit{good} minimal model over $C^{\sta}$, \Cref{lem:relative-log-can-model}
gives
\[
K_{W^{\min}}
+T_{W^{\min}}
+(c+\epsilon)D_{W^{\min}}
+G^{\min}
+W^{\min}_\xi
\ \sim_{\bQ,C^{\sta}}\
\epsilon D_{W^{\min}}.
\]
Hence, for any $0<\epsilon\ll1$ and any sufficiently divisible integer $N=N(\epsilon)>0$,
\[\begin{split}
    Y \ &= \  \Proj_{C^{\sta}}\bigoplus_{m\ge 0}\mu_{*}
\cO_{W^{\min}}(
mN\big(
K_{W^{\min}}
+T_{W^{\min}}
+(c+\epsilon)D_{W^{\min}}
+G^{\min}
+W^{\min}_\xi
\big)
\big)\\
& =\ \Proj_{C^{\sta}}\bigoplus_{m\ge 0}\mu_{*}
\cO_{W^{\min}}(mN\epsilon D_{W^{\min}}).
\end{split}
\] The right-hand side is independent of $\epsilon$ because
varying $\epsilon$ only amounts to passing to a different Veronese subring of the variety, and hence $Y$ does not depend on $0<\epsilon\ll1$ either; in particular, $D_Y$ is $\bQ$-Cartier and
relatively ample over $C^{\sta}$. Moreover, since for any $0<\epsilon\ll1$, $\big(Y,(c+\epsilon)D_Y\big)$ is the relative log canonical model of $\big(W^{\sta}, T_{W^{\sta}} +(c+\epsilon)D_{W^{\sta}} +G^{\sta}\big)$ over $C^{\sta}$, we know that $(Y,cD_Y)$ is a relative weak canonical model of $\big(W^{\sta}, T_{W^{\sta}} +cD_{W^{\sta}} +G^{\sta}\big)$. Therefore, by \Cref{lem:relative-log-can-model}, the relative log canonical model of $(Y,cD_Y+Y_\xi)$ over $C^{\sta}$ is
$C^{\sta}$. In particular, one has \[K_Y+cD_{Y}\ \sim_{\bQ,C^{\sta}}0.\]
\end{proof}

\begin{lem}\label{step_10_properness}
For any sufficiently small $0<\epsilon'\ll1$, the pair
\[
\big(Y,(c+\epsilon)D_Y+\epsilon'F_Y+Y_\xi\big)
\]
is log canonical, where $F_Y$ denotes the fiber of $Y\to C^{\sta}$ over $p$.
\end{lem}

\begin{proof}
By \Cref{lem:two-statements}, the pair
\begin{equation}\label{pair_W_withfibers}
\big(
W^{\sta},
(h_*^{\sta})^{-1}(T_Z+(c+\epsilon)D_Z)
+G^{\sta}
+\epsilon'F_{W^{\sta}}
+W^{\sta}_\xi
\big)
\end{equation}
is log canonical for every $0\le \epsilon\le \epsilon_0$ and every sufficiently small
$0<\epsilon'\ll1$. Moreover, the same argument as in the proof of
\Cref{Step_7_properness} shows that this pair admits a relative log canonical
model over $C^{\sta}$. Since $F_{W^{\sta}}$ is the pullback of the Cartier divisor
$p\subset C^{\sta}$, we have
\[
(h_*^{\sta})^{-1}(T_Z+(c+\epsilon)D_Z)
+G^{\sta}
+\epsilon'F_{W^{\sta}}
+W^{\sta}_\xi
\ \sim_{\bQ,C^{\sta}}\ 
(h_*^{\sta})^{-1}(T_Z+(c+\epsilon)D_Z)
+G^{\sta}
+W^{\sta}_\xi.
\]
Therefore, by \cite[Lemma 1.28]{Kol13}, the relative log canonical model of
the pair in \Cref{pair_W_withfibers} over $C^{\sta}$ is
$\big(Y,(c+\epsilon)D_Y+\epsilon'F_Y+Y_\xi\big)$.
In particular, this pair is log canonical.
\end{proof}

\begin{Step}
Apply the canonical bundle formula to $(Y,cD_Y)$, run a minimal model program for the induced generalized pair, and lift the resulting birational transformations to $(Y,cD_Y)$.
\end{Step}

By \cite[Definition--Theorem 2.3]{Kol23}, the morphism
\[
f_Y \colon (Y,cD_Y)
\xrightarrow{\ \phi_Y\ }
C^{\sta}
\longrightarrow
\Spec R
\]
is a locally stable family over $\Spec R$. Moreover, $D_Y$ is ample over
$C^{\sta}$, and $\phi_Y$ is a log Calabi--Yau fibration over $C^{\sta}$ by
\Cref{lem:log-CY-over-Cs}. Apply the canonical bundle formula to $\phi_Y$, and let
$B^{(0)}$ and $\bfM^{(0)}$ denote the boundary and moduli parts,
respectively. Since $B^{(0)}$ is supported on the smooth locus of
$C^{\sta}\to \Spec R$, it is a $\bQ$-Cartier divisor. Run a minimal model
program for the generalized \emph{surface} pair
$(C^{(0)},B^{(0)}+\bfM^{(0)})\coloneqq(C^{\sta},B^{(0)}+\bfM^{(0)})$ over $\Spec R$:
\begin{equation}\label{eq:MMP1}
\big(C^{(0)},B^{(0)}+\bfM^{(0)}\big)
\longrightarrow
\big(C^{(1)},B^{(1)}+\bfM^{(1)}\big)
\longrightarrow
\cdots
\longrightarrow
\big(C^{(k)},B^{(k)}+\bfM^{(k)}\big),
\end{equation}
where $K_{C^{(k)}}+B^{(k)}+\bfM^{(k)}$ is relatively nef over $\Spec R$. Note that:
\begin{enumerate}
    \item this MMP \textit{only contracts rational tails} contained in the special fiber
    $C^{(i)}_\xi$ of $C^{(i)}\to \Spec R$;

    \item $\lfloor B^{(0)}\rfloor=\lfloor B^{(j)}\rfloor=0$ for every
    $0\le j\le k$;

    \item for every sufficiently small $0<\epsilon'\ll1$, the above MMP is
    also an MMP for the generalized pair
    $(C^{\sta},B^{(0)}+\bfM^{(0)}+\epsilon' p)$, which arises from applying
    the canonical bundle formula to the log canonical pair
    $(Y,cD_Y+\epsilon'F_Y)$.
\end{enumerate}

\smallskip
\begin{prop}\label{step_properties_of_Yk_properness}
    There exist birational contractions
\[
(Y,(c+\epsilon)D_{Y})\ \dashrightarrow \ (Y^{(1)},(c+\epsilon)D_{Y^{(1)}})\dashrightarrow \ \cdots \ \dashrightarrow \ (Y^{(k)},(c+\epsilon)D_{Y^{(k)}})
\]
which admit a fibration
\[
\phi^{(i)}\ : \ (Y^{(i)},(c+\epsilon)D_{Y^{(i)}}) \ \longrightarrow \  (C^{(i)},B^{(i)} + \textbf{M}^{(i)})
\] such that the following properties hold:
\begin{enumerate}
    \item[\textup{(1)}] $(Y^{(i)},D_{Y^{(i)}})_{\eta}= (X_\eta,D_{\eta})$;
    \item[\textup{(2)}] $\phi^{(i)}$ is a log Calabi--Yau fibration with boundary \textup{(resp.} moduli\textup{)} part $B^{(i)}$ \textup{(resp.}  $ \bfM^{(i)}$\textup{)};
    \item[\textup{(3)}] $D_{Y^{(i)}}$ is ample over $C^{(i)}$;
    \item[\textup{(4)}] $\phi^{(i)}$ has equidimensional fibers; and
    \item[\textup{(5)}] the pair $\big(Y^{(i)},(c+\epsilon)D^{(i)} + Y^{(i)}_\xi\big)$ is log canonical.
\end{enumerate}
\end{prop}

\begin{proof}
We proceed by induction. Thus, suppose that
$C^{(i)}\to C^{(i+1)}$ contracts a single rational tail of
$C^{(i)}_\xi$. The desired statement follows essentially from the same
argument as in \cite[Propositions 3.1 and 3.2]{ISZ25}, which we now recall
and slightly generalize to our setting.

The pair
$\big(Y^{(i+1)},cD_{Y^{(i+1)}}\big)$
is obtained by taking the relative log canonical model of the pair
$\big(Y^{(i)},(c+\epsilon)D_{Y^{(i)}}\big)$ over $C^{(i+1)}$ for
$0<\epsilon\ll1$; see \cite[Proposition 3.2]{ISZ25}.
In \cite[Proposition 3.2]{ISZ25}, the authors assume that the pair has klt singularities. However, the klt assumption is used only to apply \cite[Theorem 5.59]{HK}, which guarantees that the log canonical model of
$\big(Y^{(i)},(c+\epsilon)D_{Y^{(i)}}\big)$ over $C^{(i+1)}$ exists and is
independent of $\epsilon$ for $0<\epsilon\ll1$. We now explain why the same
conclusion remains valid in our setting.

Since
$\big(Y^{(i)},(c+\epsilon)D_{Y^{(i)}}+Y^{(i)}_\xi\big)$
is log canonical, the pair
$\big(Y^{(i)},(c+\epsilon)D_{Y^{(i)}}\big)$
has no log canonical centers contained in the closed fiber
$Y^{(i)}_\xi$. It follows from \cite[Theorem 1.1]{HX13} that
$(Y^{(i)},cD_{Y^{(i)}})\to C^{(i+1)}$
admits a good minimal model
\[
\big(Y^{\min},cD_{Y^{\min}}\big)\ \longrightarrow \ C^{(i+1)}.
\] By \cite[Lemma 3.1]{ISZ25}, this good minimal model contracts every divisor
of $Y^{(i)}$ mapping to the tail contracted by
$C^{(i)}\to C^{(i+1)}$. Moreover,
$\big(Y^{\min},cD_{Y^{\min}}\big)$
is still log Calabi--Yau over $C^{(i+1)}$, since a positive multiple of
$K_{Y^{\min}}+cD_{Y^{\min}}$ agrees in codimension one with the pullback of
a line bundle on $C^{(i+1)}$. Since the birational map
$Y^{(i)} \dashrightarrow Y^{\min}$
is a minimal model and
$\big(Y^{(i)},(c+\epsilon)D_{Y^{(i)}}+Y^{(i)}_\xi\big)$
is log canonical for every $0<\epsilon\ll1$, the pair
$\big(Y^{\min},(c+\epsilon)D_{Y^{\min}}\big)$
is also log canonical. Furthermore,
\[
K_{Y^{\min}}+(c+\epsilon)D_{Y^{\min}}
\ \sim_{\bQ,C^{(i+1)}}\ 
\epsilon D_{Y^{\min}}.
\]
Hence the relative log canonical model over $C^{(i+1)}$
\[
\big(Y^{\min},(c+\epsilon)D_{Y^{\min}}\big)
\ \dashrightarrow\ 
\big(Y^{(i+1)},(c+\epsilon)D_{Y^{(i+1)}}\big)
\]
is independent of $\epsilon$ for $0<\epsilon\ll1$. It is also the log
canonical model of
$\big(Y^{(i)},(c+\epsilon)D_{Y^{(i)}}\big)$
over $C^{(i+1)}$.

Properties \textup{(1)}--\textup{(3)} follow immediately from the
construction, while \textup{(4)} follows from
\cite[Lemma 3.1]{ISZ25}. Finally, \textup{(5)} follows from
\cite[Lemma 1.28]{Kol13}, since the original pair
$\big(Y,(c+\epsilon)D_Y+Y_\xi\big)$
is log canonical and
$(c+\epsilon)D_Y+Y_\xi
\sim_{\bQ,\Spec R}
(c+\epsilon)D_Y$.
\end{proof}

These birational transformations are summarized in \Cref{eq:diagram-of-MMP}.

\begin{lem}\label{step_properness_on_slct}
    For every smooth point $q\in C^{(i)}_\xi$, we have
    \[
    \slct\big(Y^{(i)}_\xi,(c+\epsilon)D_{Y^{(i)}_\xi}; (\phi^{(i)})^{-1}(q)\big)>0.
    \]
    In other words, there is no log canonical center of $\big(Y^{(i)},(c+\epsilon)D_{Y^{(i)}} + Y^{(i)}_\xi\big)$ lying over $q$.
\end{lem}
\begin{proof}
Suppose $q$ is a smooth point over a neighborhood of which the morphism $\big(Y^{(i)},(c+\epsilon)D_{Y^{(i)}}\big)\to C^{(i)}$ fails to be a KSBA-stable family. Then $q$ is one of the following:
\begin{enumerate}
    \item[(a)] the intersection of the section $p\colon \Spec R\to C^{\sta}$ (introduced in the paragraph preceding \Cref{step_2_properness}) with $C^{(i)}_{\xi}$, which we denote by $p_\xi$; or
    \item[(b)] the image, under $C^{\sta}\to C^{(i)}$, of a contracted rational tail.
\end{enumerate}

\smallskip
\noindent\textbf{Case (a): $q=p_\xi$.} Consider the pair $(Y,cD_Y + \epsilon' F_Y + Y_\xi)$, which is log canonical for $0<\epsilon'\ll 1$ by \Cref{step_10_properness}. As before, we run an MMP for
    \[
    (C^{\sta},B^{(0)} + \bfM^{(0)} + \epsilon' p),
    \]
    which, for $0<\epsilon'\ll 1$, agrees with an MMP for $(C^{\sta},B^{(0)} + \bfM^{(0)})$. This produces a weak canonical model
    \[
    \widetilde{\phi}^{(i)}\ : \ \big(\widetilde{Y}^{(i)},(c+\epsilon)\widetilde{D}_{Y^{(i)}} + \epsilon' F_Y^{(i)}\big) \ \longrightarrow \  C^{(i)}
    \]
    whose fibers are equidimensional. Consequently, no divisor of $\widetilde{Y}^{(i)}_\xi$ is contracted to a closed point of $C^{(i)}_\xi$, so $F_Y^{(i)}$ is the pullback of the Cartier divisor $p^{(i)}\subseteq C^{(i)}$ (the two agree in codimension one). In particular,
    \[
    K_{\widetilde{Y}^{(i)}}+(c+\epsilon)\widetilde{D}_{Y^{(i)}}\sim_{\bQ, C^{(i)}} \epsilon\widetilde{D}_{Y^{(i)}}
    \]
    is ample over $C^{(i)}$. Hence $(\widetilde{Y}^{(i)},\widetilde{D}_{Y^{(i)}})$ and $(Y^{(i)},D^{(i)})$ are isomorphic in codimension one, and therefore isomorphic, being the Proj of the same graded ring. In particular, the pair
    \[
    \big(Y^{(i)}, (c+\epsilon)D^{(i)} + \epsilon' (\phi^{(i)})^{-1}(p^{(i)})\big)
    \]
    is log canonical, and by \cite[Lemma 1.28]{Kol13}, so is
    \[
    \big(Y^{(i)}, (c+\epsilon)D^{(i)} + \epsilon' (\phi^{(i)})^{-1}(p^{(i)}) + Y^{(i)}_\xi\big).
    \]
    Adjunction now yields the desired statement.

\smallskip
\noindent\textbf{Case (b).} We may assume $q$ is the image of a rational tail $\Gamma$ contracted by $C^{(i-1)}\to C^{(i)}$. Choose a general smooth point $s_\xi$ on $\Gamma$; after possibly a base change of $\Spec R$, we may assume $s_\xi$ is the intersection of a section $s$ of $C^{(i-1)}\to \Spec R$ with $C^{(i-1)}_\xi$. Let $F^{(i-1)}$ denote the inverse image of $s$ under $\phi^{(i-1)}$. Then $\big(Y^{(i-1)}, (c+\epsilon)D_{Y^{(i-1)}}+\epsilon' F^{(i-1)}\big)$ is log canonical for all $0<\epsilon'\ll 1$. Since an MMP for $(C^{(i-1)},B^{(i-1)}+\bfM^{(i-1)})$ is also an MMP for $(C^{(i-1)},B^{(i-1)}+\epsilon's+\bfM^{(i-1)})$, the same argument as in Case (a) applies to give the result.
\end{proof}

\begin{Step}\label{step_properness_construction_fam_over_twisted_curve}
Modify the family near the nodes of $C^{(k)}\to \Spec R$ to obtain a log Calabi--Yau fibration over a twisted curve, and verify the stability conditions.
\end{Step}

\begin{prop}\label{prop:stacky-curve}
For every $i=0,\dots,k$, the following statements hold.
\begin{enumerate}
    \item[\textup{(1)}]
    There exists a twisted curve
    $\cC^{(i)}\to \Spec R$ together with a flat projective morphism
    \[
    f^{(i)}\colon
    \big(\cY^{(i)},c\cD^{(i)}\big)
    \longrightarrow
    \cC^{(i)}
    \]
    whose induced morphism on coarse moduli spaces is
    \[
    \phi^{(i)}\colon
    \big(Y^{(i)},cD_{Y^{(i)}}\big)\longrightarrow C^{(i)}.
    \]

    \item[\textup{(2)}]
    There exists a finite (possibly empty) set
    $S^{(i)}\subseteq (\cC^{(i)}_\xi)_{\mathrm{sm}}$
    such that
    \[f^{(i)}\colon
    \big(\cY^{(i)},(c+\epsilon)\cD^{(i)}\big)
    \longrightarrow
    \cC^{(i)}
    \]
    is KSBA-stable over
    $\cC^{(i)}\setminus S^{(i)}$
    for every sufficiently small $0<\epsilon\ll1$.
\end{enumerate}
\end{prop}

\begin{proof}
This follows from observation \textup{(1)} following \Cref{eq:MMP1}, namely,
that the MMP only contracts rational tails contained in the special fiber
$C^{(i)}_\xi$ of $C^{(i)}\to \Spec R$. Since every rational tail is contracted
to a smooth point, there exists an open neighborhood
$U\subseteq C^{(i)}$ containing all the nodes of $C^{(i)}_\xi$ such that the
morphism $C^{\sta}\to C^{(i)}$ is an isomorphism over $U$.

After possibly shrinking $U$, we may further assume that its preimage in
$C^{\sta}$ is disjoint from the section $p$. It follows that, over $U$, the
family
\[
(Y^{(i)},cD_{Y^{(i)}})|_U\to U
\]
agrees with
\[
(X^{\sta},cD^{\sta})|_U\to U.
\]
Therefore, it suffices to replace $U$ by the corresponding open substack of
$\cC^{\sta}$ and
$(Y^{(i)},cD_{Y^{(i)}})|_U$
by
$(\cX^{\sta},c\cD^{\sta})|_U$.
\end{proof}

\smallskip

For simplicity, we denote the morphism
$f^{(k)}:(\cY^{(k)},c\cD^{(k)})\to \cC^{(k)}$
by $f:(\cX,c\cD)\to \cC$. By \Cref{step_properties_of_Yk_properness}, its generic fiber is isomorphic to
$f_{\eta}:(\cX_{\eta},c\cD_{\eta})\to \cC_{\eta}$. Let $(X,cD)\to C$ denote the coarse moduli space of $(\cX,c\cD)\to \cC$.

\smallskip
Denote by $\lambda_{\Cho}^{(i)}$ the Chow $\bQ$-line bundle of the fibration
$f^{(i)}$ from \Cref{prop:stacky-curve}(2), and by $\Lambda_{\Cho}^{(i)}$ its
descent to $C^{(i)}$.

\begin{prop}\label{step_properness_where_we_prove_stability}
For every irreducible component $\Gamma\subseteq C^{(k)}_\xi$ with
$\deg_{\Gamma}(K_{C^{(k)}}+B^{(k)}+\bfM^{(k)})= 0$, one has
$\deg_{\Gamma}(\Lambda_{\Cho}^{(k)})>0$.
\end{prop}
\begin{proof}
By the construction of $C^{\sta}$, there is a morphism
$\rho\colon C^{(0)} =C^{\sta}\to \bf{PCY}$. Let $L_{\Cho}$ be the Chow line
bundle (equivalently, the CM line bundle) on $\bf{PCY}$. Here we do not
distinguish between the moduli of stable Calabi--Yau pairs and the
corresponding KSBA moduli space, so $L_{\Cho}$ is ample on $\bf{PCY}$. By
\Cref{cor_degree_of_chow_minimized_by_pull_back_from_KSBA}, for sufficiently
divisible $m>0$ we have
\[
\Lambda^{\otimes m}_{\Cho,f_\eta}\ \simeq \ \big(\rho^*L^{\otimes m}_{\Cho}|_{C^{\sta}_\eta}\big) \otimes \cO_{C^{\sta}_\eta}(E_\eta)
\]
for an effective Cartier divisor $E_\eta$ with $\Supp E_\eta= \sum_j p_{j,\eta}$;
see \Cref{step_1_properness} for the definition of $p_{j,\eta}$. Denote by
$E^{(i)}$ the closure of $E_\eta$ in $C^{(i)}$. Then
\[
(\Lambda_{\Cho}^{(0)})^{\otimes m}\ =\ \rho^*L_{\Cho}^{\otimes m}\otimes \cO_{C^{\sta}}(E^{(0)})
\qquad\text{and}\qquad
E^{(0)}\ \ge\ \sum_i p_{i};
\]
the equality holds because the two line bundles agree in codimension one:
over the generic fiber, and over the locus where the fibration
$\big(X^{(0)},(c+\epsilon)D^{(0)}\big)\to C^{\sta}=C^{(0)}$ is
KSBA-stable. By the construction of $C^{(i)}$, the divisors $E^{(i)}$ are
Cartier, since each of their components is contained in the smooth locus of
$C^{(i)}\to \Spec R$.

Since $(\cC^{\sta},\mathfrak{p}_1 + \ldots + \mathfrak{p}_n)\to \PCY$ is a
twisted \textit{stable} map, for every rational tail or rational bridge
$\Gamma$ of $C^{\sta}_\xi$ we have
\begin{center}
either $E^{(0)}\cap \Gamma \neq \emptyset$, \quad or \quad
$\deg_{\Gamma}(\rho^*L_{\Cho})>0$.
\end{center}
In particular, for every $\Gamma$ that is a tail or a rational bridge of
$C^{\sta}_\xi$, we have
\[
\deg\big(\Lambda_{\Cho}^{(0)}|_\Gamma\big)>0.
\]
One checks that this condition is preserved after contracting a tail, using
that
\[
\deg\big(\Lambda_{\Cho}^{(i)}|_{C^{(i)}_\xi}\big)\ =\ \deg\big(\Lambda_{\Cho}^{(i)}|_{C^{(i)}_\eta}\big)\ =\ \deg\big(\Lambda_{\Cho}^{(i+1)}|_{C^{(i+1)}_\eta}\big)\ =\ \deg\big(\Lambda_{\Cho}^{(i+1)}|_{C^{(i+1)}_\xi}\big),
\]
where the first and third equalities hold because $C^{(i)}\to \Spec R$ is
flat, and the second holds because $C^{(i)}_\eta = C^{(i+1)}_\eta$.
\end{proof}

By \Cref{step_properties_of_Yk_properness}(5), the morphism
$(X,(c+\epsilon)D)\to \Spec R$ is locally stable. We claim that
$(\cX,c\cD)\to \Spec R$ is also locally stable. Indeed, this is clear over
the locus where $\cX\to X$ is an isomorphism. Over the nodes of
$\cC_\xi$, the morphism
$f\colon (\cX,c\cD)\to \cC$
is locally stable and
$\cC\to \Spec R$
is a family of nodal curves. Hence the composition
\[
(\cX,c\cD)\longrightarrow \cC \longrightarrow \Spec R
\]
is locally stable. We now verify that
$f\colon (\cX,c\cD)\to \cC\to \Spec R$
is a stable log Calabi--Yau fibration (cf. \Cref{def_CY_fibration}). Indeed,

\begin{itemize}
    \item Condition \textup{(LCY)} follows from
    \Cref{step_properties_of_Yk_properness}(2) and
    \Cref{eq:log-CY-condition}.

    \item Condition \textup{(Sing)} follows from
    \Cref{step_properness_on_slct}.

    \item Condition \textup{(Qmap)} follows from
    \Cref{prop:stacky-curve}(2).

    \item Condition \textup{(Stab)} follows from
    \Cref{step_properties_of_Yk_properness}(3) and
    \Cref{step_properness_where_we_prove_stability}.
\end{itemize} 
\end{proof}

Denote by $\delta$ the degree of the map $C_\eta\to \bf{PCY}$ in \Cref{thm:valuative-criterion}, and by $d$ the maximum number of irreducible components of a curve $\widehat{C}$ with a Kontsevich stable map $\widehat{C}\to \bf{PCY}$ of degree $\delta$.

\begin{cor}\label{cor:number-components-bounded}
     The number of irreducible components of $\cC_\xi$ is bounded above by $d$.
\end{cor}

\begin{proof}
By \Cref{lemma_tsm_family_is_stable} the map $C^{\sta}\to \bf{PCY}$ is $n$-pointed Kontsevich-stable, of degree $\delta$, and $C^{\sta}\to C$ is a contraction.
\end{proof}

\begin{thm}\label{thm:extension-of-automorphism}
Let \[\begin{tikzcd}[ampersand replacement=\&]\ 
	{(\cX,c\cD)} \& \cC \& {\Spec R}
	\arrow["f", from=1-1, to=1-2]
	\arrow[ from=1-2, to=1-3]
\end{tikzcd}\] be the family of stable log Calabi--Yau fibrations constructed in \Cref{thm:valuative-criterion}. Then any automorphism $s_{\eta}$ of the generic fiber extends uniquely over $\Spec R$.
\end{thm}

\begin{proof}
Roughly speaking, the argument amounts to checking that all the steps in \Cref{thm:valuative-criterion} can be performed equivariantly, together with the input that the moduli space of pointed twisted stable maps is separated.

In \Cref{step_1_properness} of \Cref{thm:valuative-criterion}, we start by replacing the original family
with a fibration
\[
f^{\sta}_\eta:(\cX^{\sta}_\eta,c\cD^{\sta}_\eta)\to\cC^{\sta}_\eta
\]
over a pointed stacky curve $(\cC_\eta^{\sta};\mathfrak{p}_{1,\eta},\ldots ,\mathfrak{p}_{n,\eta})$, and the family above comes from a morphism $\cC_\eta^{\sta}\to \PCY$. We then use pointed twisted stable maps to extend this family over $\Spec R$:
\[
f^{\sta}:(\cX^{\sta},c\cD^{\sta})\to(\cC^{\sta}, \mathfrak{p}_{1},\ldots ,\mathfrak{p}_{n}).
\]
It follows from \cite[Theorem 3.1]{BV24} that the family $f^{\sta}_\eta$ is uniquely determined by the original family $f_\eta$, so $s_\eta$ induces an automorphism of the fibration $f^{\sta}_\eta$, which we denote by $\sigma_\eta$. Observe that this need not induce an isomorphism of the corresponding \textit{pointed} twisted stable map, as the gerbes $\mathfrak{p}_{i,\eta}$ need not be fixed. Instead, it permutes the gerbes $\mathfrak{p}_{i,\eta}$, sending 
\[
\mathfrak{p}_{i,\eta}\mapsto \mathfrak{p}_{\rho(i),\eta}\ \ \ \text{ for } \ \ \ \rho \in S_n.
\]
As the moduli stack of pointed twisted stable maps is separated, the automorphisms $\sigma_\eta,\tau_{\eta}$ 
\[
\xymatrix{(\cX_\eta^{\sta},c\cD^{\sta}_\eta)\ar[r]^-{\sigma_\eta} \ar[d] & (\cX_\eta^{\sta},c\cD_\eta^{\sta})\ar[d] \\ (\cC_\eta^{\sta};\mathfrak{p}_{1,\eta},\ldots,\mathfrak{p}_{n,\eta})\ar[r]^-{\tau_\eta} & (\cC_\eta^{\sta};\mathfrak{p}_{\rho^{-1}(1),\eta},\ldots,\mathfrak{p}_{\rho^{-1}(n),\eta}).}
\]
extend uniquely to isomorphisms $\sigma,\tau$ as follows: 
\[
\xymatrix{(\cX^{\sta},c\cD^{\sta})\ar[r]^-{\sigma} \ar[d] & (\cX^{\sta},c\cD^{\sta})\ar[d] \\ (\cC^{\sta};\mathfrak{p}_{1},\ldots,\mathfrak{p}_{n})\ar[r]^-{\tau} & (\cC^{\sta};\mathfrak{p}_{\rho^{-1}(1)},\ldots,\mathfrak{p}_{\rho^{-1}(n)}).}
\]
Now, if we denote by $(X^{\sta},cD^{\sta})\to C^{\sta}$ the morphism on coarse moduli spaces, then:
\begin{enumerate}
    \item the boundary part is supported at $p_i$, with coefficient $b_i^{\sta}$;
    \item since $s_\eta$ is an isomorphism of the \textit{morphism} $(X_\eta,cD_\eta)\to C_\eta$, we have $b_i^{\sta} = b_{\rho(i)}^{\sta}$ for every $i$.
\end{enumerate}
In other words, the boundary part for the morphism $(X^{\sta},cD^{\sta})\to C^{\sta}$ is $\tau$-invariant.

Since $s_\eta$ sends $D_\eta$ to itself, we have $\sigma_\eta^*(\cD_\eta) = \cD_\eta$, and since $\cD$ is horizontal, we also have $\sigma^*(\cD) = \cD$. In particular, from the canonical bundle formula, the moduli part for the morphism $(X^{\sta},cD^{\sta})\to C^{\sta}$ is also $\tau$-invariant. 

Now, recall that in \Cref{Step_7_properness} we constructed a pair $(Y,cD)\to C^{\sta}$ that is birational to $(X^{\sta}, cD^{\sta})\to C^{\sta}$ over $C^{\sta}$, agrees with $(X_\eta,cD_\eta)\to C_\eta$ over the generic fiber, and whose boundary part $B_Y$ is the closure of the boundary part for $(X_\eta,cD_\eta)\to C_\eta$.
The boundary part for the original morphism $(X_\eta,cD_\eta)\to C_\eta$ is supported at the points $p_{i,\eta}$, the coarse spaces of $\mathfrak{p}_{i,\eta}$; let $b_i$ be the coefficient of the point $p_{i,\eta}$. Since $s_\eta$ is an isomorphism of the pair $(X_\eta,cD_\eta)\to C_\eta$, it follows that \[b_i = b_{\rho(i)}\text{ for every }i.\]
In particular, the boundary part for $(X_\eta,cD_\eta)\to C_\eta$ is $\tau_\eta$-invariant, so $B_Y$ is also $\tau$-invariant. 
It follows from \cite[Theorem 8.5.1]{Kol07} that the moduli part $\bfM_Y$ of $(Y,cD_Y)\to C^{\sta}$ depends only on the generic fiber, so it agrees with the one for $(X^{\sta},cD^{\sta})\to C^{\sta}$, which is $\tau$-invariant. Therefore $\tau_\eta$ extends to an automorphism $\tau^{(k)}$ of the minimal model $(C^{(k)},B^{(k)} + \bfM^{(k)})$ of the generalized pair $(C^{\sta}, B_Y + \bfM_Y)$. 

Since the coarse space of $\cC$ agrees with $C^{(k)}$,  we can find a $\tau^{(k)}$-invariant divisor $S$ supported on the smooth locus of $C^{(k)}\to \Spec R$ such that the pair $(X,(c+\epsilon)D + \phi^{-1}(S))$ is KSBA-stable.
As the KSBA-moduli space is separated, the automorphism $\tau_\eta$ extends to an automorphism of $(X,(c+\epsilon)D+\phi^{-1}(S))$. As the fibration $X\to C$ is the ample model of $(X,cD)$, this automorphism induces an automorphism of the fibration $(X,cD)\to C$. Since $\tau$ is an isomorphism of the \textit{twisted} curve $\cC^{\sta}$, the induced automorphism of $C$ lifts, by the construction of $\cC$, to an automorphism of the original twisted curve $\cC$. To summarize, for the original family $(\cX,c\cD)\to \cC$, we have produced an extension of the automorphism $\tau_\eta$ to $\cC$, and to $(\cX,c\cD)$ away from finitely many (stacky) fibers of the central fiber. Since $\cD$ is ample over $\cC$, we may write $\cX$ as $\Proj$ of a $\tau$-invariant algebra, and hence this automorphism extends to an automorphism of $(\cX,c\cD)\to \cC$.
\end{proof}

\smallskip

\subsection{Separatedness}

We prove that the moduli stack of stable log Calabi--Yau fibrations is separated.

\begin{thm}[Separatedness]\label{thm:separatedness}
Let $(R,\eta,\xi)$ be a DVR, and for $i=1,2$ let
\[
\sigma_i\colon\
(\cX_i,\cD_i)\ \xrightarrow{\ \phi_i\ }\
\cC_i\ \xrightarrow{\ \pi_i\ }\
\Spec R
\]
be two families of stable log Calabi--Yau fibrations over $\Spec R$. Assume that the
generic fibers of $\sigma_i$ are normal and that
\[\begin{tikzcd}[ampersand replacement=\&]
	{\zeta_{\eta}\colon} \& {(\cX_1,\cD_1)|_{\eta}} \&\& {(\cX_2,\cD_2)|_{\eta}} \\
	{\vartheta_{\eta}\colon} \& {\cC_{1,\eta}} \&\& {\cC_{2,\eta}}
	\arrow["\begin{array}{c} \\ \sim \end{array}", from=1-2, to=1-4]
	\arrow[from=1-2, to=2-2]
	\arrow[from=1-4, to=2-4]
	\arrow["{  \sim}", from=2-2, to=2-4]
\end{tikzcd}\]
is a commutative diagram of isomorphisms. Then the isomorphisms extend to a unique commutative diagram of isomorphisms
\[\begin{tikzcd}[ampersand replacement=\&]
	{\zeta\colon} \& {(\cX_1,\cD_1)} \&\& {(\cX_2,\cD_2)} \\
	{\vartheta\colon} \& {\cC_{1}} \&\& {\cC_{2}}
	\arrow["\begin{array}{c} \\ \sim \end{array}", from=1-2, to=1-4]
	\arrow[from=1-2, to=2-2]
	\arrow[from=1-4, to=2-4]
	\arrow["{  \sim}", from=2-2, to=2-4]
\end{tikzcd}\]
over $\Spec R$.
\end{thm}

\begin{proof}
Let
\[
(X_i,cD_i)\ \xrightarrow{\ f_i\ }\ C_i\ \longrightarrow\ \Spec R
\]
be the coarse spaces associated with the two families $\sigma_i$. Let $\widehat{C}$ be the
normalization of the graph of the birational map $C_1\dashrightarrow C_2$, and let
$(\widetilde{C},\widetilde{C}_{\xi})$ be a log resolution of $(\widehat{C},\widehat{C}_{\xi})$.
After a finite base change of $\Spec R$ and a further minimal log resolution, we may assume
that $(\widetilde{C},\widetilde{C}_{\xi})$ is an SNC pair. Let $(C,C_{\xi})$ be the relative
log canonical model of $(\widetilde{C},\widetilde{C}_{\xi})$ over $\widehat{C}$, and let
$g_i\colon C\to C_i$ be the induced morphisms for $i=1,2$.

\begin{lemma}\label{lem:cannot-contract-by-both}
If $\Gamma$ is a rational tail of $C_{\xi}$, then $(g_i)_*\Gamma\neq 0$ for at least one of
$i=1,2$.
\end{lemma}

\begin{proof}
Suppose, to the contrary, that $(g_i)_*\Gamma=0$ for both $i=1,2$. Then $\Gamma$ is contracted
by the induced morphism $C\to\widehat{C}$. Since $\Gamma$ is a rational tail, it is a
$(-1)$-curve on $C$, contradicting the minimality of the relative log minimal model $C$.
\end{proof}

\begin{lemma}\label{lem:E_1=E_2}
The two pairs $(X_1,cD_1)$ and $(X_2,cD_2)$ are crepant.
\end{lemma}

\begin{proof}
Let $X$ be a normal variety fitting into a diagram as follows, with the $h_i$ isomorphisms
over $\eta$:
\[
\begin{tikzcd}[ampersand replacement=\&]
X_1 \&\& X \&\& X_2 \\
C_1 \&\& C \&\& C_2
\arrow["f_1", from=1-1, to=2-1]
\arrow["h_1"', from=1-3, to=1-1]
\arrow["h_2", from=1-3, to=1-5]
\arrow["f"', from=1-3, to=2-3]
\arrow["f_2"', from=1-5, to=2-5]
\arrow["g_1"', from=2-3, to=2-1]
\arrow["g_2", from=2-3, to=2-5]
\end{tikzcd}.
\]
As both $D_1$ and $D_2$ are horizontal and agree over the generic point $\eta$, their proper
transforms on $X$ coincide; we denote this common divisor by $D$. Since
$(X_i,cD_i+X_{i,\xi})$ is log canonical for $i=1,2$, and $h_i$ only extracts divisors lying
over $X_{i,\xi}$, one can write
\[
K_X+cD\ \sim_{\bQ}\ h_i^*(K_{X_i}+cD_i)+E_i
\]
for some effective $h_i$-exceptional $\bQ$-divisor $E_i$. Hence
\[
h_1^*(K_{X_1}+cD_1)-h_2^*(K_{X_2}+cD_2)\ \sim_{\bQ}\ E_2-E_1.
\]
Since $K_{X_2}+cD_2$ is nef, the divisor $E_2-E_1$ is $h_1$-anti-nef. Moreover,
\[
(h_1)_*(E_2-E_1)=(h_1)_*E_2\geq 0.
\]
By the negativity lemma, it follows that $E_1-E_2\geq 0$. By symmetry, we also have
$E_2-E_1\geq 0$, and hence $E_1=E_2$.
\end{proof}

\begin{lemma}\label{prop:contracted-to-node}
If $\Gamma$ is an irreducible component of $C_{\xi}$ that is $g_i$-exceptional for either
$i=1$ or $i=2$, then $\Gamma$ is contracted to a node of $C_{i,\xi}$.
\end{lemma}

\begin{proof}
It suffices to treat the case where $\Gamma$ is a rational tail. Indeed, if a non-tail
component of $C_{\xi}$ is contracted to a smooth point of either $C_1$ or $C_2$, then there
exists a rational tail of $C_{\xi}$ contracted to the same point. Thus, without loss of
generality, we may assume that $g_1(\Gamma)=p$ is a \emph{smooth} point of $C_{1,\xi}$.

By \Cref{lem:cannot-contract-by-both}, the curve $\Gamma$ is not contracted by $g_2$; denote
its image by $\Gamma_2:=g_2(\Gamma)$. Let $\Xi$ be the proper transform in $X$ of an
irreducible component of $X_2|_{\Gamma_2}$. Then $\Xi$ is a divisor on $X$ mapping onto
$\Gamma\subseteq C$. Since $\Gamma$ is contracted by $g_1$, the divisor $\Xi$ is contracted to
the point $p$ by the morphism $f_1\circ h_1$. As $f_1$ has equidimensional fibers and
$\dim C_1=2$, there are no divisors on $X_1$ contracted to a point of $C_1$. Hence $\Xi$ must
be contracted by $h_1$.

By condition (Sing) of \Cref{def_CY_fibration}, the fibers of $f_1$ over smooth points of
$C_1\to\Spec R$ do not contain any log canonical center of
$\big(X_1,(c+\epsilon)D_1+X_{1,\xi}\big)$. Therefore, for a general very ample divisor $H$ on
$C_1$ passing through $p$, the pair $\big(X_1,cD_1+\epsilon f_1^*H+X_{1,\xi}\big)$ is log
canonical for any $0<\epsilon\ll 1$. Since $g_1\circ f(\Xi)=p$, we may write
\[
K_X+cD+(h_1)_*^{-1}(X_{1,\xi}+\epsilon f_1^*H)
\ \sim_{\bQ}\
h_1^*\big(K_{X_1}+cD_1+X_{1,\xi}+\epsilon f_1^*H\big)+R+a\,\Xi,
\]
where $R$ is $h_1$-exceptional with $\Xi\not\subseteq\Supp R$, and
\[
a\ =\ a_{\Xi}(X_1,cD_1)-\mult_{\Xi}(X_{1,\xi})-\epsilon\,\ord_{\Xi}\big(h_1^*f_1^*H\big)\ \geq\ -1.
\]
In particular, $a_{\Xi}(X_1,cD_1)>0$, and it follows that
\[
\Xi\ \subseteq\ \Supp E_1=\Supp E_2.
\]
This is impossible, since $E_2$ is $h_2$-exceptional whereas $\Xi$ is not contracted by $h_2$.
\end{proof}

\begin{lemma}\label{step_construction_twisted_curve_separatedness}
There exists a twisted curve $\cC$ with coarse space $C$, together with morphisms
\[
\cC\longrightarrow\cC_1
\qquad\text{and}\qquad
\cC\longrightarrow\cC_2
\]
lifting the maps $g_i\colon C\longrightarrow C_i$.
\end{lemma}

\begin{proof}
There is a map $C_\eta\to\cC_1\times_R\cC_2$. Let $\Gamma$ be an irreducible component of
$C_\xi$, and let $\eta_{\Gamma}$ be its generic point. As $\eta_{\Gamma}$ is a codimension-one
point of $C$, the local ring $A:=\cO_{C,\eta_{\Gamma}}$ is a DVR, together with a morphism
$\rho\colon R\to A$. As the central fiber $C_\xi$ is reduced, the image $\rho(\varpi)$ of a
uniformizer $\varpi$ of $R$ is a uniformizer for $A$. By the valuative criterion for the
properness of tame Deligne--Mumford stacks \cite[Tag 0CLY]{stacks-project}, there is a diagram
as follows for some integer $N>0$:
\[
\begin{tikzcd}[ampersand replacement=\&]
	\&\& {\Spec K(A)} \& {\cC_1\times_{R}\cC_2} \\
	{\Spec A[t]/(t^N-\rho(\varpi))} \&\& {\Spec A} \& {\Spec R}
	\arrow[from=1-3, to=1-4]
	\arrow[from=1-3, to=2-3]
	\arrow[from=1-4, to=2-4]
	\arrow["\exists"{description}, dotted, from=2-1, to=1-4]
	\arrow[from=2-1, to=2-3]
	\arrow[from=2-3, to=2-4]
\end{tikzcd},
\]
where $K(A)$ is the fraction field of $A$. Observe that $\Spec K(A)$ is the generic point of
$C$. Then, up to possibly replacing $\Spec R$ with $\Spec R[t]/(t^N-\varpi)$, we have the
commutative diagram
\[
\begin{tikzcd}[ampersand replacement=\&]
	{\Spec K(A)} \&\& {\cC_1\times_{R}\cC_2} \\
	{\Spec A} \&\& {\Spec R}
	\arrow[from=1-1, to=1-3]
	\arrow[from=1-1, to=2-1]
	\arrow[from=1-3, to=2-3]
	\arrow["\exists"{description}, dotted, from=2-1, to=1-3]
	\arrow[from=2-1, to=2-3]
\end{tikzcd}.
\]
In particular, the rational map $C_\eta\to\cC_1\times_{R}\cC_2$ extends to $C$ away from a
finite set of closed smooth points $\{x_1,\ldots,x_m\}\subseteq C_\xi$:
\[
\begin{tikzcd}[ampersand replacement=\&]
C\setminus\{x_1,\ldots,x_m\} \arrow[d] \arrow[r] \& \cC_1\times_{R}\cC_2 \arrow[d] \\
C \arrow[r] \& C_1\times_{R}C_2.
\end{tikzcd}
\]
As $C$ has only $A_m$-singularities, one can take the canonical smooth covering stack
$\cC\to C$ (cf.\ \cite[2.8 and 2.9]{vistoli1989intersection} and \cite[Theorem 4.6]{FMN10}).
Formally locally over an $A_m$-singularity of $C$, the covering stack is of the form
\[
\big[(\Spec k[\![u,v]\!])/\bmu_m\big]\ \longrightarrow\ \Spec k[\![x,y,z]\!]/(xy-z^m),
\]
where a generator $\zeta\in\bmu_m$ acts via $\zeta\ast u=\zeta u$ and
$\zeta\ast v=\zeta^{-1}v$. It is straightforward to check that $\cC\to\Spec R$ is a twisted
curve. Then by \cite[Lemma 2.1]{twisted_map_2}, the diagram
\[
\begin{tikzcd}[ampersand replacement=\&]
\cC\setminus\{x_1,\ldots,x_m\} \arrow[d] \arrow[r] \& \cC_1\times_{R}\cC_2 \arrow[d] \\
\cC \arrow[r] \& C_1\times_{R}C_2.
\end{tikzcd}
\]
admits a lifting $\cC\to\cC_1\times_{R}\cC_2$.
\end{proof}

Let $\wt{\varphi}_i\colon(\wt{\cX}_i,c\wt{\cD}_i)\to\cC$ be the pullback of the family
$(\cX_i,c\cD_i)\to\cC_i$ along $\gamma_i\colon\cC\to\cC_i$, for $i=1,2$.

\begin{lemma}\label{prop:two-pullback-families-isomorphic}
The two families
\[
\wt{\varphi}_i\colon(\wt{\cX}_i,c\wt{\cD}_i)\longrightarrow\cC,\qquad i=1,2,
\]
are isomorphic over $\cC$.
\end{lemma}

\begin{proof}
By \Cref{prop:contracted-to-node}, each component of $C_{\xi}$ contracted by $C\to C_i$ is
contracted to a node of $C_{i,\xi}$. Since the fibers of $(\cX_i,(c+\epsilon)\cD_i)\to\cC_i$
over these nodes are KSBA-stable for any $0<\epsilon\ll1$ and $i=1,2$, the uniqueness of the
KSBA-stable limit implies that the two families $(\wt{\cX}_i,c\wt{\cD}_i)\to\cC$ agree in
codimension one. As $\cD_i$ is relatively ample and $\wt{\cX}_i$ is $S_2$, the two families
agree, since they are the relative $\operatorname{Proj}$ over $\cC$ of the same algebra
$\bigoplus_{m\geq0}(\wt{\varphi}_i)_{*}\cO_{\wt{\cX}_i}(m\wt{\cD}_i)$.
\end{proof}

\begin{prop}\label{step_sep_isom_fam}
The two families $\sigma_i\colon(\cX_i,c\cD_i)\to\cC_i\to\Spec R$ for $i=1,2$ are isomorphic.
\end{prop}

\begin{proof}
It suffices to show that $\cC_1$ and $\cC_2$ are isomorphic over $R$. Indeed, if this is the
case, then one can take $\cC=\cC_i$ and $(\cX_i,c\cD_i)\simeq(\wt{\cX}_i,c\wt{\cD}_i)$, so the
desired result follows from \Cref{prop:two-pullback-families-isomorphic}.

Suppose that $\cC_1$ and $\cC_2$ are not isomorphic. We may assume that there exists a rational
bridge or rational tail $\Gamma$ of $\cC$ whose coarse space is contracted by $g_1$ but not by
$g_2$; denote $g_2(\Gamma)$ by $\Gamma_2$. Then the family
\[
(\wt{\cX}_1,(c+\epsilon)\wt{\cD}_1)|_\Gamma\ =\ (\wt{\cX}_2,(c+\epsilon)\wt{\cD}_2)|_\Gamma\longrightarrow\Gamma
\]
is trivial, with all fibers KSBA-stable. In particular, the Chow line bundle associated with
the family over $\Gamma$ is trivial, and hence so is the one associated with
$(\cX_2,(c+\epsilon)\cD_2)|_{\Gamma_2}\to\Gamma_2$. However, this contradicts the stability of
the family $(\cX_2,c\cD_2)\to\cC_2$.
\end{proof}

To complete the proof, it suffices to show that any automorphism of the generic fiber of a
fixed family over $R$ extends uniquely to the whole family. This is the content of
\Cref{thm:extension-of-automorphism}.
\end{proof}

\section{Boundedness of stable log Calabi--Yau fibrations}\label{sec:Boundedness of stable log Calabi--Yau fibrations}

Throughout this section, we fix an admissible quadruple $\Phi$ and fix an $\epsilon$-coefficient $\epsilon_0$ for $\Phi$ (cf. \Cref{defn:epsilon-coefficient}). The goal of this section is to establish the following result, which
guarantees that the moduli space of stable log Calabi--Yau fibrations of
non-degenerate type with fixed numerical invariants is bounded.

\begin{thm}[Boundedness]\label{thm_bounding_fam}
There exist a smooth scheme $B$ of finite type over $\bC$ and a family
\[
(\cX_B,c\cD_B)\longrightarrow \cC_B\longrightarrow B
\]
of stable log Calabi--Yau fibrations of non-degenerate type with numerical invariant $\Phi$, such that every such fibration is isomorphic to a fiber over $B$.
\end{thm}

\begin{defn}
Fix a positive integer $d$. Define $\cS_{\Phi,d}$ to be the collection of stable log Calabi--Yau fibrations
\[
\phi \colon (\cX,c\cD)\to \cC
\]
with numerical invariant $\Phi$ such that $\cC$ has at most $d$ irreducible components.
\end{defn}

To prove \Cref{thm_bounding_fam}, we first prove the following weaker version.

\begin{thm}\label{thm_weak_boundedness}
For any fixed integer $d\geq 1$, there exists a smooth scheme $B$ of finite type, a family of twisted curves
$\cC_B\to B$, and a flat proper morphism $(\cX_B,c\cD_B)\longrightarrow \cC_B$ such that
\begin{enumerate}
\item[\textup{(1)}] every object of $\cS_{\Phi,d}$ is isomorphic to a fiber \((\cX_b,c\cD_b)\to\cC_b\) for some $b\in B$; and
\item[\textup{(2)}] for every $b\in B$, the fiber $(\cX_b,c\cD_b)\to\cC_b$ belongs to $\cS_{\Phi,d}$.
\end{enumerate}
\end{thm}

\begin{proof}
Let $\phi\colon (\cX,c\cD)\to \cC$ be an element of \(\cS_{\Phi,d}\), and let
\((X,cD)\to C\) denote its coarse space. Note that the set of all possible curves \(C\) is bounded since the number of components and the genus are bounded.

\begin{lem}\label{cor_cms_are_bounded}
There exists a finite collection of Hilbert polynomials
\(\sigma_1(t),\ldots,\sigma_n(t)\) such that the following holds. For every pair
\((X,cD)\to C\) as above, there exists a divisor $F$ which is a sum of fibers of \(X\to C\) such that \((X,cD+F;D)\) determines a point of
\(\coprod_{i=1}^n \MM^{\Bir}_{\sigma_i(t)}\), where $\MM^{\Bir}_{\sigma_i(t)}$ is the moduli stack of stable minimal models with volume function $\sigma_i(t)$.
\end{lem}
\begin{remark}
In the definition of the moduli of stable minimal models (cf.
\Cref{defn:data-set} and \Cref{thm:stable-minimal-model}), there are other
numerical invariants that one needs to fix. However, in this context these
should be clear and are determined by $\Phi$; hence, in order not to
introduce too many notations, we are deliberately imprecise here and only
record the volume function $\sigma(t)$.
\end{remark}

\begin{proof}
By assumption, the number of irreducible components of \(C\) is bounded by \(d\). Let \(k\) be the number of irreducible components of \(C\). Choose general points \(p_1,\ldots,p_{3k}\in C\) such that each irreducible component of \(C\) contains three of them. Then
\[
\bigl(X,cD+f^*(p_1+\cdots+p_{3k});D\bigr)
\]
is a \emph{stable minimal model} (cf. \Cref{defn:stable-mimimal-model}). Moreover, the divisor \[K_X+(c+t)D+f^*(p_1+\cdots+p_{3k})\] has volume polynomial \(V(t,3k)\) for \(0<t\ll 1\). Hence the above pair determines a point of \(\MM^{\Bir}_{V(t,3k)}\), and therefore a point of
\(\coprod_{i=1}^{d}\MM^{\Bir}_{V(t,3i)}\).
\end{proof}

Fix a surjective morphism $B\longrightarrow \coprod_i \MM^{\Bir}_{\sigma_i(t)}$ from a smooth scheme $B$, and let \[(X,cD+F;D) \ \longrightarrow \ B\] be the pullback of the universal family over \(\coprod_i \MM^{\Bir}_{\sigma_i(t)}\). We will construct a locally closed partial decomposition of $B$ (cf.~\cite[Definition~10.83]{Kol23}), which, by abuse of notation, will still be denoted by $B$, such that the resulting family
\((X,cD+F;tD)\rightarrow B\)
satisfies the desired properties.

\smallskip

To simplify the statements, we use the following terminology. 
\begin{defn}
Let $\cP$ be a property of families over $B$ that is preserved by pullback. We say that $\cP$ is
\emph{constructible} if the set of points of $B$ satisfying $\cP$ is constructible. Equivalently, if there exists a morphism of finite type
\(i_{\cP}\colon B^{\cP}\to B\) from a (smooth) scheme such that:
\begin{enumerate}
\item for every $b\in B$, if the fiber of the family over the point $b$ satisfies $\cP$, then there is $b^{\cP}\in B^{\cP}$ mapping to $b$; and
\item the pullback of the family to \(B^{\cP}\) is such that every fiber satisfies \(\cP\).
\end{enumerate}
\end{defn}
In other terms, every fiber over $B$ satisfying $\cP$ has a representative over $B^{\cP}$, and every fiber over $B^{\cP}$ satisfies $\cP$.

\smallskip

The following lemmas show that all the desired properties are constructible. Throughout the argument, whenever a property \(\cP\) is shown to be constructible, we replace \(B\) and the family over it by the corresponding base \(B^{\cP}\) and the pullback family.

\begin{lemma}\label{first_lemma_boundedness}
   The following properties of the morphism
    \[
    (X,cD+F;D)\ \longrightarrow \ B
    \]
   are constructible:
   \begin{enumerate}
       \item[\textup{(1)}] there is a flat family of nodal curves $C\to B$ and a morphism $X\stackrel{\pi}{\to} C\to B$ such that for every $b\in B$, $C_b$ is the ample model of $(X_b,cD_b + F_b)$;
       \item[\textup{(2)}] $C$ has locally constant topological type, and
       \item[\textup{(3)}] every fiber of $\pi$ has pure-dimension $n-1$.
       
   \end{enumerate}
\end{lemma}
\begin{proof}{From \cite[Lemma 2.14]{Bir22}, we can stratify $B$ so that $K_{X/B}+cD+F$ is semiample over $B$. Consider then the semiample model $X\to C\to B$.}
By \cite[Theorem 3.19]{Kol23}, flatness is representable, so we can stratify $B$ so that $C\to B$ is flat.  Since for every $b\in B$, $X_b\to C_b$ is the ample model for $K_{X_b}+cD_b+F_b$, the map $C\to B$ is a family of nodal curves.
We can again stratify so that the topological type of $C\to B$ is constant on connected components. From upper-semicontinuity of the fiber dimension applied to $X\to C$, and from the fact that $C\to B$ is universally closed, we can further stratify so that the fibers of $X\to C$ have dimension $n-1$.    
\end{proof}

\begin{lemma} The following property of the family
\(
(X,cD+F;D)\to C\to B
\)
is constructible: 
 the fibers of
\[
\bigl(X,(c+\epsilon_0)D+F\bigr)\to C
\] over the smooth locus of \(C\to B\),
are semi-log canonical away from possibly finitely many points of each fiber of \(C\to B\).
\end{lemma}

\begin{proof}
By \cite[Proposition~III.9.7]{Har77}, the restriction $X|_{C^{\sm}}\longrightarrow C^{\sm}$ is flat. By \cite[Corollary~4.45]{Kol23}, the locus where the fibers fail to be semi-log canonical is closed. Let $Z\subseteq C^{\sm}$ denote this locus. Then the desired property is equivalent to requiring that the fibers of \(Z\to B\) are finite. Since \(\rho\colon Z\to B\) is of finite type, upper semicontinuity of fiber dimension implies that the locus
\[
\big\{\,z\in Z \mid \dim Z_{\rho(z)}\le 0\,\big\}
\]
is open, so its image in \(B\) is constructible. Hence the property is represented by a constructible subscheme of $B$.
\end{proof}

\begin{lemma}\label{lemma_line_bundle_G}
The following property of the family
  \((X,cD+F;D)\xrightarrow{\pi}  C \to  B
    \)
is represented by an open subscheme of \(B\): the divisor \(F\) is supported on a union of smooth fibers of \(X\to C\). In particular, after replacing \(B\) by this open subscheme, there exists a line bundle \(\cG\) on \(C\) such that
\[K_{X/B}+cD\ \sim_{\bQ}\ \pi^*\cG.\]
\end{lemma}

\begin{proof}
Let \(\pi(F)\subseteq C\) denote the scheme-theoretic image of \(F\) under \(\pi\). By upper semicontinuity of fiber dimension, after replacing \(B\) by a nonempty open subscheme, we may assume that the fibers of \(\pi(F)\to B\) are zero-dimensional. Since the family over \(B\) consists of stable minimal models, \(\pi(F)\) is disjoint from the nodal locus of \(C\to B\). It follows that \(\pi(F)\to B\) is finite \'etale. Hence, \'etale locally over \(B\), the scheme \(\pi(F)\) is a disjoint union of sections of \(C\to B\).

There exists a natural morphism
\[
F\to \pi^{-1}(\pi(F)),
\]
and the sublocus of $B$ over which this morphism is an isomorphism is open in \(B\). Indeed, its complement is the image under the proper morphism
\(\pi^{-1}(\pi(F))\to B\) of the support of the ideal sheaf and the Coker of
\[
\cO_{\pi^{-1}(\pi(F))}
\longrightarrow
\cO_F.
\] Therefore, the locus in \(B\) over which
\(F=\pi^{-1}(\pi(F))\) is open. In particular, over this open subscheme, \(F\) is a union of smooth fibers of \(\pi\). Let \(\cM\) be the line bundle on \(C\) satisfying
\(K_{X/B}+cD+F\sim_{\bQ}\pi^*\cM\), which exists from the first point of \Cref{first_lemma_boundedness}.
Since \(F=\pi^{-1}(\pi(F))\), we have
\[
K_{X/B}+cD+\pi^{-1}(\pi(F))
\ \sim_{\bQ}\ 
\pi^*\cM,
\]
and therefore
\[
K_{X/B}+cD
\ \sim_{\bQ}\ 
\pi^*\bigl(\cM(-\pi(F))\bigr).
\]
Setting \(\cG:=\cM(-\pi(F))\) completes the proof.
\end{proof}

As a consequence, both \(K_{X/B}\) and \(cD\) are \(\bQ\)-Cartier. Moreover, away from finitely many points on each fiber of \(C\to B\), the fibers of
\((X,(c+\epsilon_0)D)\to C\) are KSBA-stable. After replacing \(B\) by a union of connected components, we may assume that for every \(b\in B\) and every general point \(z\in C_b\),
\[
v(t)=\vol\bigl(K_{X_z}+cD_z+tD_z\bigr).
\]
In particular, the general fibers of \((X,cD;D)\to C\) determine points of
\(\PCY_{v(t)}\).

\begin{lemma}
The following property of the family
\(
(X,cD) \to C\to  B
\)
is constructible: for every \(b\in B\), the semi-log canonical thresholds of
\((X_b,cD_b)\) with respect to the fibers of
\(X_b\to C_b\) over smooth points of \(C_b\) are positive.
\end{lemma}
\begin{proof}
This follows from \Cref{lemma_slct_positive_is_open}.
\end{proof}

The following result is well-known.

\begin{lemma}
There exists a locally closed stratification of $B$ followed by a finite cover such that the pullback family of \(C\to B\) admits a simultaneous normalization
\[
C^\nu\to C,
\]
and the preimage of the double locus is given by pairwise disjoint sections of \(C^\nu\to B\).
\end{lemma}

Denote by \((C^\nu,\Delta)\to C\to B\) the simultaneous normalization of \(C\to B\), where \(\Delta\) is the preimage of the nodal locus. Consider the pullback of the family \((X,(c+\epsilon_0)D)\to C\) to \(C^\nu\). Then there exists a dense open subset \(U\subseteq C^\nu\) such that the induced family over \(U\) is KSBA-stable, and hence determines a natural morphism
\[
U \ \longrightarrow \ \PCY_{v(t)}.
\]
We assume that \(U\) is maximal with this property.

\begin{lemma}
After passing to a surjective morphism \(B'\to B\) from a smooth scheme \(B'\) of finite type, there exists a root stack
\(\cC^\nu\to C^\nu\) such that the morphism
\(U\to \textup{\PCY}_{v(t)}\) extends across \(\Delta\).
\end{lemma}

\begin{proof}
We argue by Noetherian induction on \(B\). Since we only need to extend the morphism across \(\Delta\), we may replace \(C^\nu\) by an open neighborhood of \(\Delta\). Working one connected component at a time, we may further assume that \(B\) is connected and smooth, and let \(\eta\in B\) be its generic point.

By \cite[Proposition~1.6]{bejleri2025root}, the morphism
\(U_\eta\to \PCY_{v(t)}\) extends to a morphism from a root stack
\(\cC^\nu_\eta\to C^\nu_\eta\). Since both the construction of root stacks and the stack \(\PCY_{v(t)}\) are of finite type, this extension spreads out over an open neighborhood \(V\subseteq B\) of \(\eta\). Applying Noetherian induction to the complement \(B\setminus V\), we obtain the desired extension after replacing \(B\) by another smooth scheme of finite type.
\end{proof}
Let \(\widetilde{\Delta}\subseteq \cC^\nu\) denote the preimage of \(\Delta\), which is locally a gerbe over $B$. We denote by \(\widetilde{\cU}\subseteq \cC^\nu\) the union of the preimage of \(U\) and \(\widetilde{\Delta}\), which is an open substack of $\cC^{\nu}$. 

\begin{lemma}\label{lemma_boundedness_constructing_stack}
Up to replacing $B$ by another smooth scheme of finite type $B'$ equipped with a morphism $B'\to B$, there exists a family of twisted curves $\cC\to B$ together with a flat and proper morphism
\[
(\cX,c\cD)\ \longrightarrow \ \cC\ \longrightarrow \  B,
\]
whose coarse moduli space is
\[
(X,cD)\ \longrightarrow \  C\ \longrightarrow \  B,
\]
such that there exists an open neighborhood $\VV$ of the nodal locus $\cN\subseteq \cC$ for which the restricted family
\[
\big(\cX,(c+\epsilon_0)\cD\big)|_{\VV}\ \longrightarrow \   \VV
\]
is KSBA-stable.
\end{lemma}

\begin{proof}
Since $\Delta$ is the preimage of the double locus under the normalization $C^\nu\to C$, there is an involution on $\Delta$, which we denote by $\tau$. Consider now
\[\operatorname{Isom}_B(\widetilde{\Delta},\widetilde{\Delta}),\]
the algebraic stack parametrizing automorphisms of $\widetilde{\Delta}$ over $B$. This is an open substack of the Hom-stack $\operatorname{Hom}_B(\widetilde{\Delta},\widetilde{\Delta})$ of \cite{HR19}. Any automorphism of $\widetilde{\Delta}$ induces an automorphism of its coarse moduli space $\Delta$, so there is a morphism
\[
\operatorname{Isom}_B(\widetilde{\Delta},\widetilde{\Delta})\to \operatorname{Isom}_B({\Delta},{\Delta}).
\]
Now, the map $\tau$ gives a section $B\to \operatorname{Isom}_B({\Delta},{\Delta})$, and one can consider the following fiber diagram
\[
\xymatrix{B'\ar[r]\ar[d] & B\ar[d]\\\operatorname{Isom}_B(\widetilde{\Delta},\widetilde{\Delta})\ar[r] & \operatorname{Isom}_B({\Delta},{\Delta}).}
\]
By \cite[Theorem C.2]{AOV} together with the universal property of coarse moduli spaces, the bottom arrow is of finite type. We may replace $B'$ with a resolution of singularities; over $B'$ we then have an automorphism, which we denote by $\widetilde{\tau}$, of (the pullback of) $\widetilde{\Delta}$ inducing $\tau$ on coarse moduli spaces.

Consider the following fiber product
\[
\xymatrix{
I\ar[rr] \ar[d]  && \PCY_{v(t)}\ar[d]^{\operatorname{diag}}\\ \widetilde{\Delta}\ar[rr]^-{(F,F\circ\widetilde{\tau})}&& \PCY_{v(t)}\times \PCY_{v(t)}
}
\]
where $F\colon \widetilde{\Delta}\hookrightarrow \widetilde{\cU}\to \PCY_{v(t)}$ is the natural map. Since $\PCY_{v(t)}$ is a proper Deligne--Mumford stack, its diagonal morphism is finite, and hence $I\rightarrow\wt{\Delta}$ is finite. Moreover, $I$ parametrizes isomorphisms between the two families over $\widetilde{\Delta}$ corresponding to $F$ and $F\circ \tau$.

After replacing $B$ with the resolution of singularities of an \'{e}tale surjective cover of $I$, we may assume that there is an isomorphism over $B$ between the two maps $F$ and $F\circ\tau$. We can then glue $\cC^\nu$ to itself via $\tau\colon \widetilde{\Delta}\to \widetilde{\Delta}$, using \cite[Theorem 1.8]{alper2024artin}, to obtain a family of nodal stacky curves $\cC\to B$. More explicitly, if we write $\widetilde{\Delta}=\widetilde{\Delta}_1 \sqcup \widetilde{\Delta}_2$ with $\tau$ sending $\widetilde{\Delta}_1$ to $\widetilde{\Delta}_2$, the family of curves is the pushout of the diagram
\[\begin{tikzcd}[ampersand replacement=\&]
	{\widetilde{\Delta}_1\sqcup \widetilde{\Delta}_2} \&\& {\cC^\nu} \\
	{\widetilde{\Delta}_1} \&\& {\cC}
	\arrow[hook, from=1-1, to=1-3]
	\arrow["{(\Id,\widetilde{\tau} )}"', from=1-1, to=2-1]
	\arrow[from=1-3, to=2-3]
	\arrow[from=2-1, to=2-3]
\end{tikzcd}\]
Since $\widetilde{\Delta}_1\sqcup \widetilde{\Delta}_2\to B$ is flat over $B$ and the pushout is a \emph{geometric pushout} (cf. \cite[Definition 4.1 and Theorem 4.2]{alper2024artin}), it follows from \Cref{lemma_pushout} that its formation commutes with base change with respect to $B$. Hence we may verify that \(\cC\to B\) has nodal twisted curves as
geometric fibers by taking \(\TT=\overline{k(b)}\) for a point \(b\in B\), in which case the claim is trivial.
Moreover, since $B$ is smooth, \(\cC\to B\) is flat by miracle flatness
\cite[\href{https://stacks.math.columbia.edu/tag/00R4}{Tag 00R4}]{stacks-project}.

By construction, the morphism \((\cX,c\cD)\to \cC\) agrees with
\((X,cD)\to C\) away from the singular locus of \(C\). Consider $(X',cD')\to B$, the coarse moduli space of $(\cX,c\cD)$. By \Cref{lemma_line_bundle_G}, $F$ consists of fibers of $X\to C$ over certain smooth points of $C$, and we let $F'$ denote the fibers of $X'\to C$ over the same points. Then the morphism $(X',(c+\epsilon)D'+F')\to B$ is a family of stable minimal models; together with $(X,(c+\epsilon)D+F)\to B$, this gives two morphisms $\xi,\xi'\colon B\to \MM^{\Bir}_{\sigma(t)}$ to the moduli of stable minimal models. Consider the Isom-scheme $\operatorname{Isom}_B(\xi,\xi')$. By definition, the two families are isomorphic when pulled back to $\operatorname{Isom}_B(\xi,\xi')$ along $\operatorname{Isom}_B(\xi,\xi')\rightarrow B$. Since taking the coarse moduli space commutes with base change, after replacing $B$ with a resolution of $\operatorname{Isom}_B(\xi,\xi')$, we may assume that the rational map
\((\cX,c\cD)\dashrightarrow (X,cD)\) is everywhere defined and coincides
with the coarse moduli space morphism.
\end{proof}

The following lemma is a standard result on geometric pushouts (cf. \cite[Definition 4.1 and Theorem 4.2]{alper2024artin}) that was used in the proof of \Cref{lemma_boundedness_constructing_stack}; we record it here due to the lack of an available reference.

\begin{lemma}\label{lemma_pushout}
Let $X$, $Y$, $Z$ be quasi-separated algebraic stacks, flat over a base scheme $B$, and let
\[
\xymatrix{X\ar[r]^\iota \ar[d]_\rho & Y \ar[d] & \\ Z\ar[r] & P\ar[r] &B}
\]
be a geometric pushout diagram in which $\iota$ is a closed embedding and $\rho$ is finite. Then $P$ is also flat over $B$, and for every $T\to B$, the fiber product $P_T$ is the pushout of $X_T$, $Y_T$, and $Z_T$.
\end{lemma}
\begin{proof}
By \cite[Lemma~4.3]{alper2024artin}, the formation of $P$
commutes with flat base change. We first show that the formation of $P$ commutes with arbitrary base change. After replacing $B$ by an affine open cover, we may assume that
$B=\Spec \BB$. Given a morphism $T\to B$ from an affine scheme
$T=\Spec \TT$, consider the pushout diagram
\begin{equation}\label{eq:pushout}
\xymatrix{
X_T
\ar[r]\ar[d]
&
Y_T\ar[d]
\\
Z_T\ar[r]
&
P'.
}\end{equation}
By the universal property of pushouts, there is a natural morphism
$P'\to P$. For any flat morphism $\Spec \AAA\to P$, by \cite[Lemma 4.3]{alper2024artin}, the fiber products
\[
X\times_{P}\Spec\AAA,\qquad
Y\times_{P}\Spec\AAA,\qquad
Z\times_{P}\Spec\AAA
\]
are affine, which we denote by
$\Spec \AAA_1$, $\Spec \AAA_2$, and $\Spec \AAA_3$, respectively. By \cite[Theorem~4.2]{alper2024artin},
$P'\times_{P}\Spec(\AAA\otimes_{\BB}\TT)$
is likewise affine. Since the diagram \eqref{eq:pushout} is a geometric pushout, we have
$P'\times_{P}\Spec(\AAA\otimes_{\BB}\TT)=\Spec \PP'$ for some $\PP'$ fitting into an exact sequence
\[
0 \ \longrightarrow \  \PP' \ \longrightarrow \ 
(\AAA_1\otimes_{\BB}\TT)\times
(\AAA_2\otimes_{\BB}\TT)
 \ \longrightarrow \ 
\AAA_3\otimes_{\BB}\TT
 \ \longrightarrow \  0.
\]
Similarly, since the diagram defining $P$ is itself a geometric pushout, we have
\begin{equation}\label{eq_lemma_pushouts}
0 \ \longrightarrow \  \AAA \ \longrightarrow \ \AAA_1\times 
\AAA_2 \ \longrightarrow \ \AAA_3 \ \longrightarrow \  0.
\end{equation}
Since $\AAA_3$ is flat over $\BB$, we obtain
$\PP'\simeq \AAA\otimes_{\BB}\TT$.
Flatness of $P$ over $B$ then follows from the long exact sequence for $\mathrm{Tor}$ applied to \Cref{eq_lemma_pushouts}.
\end{proof}

In particular, there exists a \emph{flat} family
\(\cX\to \cC\) together with a divisor \(\cD\subseteq \cX\) such that, over the nodal locus of \(\cC\), it agrees with the family constructed above, while over the smooth locus it agrees with \((X,cD)\to C\).

\begin{lemma}\label{lemma_for_boundedness_G_nef}
    After replacing $B$ by an open subset, we may assume that the pairs $(\cX,c\cD)\to \cC$ satisfy $\cO_{\cX}\big(m(K_{\cX/B}+c\cD)\big) = \pi^*\cL$ with $\cL$ nef, for every sufficiently divisible $m$, and that $\cD$ is ample on $\cX$ over $\cC$.
\end{lemma}
\begin{proof}
    Let $\cG$ be the line bundle on $C$ such that $K_{X/B}+cD \sim_\bQ \pi^*\cG$ as in \Cref{lemma_line_bundle_G}, and $\cL$ be the pullback of $\cG$ to $\cC$. The topological type of the family $C\to B$ is constant on connected components of $B$, so there are finitely many connected components of $B$ on which $\cG$ is nef on the fibers of $C\to B$; we replace $B$ by this union of connected components. Since being ample is an open condition and $\cC\to B$ is proper, after further shrinking $B$ we may also assume that $\cD$ is ample on $\cX$ over $\cC$.
\end{proof}
We are now ready to finish the proof. We constructed a smooth scheme $B$ with flat morphisms
\[
(\cX,c\cD)\to \cC \to B
\]
with $\cC$ a family of twisted curves, $\cD$ a $\bQ$-Cartier divisor ample over $\cC$, and satisfying conditions (LCY), (Sing) and (Qmap). Consider then $\lambda_{\Chow}$ the Chow line bundle for $\cX\to \cC$ with respect to the polarization $K_{\cX/\cC}+(c+\epsilon_0)\cD$. It follows from \Cref{cor_degree_of_chow_minimized_by_pull_back_from_KSBA} that $\lambda_{\Chow}$ is nef. In particular, since the line bundle $\cG$ is nef, the locus of points $b\in B$ where $\cG\otimes \lambda_{\Chow}^{m}$ is ample does not depend on $m$ as long as $m>0$. Therefore, also the locus of points $b\in B$ where
\[
K_{\cX/B} + (c +\epsilon^2)\cD + \pi^*\lambda_{\Chow}^{\otimes \epsilon}
\]
descends to an ample line bundle on the coarse space of $\cX_b$ does not depend on $\epsilon$, as long as $0<\epsilon \ll 1$. Consider then $M>0$ divisible enough such that $\cO_{\cX}\big(M(K_{\cX/B} + c\cD +\epsilon^2\cD)\big)\otimes \pi^*\lambda_{\Chow}^{\otimes \epsilon M}$
descends to a line bundle on $X$. The locus where (Stab) is satisfied is open since being ample is an open condition \cite[\href{https://stacks.math.columbia.edu/tag/0D2S}{Tag 0D2S}]{stacks-project}.
\end{proof}
The following is an immediate consequence of \Cref{thm_weak_boundedness} with \(d=1\).

\begin{cor}\label{cor_interior_is_bounded}
The set of non-degenerate stable log Calabi--Yau fibrations
\(\pi:(X,cD)\to C\) with numerical invariant \(\Phi\) is bounded.
\end{cor}

\begin{proof}[Proof of \Cref{thm_bounding_fam}]
By \Cref{cor_interior_is_bounded}, non-degenerate Calabi--Yau fibrations with numerical invariant $\Phi$ are bounded; let
\[
(X_{B^\circ},cD_{B^\circ})\ \stackrel{\pi^\circ}{\longrightarrow} \ C_{B^\circ} \ \longrightarrow \ B^\circ
\]
be such a family, and let $m>0$ be an integer such that the $m$-th multiple $\lambda_{\Chow,\pi^\circ}^{\otimes m}$ of the Chow ($\bQ$-)line bundle associated to $\pi^\circ:\big(X_{B^\circ},(c+\epsilon_0)D_{B^\circ}\big)\to C_{B^\circ}$ is a line bundle. By \Cref{cor_degree_of_chow_minimized_by_pull_back_from_KSBA}, for every non-degenerate stable log Calabi--Yau fibration $(X,cD)\xrightarrow{\pi} C$ with numerical invariant $\Phi$ and for any sufficiently divisible $m>0$, the line bundle \[(\lambda_{\Chow,\pi}\otimes f_\pi^{*}L_{\Chow}^{-1})^{\otimes m}\] is effective, where $f_{\pi}:C\to \bf{PCY}$ is the morphism induced by $\pi$ and $L_{\Chow}$ is the Chow line bundle on $\bf{PCY}$. Fix such an $m$, and let $\delta$ be an upper bound for the degree of the restriction of $m\lambda_{\Chow,\pi^\circ}$ to the fibers of $C_{B^\circ}\to B^\circ$. Then the degree of the corresponding map $C\to \textbf{PCY}$ is likewise bounded by $\delta$.

By \Cref{cor:number-components-bounded}, for any stable log Calabi--Yau fibration $(X_0,cD_0)\to C_0$ arising as a limit of $(X,cD)\to C$, the curve $C_0$ has a bounded number of irreducible components. Then, by \Cref{thm_weak_boundedness}, all stable log Calabi--Yau fibrations of non-degenerate type belong to a bounded family.
\end{proof}

\section{Representability by a Deligne--Mumford stack}\label{sec:Representability by a Deligne--Mumford stack}In this section, we construct an algebraic stack which, over reduced bases, represents the functor of families of stable log Calabi--Yau fibrations of non-degenerate type. We will also prove that the resulting algebraic stack is proper and Deligne--Mumford; the following is our main result.
\begin{thm}\label{thm:representability}
For every admissible quadruple $\Phi$, there exists a reduced proper Deligne--Mumford stack $\overline{\MM}_{\Phi}$ whose $S$-points, for every reduced scheme $S$, parametrize families
\[(\cX_S,c\cD_S)\ \longrightarrow \ \cC_S\ \longrightarrow \ S\] of stable log Calabi--Yau fibrations of non-degenerate type with numerical invariant $\Phi$.
\end{thm}

\begin{proof}
Once the existence of the moduli stack $\ove{\MM}_{\Phi}$ has been established, its properness follows immediately from \Cref{Valuative criterion} and \Cref{sec:Boundedness of stable log Calabi--Yau fibrations}. The construction of the stack proceeds in two stages:
\begin{enumerate}
    \item[(1)] Construct finite-type ambient stacks, parametrizing successively: fibers; fibrations; and fibrations equipped with a divisor.
    \item[(2)] Obtain the moduli stack by successively imposing the defining conditions.
\end{enumerate}
\smallskip

Let
\[
(\cX_B,c\cD_B)\to \cC_B\to B \ \ \ \ \textup{and} \  \ \ \ (X_B,cD_B)\to C_B\to B
\]
be the family constructed in \Cref{thm_bounding_fam} and its relative coarse space family respectively. Fix an $\epsilon$-coefficient \(0<\epsilon_0<1\) for $\Phi$ (cf. \Cref{defn:epsilon-coefficient}), and \(d\in \bN\) such that

\begin{itemize}
    \item the divisors \(dK_{\cX_B/\cC_B}\), \(d\cD_B\), \(d(c+\epsilon_0)\cD_B\), \(dK_{X_B/C_B}\), \(dD_B\), \(d(c+\epsilon_0)D_B\) are Cartier; and
    
    \item the line bundle
    \[
    \cL_B:=d\bigl(K_{\cX_B/\cC_B}+(c+\epsilon_0)\cD_B\bigr)
    \]
    is relatively very ample with vanishing higher direct images over \(\cC_B\).
\end{itemize} It follows that the fibers of \((\cX_B,\cL_B)\to \cC_B\) are parametrized by subloci of finitely many Hilbert schemes \(\Hilb_i\), \(i=1,\ldots,m\). Since the argument is identical for each \(\Hilb_i\), we suppress the index \(i\) and write simply \(\Hilb\). Let
\[\begin{tikzcd}[ampersand replacement=\&] {\sX_{\Hilb}} \&\& {\bP^{r}\times \Hilb} \\ \& \Hilb \arrow[hook, from=1-1, to=1-3] \arrow[from=1-1, to=2-2] \arrow[from=1-3, to=2-2] \end{tikzcd}.\]
be the universal family over \(\Hilb\).

\begin{Step_dm} Construct the Hilbert scheme of fibers.
\end{Step_dm}

\begin{lemma}\label{step_1_dm__get_lc_stable_locus_in_H}
There exists a $\PGL_{r+1}$-equivariant closed subscheme of \(\Hilb\), still denoted by \(\Hilb\), together with a $\PGL_{r+1}$-equivariant dense open subscheme \(\Hilb^\circ\subseteq \Hilb\), such that
\begin{enumerate}
\item[\textup{(1)}] the universal family
$\sX_{\Hilb^\circ}\to \Hilb^\circ$ is locally stable;
\item[\textup{(2)}] every fiber of
\((\cX_B,\cL_B)\to \cC_B\) endowed with the fixed embedding into \(\bP^r\) given by the complete linear system of $\cL_B$ is represented by a point of \(\Hilb\);
\item[\textup{(3)}] the subset of \(\Hilb\) corresponding to the generic fibers of
\((\cX_B,\cL_B)\to \cC_B\)
is dense in \(\Hilb\).
\end{enumerate}
\end{lemma}

\begin{proof}
By \cite[Theorem 3.2]{Kol23}, there exists a functorial locally closed partial decomposition (cf. \cite[Definition 10.83]{Kol23})
\[
\iota\colon Z \ \longrightarrow \ \Hilb
\]
whose points parametrize locally stable families. Since the decomposition is functorial, the natural \(\PGL_{r+1}\)-action on \(\Hilb\) preserves \(Z\). The family $(\cX_B,\cL_B)\rightarrow\cC_B$ gives rise to a constructible subset of \(\Hilb\), and all generic fibers occurring in this family belong to finitely many connected components of \(Z\). Replacing \(Z\) by the \(\PGL_{r+1}\)-orbit of the union of these connected components, we may assume that every fiber arising from \Cref{thm_bounding_fam} is represented by a point of \(Z\).

Finally, replace \(\Hilb\) with the scheme-theoretic image in \(\Hilb\) of each of the chosen connected components of \(Z\). Since \(\iota\) is a locally closed embedding, it follows from \cite[\href{https://stacks.math.columbia.edu/tag/01RG}{Tag 01RG}]{stacks-project} that, after this replacement, \(\Hilb\) satisfies all the desired properties.
\end{proof}

\smallskip

\begin{Step_dm} Construct the \emph{Hom} stack of fibrations. \end{Step_dm} \begin{lem}\label{step_2_consider_hom_stack} There exists an algebraic stack locally of finite presentation over $\bC$, parametrizing all flat fibrations $\cX \to \cC$ over a twisted curve of genus $g$ with fibers in $\Hilb$ such that, away from possibly finitely many smooth points of $\cC$, the morphism $\cX \to \cC$ is locally stable. \end{lem}

\begin{proof}
Since the embedding of the fibers of $\cX_B\to \cC_B$ in $\bP^r$ is via a complete linear series, there is an injective set-theoretic map from the set of isomorphic classes of all fibers of $\cX_B\to \cC_B$ to the set of $\PGL_{r+1}$-orbits in $\Hilb$. Consider the quotient stack
\[
\HH := [\Hilb/\PGL_{r+1}].
\]
An $S$-point of $\HH$ consists of a diagram
\[
\begin{tikzcd}[ampersand replacement=\&]
	\sX_S \&\& \sP_S \\
	\& S
	\arrow[hook, from=1-1, to=1-3]
	\arrow[from=1-1, to=2-2]
	\arrow[from=1-3, to=2-2]
\end{tikzcd}
\]
where $\sP_S\to S$ is a family of Brauer--Severi varieties of dimension $r$, the morphism
$\sX_S\hookrightarrow \sP_S$ is a closed immersion whose geometric fibers are represented by points of $\Hilb$, and $\sX_S\to S$ is flat. The open subscheme $\Hilb^\circ\subseteq \Hilb$ constructed in \Cref{step_1_dm__get_lc_stable_locus_in_H} induces an open dense substack
\[
\HH^\circ \ \coloneqq \ [\Hilb^\circ/\PGL_{r+1}] \ \subseteq\  \HH,
\]
which parametrizes locally stable families.

Let $\MM_g^{\tw}$ denote the moduli stack of twisted curves of genus $g$ constructed in \cite{olsson2007log}, and let
$\cC_g^{\tw}\to \MM_g^{\tw}$ be the universal twisted curve. By
\cite[Theorem~1.10 and Corollary~1.12]{olsson2007log},
$\MM_g^{\tw}$ is a smooth Artin stack locally of finite type over $\bC$. Consider the Hom-stack
\[
\underline{\Hom}  \ \coloneqq \ 
\underline{\Hom}_{\MM_g^{\tw}}
\bigl(\cC_g^{\tw},\, \HH\times \MM_g^{\tw}\bigr),
\]
an $S$-point of which consists of a twisted curve
$\cC_S\to S$ of genus $g$ together with a morphism
$\cC_S\to \HH$. Since $\cC_g^{\tw}\to \MM_g^{\tw}$ is proper and flat and $\HH$ is algebraic, quasi-separated and locally of finite presentation with affine stabilizers, \cite[Theorem~1.2]{HR19} implies that $\underline{\Hom}$ is an algebraic stack locally of finite presentation over $\bC$. Pulling back the universal family over $\HH$ yields a diagram
\[
\begin{tikzcd}[ampersand replacement=\&]
	{\cX_{\underline{\Hom}}} \&\& {\cC_{\underline{\Hom}}} \\
	\& \underline{\Hom}
	\arrow["{\pi_{\underline{\Hom}}}", from=1-1, to=1-3]
	\arrow[from=1-1, to=2-2]
	\arrow[from=1-3, to=2-2]
\end{tikzcd}
\]
whose structure morphisms are flat and proper.

After replacing $\underline{\Hom}$ by a suitable open substack, still denoted by $\underline{\Hom}$, we may furthermore assume that every morphism
$\cC\to \HH$ represented by a geometric point of $\underline{\Hom}$ satisfies the following properties:
\begin{enumerate}
    \item the complement of
    $\HH^\circ\times_{\HH}\cC \subseteq \cC$
    is either empty or consists of finitely many smooth points of $\cC$;
    \item the morphism $\cC\to \HH$ is representable. Indeed, representability is an open condition by \cite[Theorem~7.2.7]{AV02}.
\end{enumerate}
\end{proof}

From now on, we write \(\underline{\Hom}\) for the algebraic stack in \Cref{step_2_consider_hom_stack}.

\begin{Step_dm}
Construct the stack of fibrations in pairs.

\begin{lemma}\label{step_3_hilb_scheme_to_parametrize_D}
There exists an algebraic stack $\GG$ locally of finite presentation over $\bC$, whose $S$-points correspond to diagrams
\[
\xymatrix{
\cD \ar[r]^{\iota} &
\cX \ar[rd] \ar[rr]^{\pi} &&
\cC \ar[ld] \\
&& S &
}
\]
such that:
\begin{enumerate}
    \item[\textup{(1)}] $\iota$ is a closed immersion;
    \item[\textup{(2)}] the composition $\cD\to S$ is flat;
    \item[\textup{(3)}] for every geometric point $s\in S$, the fiber $\cD_s\subseteq \cX_s$ is a closed subscheme of pure codimension one;
    \item[\textup{(4)}] the fibration $\cX\to \cC\to S$ is induced by a morphism
    $S\to \underline{\Hom}$.
\end{enumerate}
\end{lemma}
\end{Step_dm}

\begin{proof}
Let $\GG'$ denote the \emph{Hilbert stack} (cf. \cite{HR15}) of closed substacks of
$\cX_{\underline{\Hom}}$ that are flat over $\underline{\Hom}$.
By \cite[Theorem~4.4]{HR15}, the stack $\GG'$ is algebraic.
Moreover, it is locally of finite presentation over $\underline{\Hom}$.
Indeed, this property can be checked after passing to an \'etale cover of $\underline{\Hom}$, where it follows from \cite[Theorem~1.1]{OS03}.

Pulling back the universal family $\cX_{\underline{\Hom}}
\to
\cC_{\underline{\Hom}}
\to
\underline{\Hom}$ along the morphism $\GG'\to \underline{\Hom}$, we obtain
\[
\cD_{\GG'} \ \hookrightarrow\  \cX_{\GG'}
\ \longrightarrow \ \cC_{\GG'}
\ \longrightarrow \ \GG',
\]
where $\cD_{\GG'}\hookrightarrow \cX_{\GG'}$ is the universal closed substack and
$\cD_{\GG'}\to \GG'$ is flat. By \cite[Chapter III, Corollary 9.6]{Har77}, the subset
\[
\GG
\ 
\coloneqq \ \bigl\{
h\in \GG'
\mid
(\cD_{\GG})_h \subseteq (\cX_{\GG})_h
\text{ is of pure codimension one}
\bigr\}
\]
is a union of connected components of $\GG'$.
Hence $\GG$ is an open and closed substack of $\GG'$, and $\GG$ satisfies the required properties.
\end{proof}

To simplify the notation, we henceforth suppress the subscript $\GG$ and denote by \[ (\cX,\cD) \xrightarrow{\pi} \cC\to \GG \qquad\text{and}\qquad (X,D) \to C\to \GG \] the universal family over \(\GG\) and its relative coarse moduli space, respectively.

\smallskip

The second stage of the proof is to impose additional conditions on $\GG$ and show that each of them is represented by a functorial locally closed substack.

\begin{Step_dm}\label{step_construction_dm_stack_locally_stable} Representability of local stability and Cartierness of divisors. 
\end{Step_dm}
 
 \begin{lem}\label{lem:construction_dm_stack_locally_stable}
There exists a functorial locally closed partial decomposition of $\GG$ parametrizing families
\((\cX,\cD)\to \cC\to S\) such that the coarse moduli space families
\[
(X,(c+\epsilon_0)D)\to S
\qquad\text{and}\qquad
(X,cD)\to S
\]
are locally stable, and
\[
\bigl(K_{X/S}+(c+\epsilon_0)D\bigr),\qquad K_{X/S},
\qquad\text{and}\qquad D
\]
are $\bQ$-Cartier.
\end{lem}

\begin{proof}
Recall that the locus
\[
\Bigl\{h\in \GG \ \Big| \ \sX_h \text{ is reduced}\Bigr\}
\]
is open. Replacing $\GG$ by this open substack, we may therefore assume that
$\sX\to \GG$ has reduced fibers. Since the quotient of a reduced scheme by a finite group remains reduced, the relative coarse moduli space $X\to \GG$ also has reduced fibers. Consequently, the relative singular loci $\operatorname{Sing}(\sX\to \GG)$ and $\operatorname{Sing}(X\to \GG)$
have codimension at least one in every fiber. Since both singular loci are closed, upper semicontinuity of
fiber dimensions \cite[\href{https://stacks.math.columbia.edu/tag/0DRQ}{Tag 0DRQ}]{stacks-project}
implies that the locus
\[
\Bigl\{
h\in \GG \;\Big|\;
\text{Sing}(\sX_h)\cap \sD_h
\text{ and }
\text{Sing}(X_h)\cap D_h
\text{ have codimension at least }2
\Bigr\}
\]
is open.
Replacing $\GG$ by the reduced structure on this open substack, it follows
from \cite[Theorem~4.3]{Kol23} that $(X,D)\to \GG$ is a well-defined family
of pairs. Since the relative coarse moduli space of a flat Deligne--Mumford morphism of algebraic stacks is itself flat, the morphism $D\to \GG$ is flat. Applying \cite[Theorem~4.42]{Kol23}, we conclude that the
locus where both
\[
\big(X,(c+\epsilon_0)D\big)\to \GG
\qquad\text{and}\qquad
(X,cD)\to \GG
\]
are locally stable is represented by a locally closed partial decomposition of $\GG$. In particular, by \cite[Definition--Theorem 3.1]{Kol23}, $X\to \GG$ is locally stable as well. Finally, the $\bQ$-Cartierness of
$(K_{X/S}+(c+\epsilon_0)D)$, $K_{X/S}$, and $D$ follows immediately from the local stability of the two families.
\end{proof}

\begin{Step_dm}\label{step_dm_pi_is_generically_ksba_stab_is_open} Representability of KSBA-stability away from finitely many smooth points.

\end{Step_dm}
\begin{lemma}\label{lem:dm_pi_is_generically_ksba_stab_is_open}
Let $(\cX,(c+\epsilon_0)\cD)\to \cC$ be the family over $\GG$. Then the locus in $\GG$ over which this family is KSBA-stable away from a closed substack
$\cW\subseteq \cC$, whose fiber over every geometric point of $\GG$ is a \textup{(}possibly empty\textup{)}
union of smooth points of the corresponding curve, is open.
\end{lemma}

\begin{proof}
The sublocus of $\cD$ where $\cD\to \cC$ is flat is open by
\cite[\href{https://stacks.math.columbia.edu/tag/0399}{Tag 0399}]{stacks-project}.
Let $\cW^{(1)}\subseteq \cC$ denote the image of its complement. The upper semicontinuity of fiber dimensions implies that the locus where
\(\text{Sing}(\cX\rightarrow \cC)\cap \cD\) has codimension at least $2$ in every fiber of $\cX\to \cC$ is open. Let $\cW^{(2)}\subseteq \cC$ be its complement and set
\[
\cW\coloneqq \cW^{(1)}\cup \cW^{(2)}.
\] By upper semicontinuity of fiber dimensions again, there exists an open
substack $\UU\subseteq \GG$ such that for every geometric point
$u\in \UU$,
\[
\cW_u\subseteq \cC_u^{\sm}
\qquad\text{and}\qquad
\dim \cW_u=0.
\] Restricting to $\UU$, \Cref{step_2_consider_hom_stack} yields an open
substack $\cV\subseteq \cC|_{\UU}$ such that
\((\cX,\cD)|_{\cV}\to \cV\)
is a well-defined family of pairs. Moreover,
\((\cC|_{\UU})\setminus \cV\)
is contained in the smooth locus of \(\cC|_{\UU}\to \UU\) and has pure codimension one in every fiber.

By \Cref{prop_descent_KX_D} and 
\Cref{rem:weaken-the-conditions}, a multiple of $K_{\cX/\GG}$ descends to a multiple of $K_{X/\GG}$ on a big open substack. Since $K_{X/\GG}$ is $\bQ$-Cartier by
\Cref{lem:construction_dm_stack_locally_stable}, it follows that $K_{\cX/\GG}$ is
also $\bQ$-Cartier. As $\cC\to \GG$ is Gorenstein, we conclude that
$K_{\cX/\cC}$ is $\bQ$-Cartier. Likewise, since $D$ is $\bQ$-Cartier, $\cD$ is also $\bQ$-Cartier. Applying \cite[Corollary~4.45]{Kol23}, we obtain an open
substack
\(
\cU\subseteq \cV\)
parametrizing those points over which the fibers of
\((\cX,(c+\epsilon_0)\cD)\to \cC\) are semi-log canonical. Since relative ampleness of
\(K_{\cX/\cC}+(c+\epsilon_0)\cD\) is an open condition, there exists an open substack of $\cU$ over which \((\cX,(c+\epsilon_0)\cD)|_{\cU}\to \cU\) is a KSBA-stable family. The complement
\((\cC|_{\UU})\setminus \cU\) is closed in $\cC|_{\UU}$ and, by construction, is contained in the smooth
locus of the fibers of $\cC|_{\UU}\to \UU$. Therefore, the image in $\GG$
of the complement of the locus described in the statement is closed.
Hence the desired locus is open.
\end{proof}

\begin{Step_dm}\label{step_dm_non_genenerate}
Closedness of the non-degenerate type locus.
\end{Step_dm}

\begin{lem}\label{lem:dm_non_genenerate}
There exists a reduced closed substack of $\GG$ parametrizing stable log
Calabi--Yau fibrations of non-degenerate type.
\end{lem}

\begin{proof}
Since being non-degenerate is an open condition, the locus parametrizing non-degenerate stable log Calabi--Yau fibrations forms an open substack of $\GG$. Replacing $\GG$ by the reduced induced
structure on its closure gives the desired reduced closed substack.
\end{proof}

\begin{Step_dm}\label{step_condition_lcy}
Openness of log Calabi--Yau condition.
\end{Step_dm}

\begin{lem}\label{lem:condition_lcy}
    The condition \textup{(LCY)} in \Cref{def_CY_fibration} is functorial and defines an open substack of $\GG$.
\end{lem}

\begin{proof}
Since the pairs under consideration form a bounded family, there exists a sufficiently divisible integer $m>0$ such that for every pair $(Y,cD_Y)$ arising either as a fiber of the bounding family
\[
(\cX_B,c\cD_B)\ \longrightarrow \ \cC_B \ \longrightarrow \ B
\]
of \Cref{thm_bounding_fam}, or as a point of $\PCY$, one has
\[
\cO_Y\bigl(m(K_Y+cD_Y)\bigr)\ \simeq\ \cO_Y.
\] Let
\(
(\cX,c\cD)\xrightarrow{\pi}\cC\to \GG\) be the universal family. By upper semicontinuity of fiber dimension, the locus
\[
\cU\ \coloneqq \ 
\Bigl\{
p\in \cC
\;\Big|\;
\dim\!\bigl(
\pi_*\cO_{\cX}(m(K_{\cX/\cC}+c\cD))
\otimes k(p)
\bigr)
\le 1
\Bigr\}
\]
is open. On the other hand, for every non-degenerate Calabi--Yau fibration
\(
(\cX_0,c\cD_0)\to \cC_0\),
we have
\[
\dim\!\bigl(
\pi_*\cO_{\cX_0}(m(K_{\cX_0/\cC_0}+c\cD_0))
\otimes k(p)
\bigr)
\ =\ 1
\]
for every $p\in \cC_0$.
Since the locus of non-degenerate Calabi--Yau fibrations is open dense in $\GG$
by \Cref{lem:dm_non_genenerate}, it follows that
\[
\dim\!\bigl(
\pi_*\cO_{\sX}(m(K_{\sX/\sC}+c\sD))
\otimes k(p)
\bigr)
=1
\]
for every $p\in \cU$. Let
\(
\Xi\coloneqq \cC\setminus \cU\) be the complement of $\cU$. Since $\cC\to\GG$ is proper, the image of $\Xi$ in $\GG$ is closed.
Therefore its complement is open in $\GG$. By construction, this open substack parametrizes precisely those families satisfying condition \textup{(LCY)}.
\end{proof}

\begin{Step_dm}\label{step_dm_cndition_stab} Openness of stability condition.
\end{Step_dm}

\begin{lem}\label{lem:dm_cndition_stab}
   The locus in $\GG$ parametrizing pairs which satisfy condition \textup{(Stab)} in \Cref{def_CY_fibration} is open.
\end{lem}

\begin{proof}
    This follows from the fact that being ample is an open condition.
\end{proof}  

Therefore, the functor of interest is represented by an algebraic stack $\overline{\MM}_{\Phi}$ of finite type. It remains to prove that $\overline{\MM}_{\Phi}$ is Deligne--Mumford. It suffices to show that the stabilizers of its geometric points are finite. Indeed, any automorphism of a stable log Calabi--Yau fibration induces an automorphism of the corresponding stacky KSBA model, which in turn induces, by \Cref{lemma_tsm_family_is_stable}, an automorphism of the associated twisted stable map. Moreover, if an automorphism $\sigma$ of a stable log Calabi--Yau fibration induces the identity automorphism of the associated twisted stable map, then $\sigma=\Id$. Since twisted stable maps have finite automorphism groups, the desired conclusion follows.
\end{proof}

\section{Projectivity}\label{sec:projectivity}

In this section, we prove the projectivity of the coarse moduli space $\overline{\mathfrak{M}}_{\Phi}$.

\begin{thm}\label{thm:ampleness}
    The coarse moduli space $\ove{\fM}_{\Phi}$ is projective.
\end{thm}
\begin{remark}
    We observe that at the end of the proof of \Cref{thm:ampleness}, we explicitly construct the ample line bundle on $\ove{\fM}_{\Phi}$.
\end{remark}

Throughout this section, we fix an admissible quadruple $\Phi$. We will consider various base schemes $B$, together with a family of stable log Calabi--Yau fibrations of non-degenerate type
\begin{equation}\label{eq:Stable-LCY-fib}
    f\ :\ (\cX,c\cD) \ \stackrel{\pi}{\longrightarrow} \ \cC \ \stackrel{g}{\longrightarrow} \  B
\end{equation}
with numerical invariant $\Phi$ over an integral scheme $B$, with coarse spaces
\[\phi\ :\ (X,cD) \ \stackrel{\psi}{\longrightarrow} \ C \ \stackrel{\tau}{\longrightarrow} \  B.\]
We fix an $\epsilon$-coefficient (cf. \Cref{defn:epsilon-coefficient}) $\epsilon_0$ of $\Phi$, and denote by $\PCY$ the moduli stack of stable Calabi--Yau pairs parametrizing the fibers of $\pi$, and by $\bf{PCY}$ its coarse moduli space. For every stable Calabi--Yau pair $(Y,cH;H)$ parametrized by $\PCY$, we identify it with the KSBA-stable pair $(Y,(c+\epsilon_0)H)$ as before, and let $\mtf{L}_{\Chow}\in \Pic(\bf{PCY})_{\bQ}$ be the descent of the Chow line bundle on $\PCY$ associated to the corresponding universal KSBA-stable family. Let $\lambda_{\Chow}$ be the Chow line bundle associated to $\pi\colon\big(\cX,(c+\epsilon_0)\cD\big)\rightarrow \cC$, and let $\Lambda_{\Cho}$ be its descent on $C$. For any rational numbers $0<\epsilon_1,\epsilon_2<\epsilon_0$ and any $n\in \bN$ such that
\begin{enumerate}
    \item $n\epsilon_2\in \bZ$ and $\lambda_{\Chow}^{\otimes n\epsilon_2}$ is a line bundle; and 
    \item $n(K_{\cX/B}+(c+\epsilon_1)\cD)$ is Cartier.
\end{enumerate}
we define \[
\cL(\epsilon_1,\epsilon_2,n)\ \coloneqq \ \cO_{\cX}\big(n(K_{\cX/B} + (c+\epsilon_1)\cD)\big)\otimes \pi^*\lambda_{\Chow}^{\otimes n\epsilon_2}
\] and its descent on the coarse space $X$ \[
L(\epsilon_1,\epsilon_2,n)\ \coloneqq \ \cO_{X}\big(n(K_{X/B} + (c+\epsilon_1)D)\big)\otimes\psi^*\Lambda_{\Chow}^{\otimes n\epsilon_2}.
\] 
We start by fixing a uniform index. Let \[
    f^{\univ}\ :\ \big(\cX^{\univ},(c+\epsilon_0)\cD^{\univ}\big) \ \stackrel{\pi^{\univ}}{\longrightarrow} \ \cC^{\univ} \ \stackrel{g^{\univ}}{\longrightarrow} \  \overline{\MM}_{\Phi}\] be the universal family over $\overline{\MM}_{\Phi}$ with relative coarse spaces \[\phi^{\univ}\ :\ \big(X^{\univ},(c+\epsilon_0)D^{\univ}\big) \ \stackrel{\psi^{\univ}}{\longrightarrow} \ C^{\univ} \ \stackrel{\tau^{\univ}}{\longrightarrow} \  \overline{\MM}_{\Phi}.\] Then one can choose an integer $N>0$ such that 
    \begin{enumerate}
        \item[(i)] $(\lambda_{\Chow}^{\univ})^{\otimes N}$ is a line bundle on $\cC^{\univ}$ which descends to a line bundle $(\Lambda_{\Chow}^{\univ})^{\otimes N}$ on $C^{\univ}$, where $\lambda_{\Chow}^{\univ}$ is the Chow ($\bQ$-)line bundle associated to $\pi^{\univ}$;
        \item[(ii)] $\omega_{C^{\univ}/\overline{\MM}_{\Phi}}\otimes (\Lambda_{\Chow}^{\univ})^{\otimes N}$ is relatively ample over $\overline{\MM}_{\Phi}$ (cf. \Cref{lemma_ampleness_of_omega_C_plus_lambda}); and
        \item[(iii)] $\mathfrak{L}_{\Chow}^{\otimes N}$ is a very ample line bundle on $\bf{PCY}$.
    \end{enumerate}
\smallskip
\noindent We will use the notation introduced above throughout this section.

\begin{lemma}\label{lemma_lambda_Chow_is_equivalent_to_an_effective_divisor}
For any smooth projective curve $B$ and any stable log Calabi--Yau fibration as in \Cref{eq:Stable-LCY-fib}, there exists an effective Cartier divisor $S$ on $C$ such that:
\begin{enumerate}
    \item[\textup{(1)}] $\Lambda_{\Cho}^{\otimes N}\simeq \cO_C(S)$;
    \item[\textup{(2)}] $\Supp(S)_\eta\subseteq C_\eta^{\sm}$; and
    \item[\textup{(3)}] $\omega_{C/B}(S)$ is relatively ample over $B$.
\end{enumerate}
\end{lemma}

\begin{proof}
By the valuative criterion for the properness of $\textbf{PCY}$ and Condition (Qmap) in \Cref{def_CY_fibration}, there exists a big open substack $\cU\subseteq\cC$, whose complement consists of finitely many smooth points of $\cC\to B$, together with a natural morphism
\[
\rho_{\cU}\colon  \cU \ \longrightarrow  \ \textbf{PCY}.
\]
Since the complement of $\cU$ consists of finitely many smooth points on the surface $\cC$, the pullback $\rho_{\cU}^*\mathfrak L_{\Cho}^{\otimes N}$ extends uniquely to a line bundle on $\cC$, which we denote by $\cL_{\Cho}^{\otimes N}$. Since $\mathfrak L_{\Cho}^{\otimes N}$ is very ample, we may choose an effective divisor on $\textbf{PCY}$ representing it whose pullback to $\cU$, denoted by $\cH_{\cU}$, is disjoint from the nodal locus of $\cC_\eta$. Let $\cH$ be the closure of $\cH_{\cU}$ in $\cC$. Then $\cL_{\Cho}^{\otimes N}\simeq \cO_{\cC}(\cH)$.

On the other hand, \Cref{cor_degree_of_chow_minimized_by_pull_back_from_KSBA}, applied to the generic fiber, yields an effective Cartier divisor $\cE_\eta$ on $\cC_\eta$ such that
\[
\cL_{\Cho}^{\otimes N}|_{\cC_\eta}\ 
\simeq\ 
\lambda_{\Cho,\pi}^{\otimes N}|_{\cC_\eta}(-\cE_\eta),
\]
whose support consists of points over which $\bigl(\cX_\eta,(c+\epsilon_0)\cD_\eta\bigr)\to\cC_\eta$ fails to be KSBA-stable. Let $\cE$ denote the closure of $\cE_\eta$ in $\cC$, which is supported on the smooth locus of $\cC\rightarrow B$ by the Condition (Qmap). Since $\cC$ is $S_2$, we obtain $\lambda_{\Cho}^{\otimes N}\simeq\cO_{\cC}(\cH+\cE)$.

Set $\cS:=\cH+\cE$. Since $\rho_{\cU}:\cU\to \bf{PCY}$ factors through a
morphism $U\to \bf{PCY}$ from the coarse space of $\cU$, the divisor $\cH_U$
is the pullback of a Cartier divisor on $U$. Moreover, since $\Supp(\cE)$ is
disjoint from the stacky locus of $\cC$, the coarse moduli map is an
isomorphism near $\Supp(\cE)$, so $\cE|_{\cU}$ also descends to a Cartier
divisor on $U$. Combining the two, the restriction $\cS|_{\cU}$ descends to
a Cartier divisor on $U$. Finally, since the stacky locus of $\cC$ is
contained in $\cU$, the coarse moduli map $\cC\to C$ is an isomorphism away
from $\cU$, and it follows that $\cS$ descends to a Cartier divisor $S$ on
$C$. Therefore $\Lambda_{\Cho}^{\otimes N}\simeq\cO_C(S)$, and $\Supp(S)_\eta\subseteq C_\eta^{\sm}$ by construction. Finally, the relative ampleness of $\omega_{C/B}(S)$ follows from \Cref{lemma_ampleness_of_omega_C_plus_lambda}.
\end{proof}

Let $\delta>0$ be a lower bound as in \Cref{lemma_slct_positive_is_open} for the universal family $f^{\univ}$, i.e., for any stable log Calabi--Yau fibration $(\cX,c\cD)\to\cC$ with numerical invariant $\Phi$ of non-degenerate type, $\slct(\cX,c\cD;\cX_p)>\delta$ for every $p\in \cC^{\sm}$. Since the degree of $(\Lambda^{\univ}_{\Chow})^{\otimes N}$ on the fiber $C_b$ is locally constant as $b$ varies over $\overline{\MM}_{\Phi}$, there exists an integer $d_0$ bounding these degrees from above. We fix such a bound $d_0$ and a rational number $\epsilon_2>0$ such that $d_0\epsilon_2<\delta$.

\begin{lemma}\label{lemma_nefness_of_vb}
For every rational number $\epsilon_1 \in (0, \epsilon_2^2)$, there exists an integer $n = n(\epsilon_1) > 0$ such that, for every smooth projective curve $B$ with generic point $\eta$ and every stable log Calabi--Yau fibration as in \Cref{eq:Stable-LCY-fib} over $B$, the following holds:
\begin{enumerate}
    \item[\textup{(1)}]
    \(f_*\cL(\epsilon_1,\epsilon_2,n)=
    g_*\!\big(
    \pi_*\cO_{\cX}\bigl(n(K_{\cX/B}+(c+\epsilon_1)\cD)\bigr)
    \otimes
    \lambda_{\Cho}^{\otimes n\epsilon_2}
    \big)
    \)
    is a nef vector bundle;
    \item[\textup{(2)}]
    \(
    \cO_{\cX}\bigl(n(K_{\cX/B}+(c+\epsilon_1)\cD)\bigr)
    \otimes
    \pi^*\lambda_{\Cho}^{\otimes n\epsilon_2}
    \)
    descends to an ample line bundle on $X$; and
    \item[\textup{(3)}]
    \(
    g_*\bigl(\omega_{\cC/B}^{\otimes n}\otimes \lambda_{\Cho}^{\otimes n^2}\bigr)
    \)
    is a nef vector bundle.
\end{enumerate}
\end{lemma}

\begin{proof}
By \Cref{lemma_lambda_Chow_is_equivalent_to_an_effective_divisor}, there exists an effective Cartier divisor $S$ on $C$ such that $\Lambda_{\Cho}^{\otimes N} \simeq \cO_{C}(S)$ and $\omega_{C/B}(S)$ is ample over $B$. Since $S_\eta$ is a Cartier divisor of degree at most $d_0$ whose support lies in $\cC_\eta^{\sm}$, the choice of $\delta$ and $\epsilon_2$ ensures that the pair $\bigl(\cC_\eta, \tfrac{\epsilon_2}{N}S_\eta\bigr)$ is semi-log canonical. Hence, by condition (Stab), for every $\epsilon_1 \in (0,\epsilon_2^2)$ the pair
\[
\textstyle
\bigl(
X_\eta,\,
(c+\epsilon_1)D_\eta
+\tfrac{\epsilon_2}{N}\pi_\eta^*S_\eta
\bigr)
\]
is KSBA-stable. Now choose an integer $n>0$ satisfying the following conditions:
\begin{enumerate}
    \item $n\epsilon_2$ is divisible by $N$;
    \item $n(K_{\cX^{\univ}/\overline{\MM}_{\Phi}}+(c+\epsilon_1)\cD^{\univ})$ is Cartier and descends to $X^{\univ}$;
    \item $n(K_{X^{\univ}/\overline{\MM}_{\Phi}}+(c+\epsilon_1)D^{\univ})\otimes (\psi^{\univ})^*(\Lambda^{\univ}_{\Chow})^{\otimes n\epsilon_2}$ is very ample over $\overline{\MM}_{\Phi}$;
    \item $\omega_{\cC^{\univ}/\overline{\MM}_{\Phi}}\otimes (\lambda_{\Chow}^{\univ})^{\otimes n}$ is a nef line bundle on $\cC^{\univ}$, whose $n$-th tensor power descends to a very ample line bundle on $C^{\univ}$.
\end{enumerate}
With this choice of $n$, we have
\[
\textstyle
\cO_{\cX}\bigl(n(K_{\cX/B}+(c+\epsilon_1)\cD)\bigr)
\otimes
\pi^*\lambda_{\Cho}^{\otimes n\epsilon_2}
\ \simeq\
\cO_{\cX}\!\left(
n(K_{\cX/B}+(c+\epsilon_1)\cD)
+\tfrac{n\epsilon_2}{N}\pi^*S
\right),
\]
and hence
\[
\textstyle
g_*\!\left(
\pi_*\cO_{\cX}\bigl(n(K_{\cX/B}+(c+\epsilon_1)\cD)\bigr)
\otimes
\lambda_{\Cho}^{\otimes n\epsilon_2}
\right)
\ \simeq\
f_*\cO_{\cX}\!\left(
n(K_{\cX/B}+(c+\epsilon_1)\cD)
+\tfrac{n\epsilon_2}{N}\pi^*S
\right).
\]
Since $\bigl(X_\eta,\,(c+\epsilon_1)D_\eta+\pi_\eta^*\tfrac{\epsilon_2}{N}S_\eta\bigr)$ is KSBA-stable, \cite[Theorem~1.11]{fuj18_semipos} implies that the vector bundle on the right-hand side is nef. The same argument shows that $g_*\bigl(\omega_{\cC/B}^{\otimes n}\otimes \lambda_{\Cho}^{\otimes n^2}\bigr)$ is also nef.
\end{proof}

\smallskip

\begin{proof}[Proof of \Cref{thm:ampleness}]

By \cite[Theorem 16.6]{LMB00}, there exists a proper scheme $B$ admitting a finite surjective morphism to $\overline{\MM}_{\Phi}$. Composing with the natural map $\overline{\MM}_{\Phi}\to \ove{\mtf{M}}_{\Phi}$, we see that $B\to \ove{\mtf{M}}_{\Phi}$ is finite and surjective as well. Consider the pullback to $B$ of the universal family over $\ove{\MM}_{\Phi}$, which we denote as in \Cref{eq:Stable-LCY-fib}, keeping all the notation introduced there. 

We apply \Cref{thm:ampleness-lemma}, which is a mild generalization of Koll\'ar's ampleness lemma (cf. \cites{kol90,KP17}). To this end, we will construct two vector bundles $\cW$ and $\cQ$, each of which is a direct sum of three vector bundles. The first two summands keep track of the pair $X$ and the closed embedding $D\hookrightarrow X$, while the third keeps track of the map $X\to C$ by representing it as the closed embedding of $X$ into $X\times C$ given by its graph.

Let $N$ and $\epsilon_2$ be the two numbers fixed previously. Choose a rational number $\epsilon_1\in(0,\epsilon_2^2)$ and an integer $M=M(\epsilon_1)>0$ such that $M\epsilon_1\in \bZ$ and $M\epsilon_1\cD^{\univ}$ is Cartier. Take a sufficiently divisible $n_0>0$ such that
\begin{enumerate}
    \item $L(\epsilon_1,\epsilon_2,n_0)$ is relatively very ample over $B$;
    \item $R^i\phi_{*}L(\epsilon_1,\epsilon_2,n_0)=0$ for all $i>0$; and
    \item the ideal sheaf of $M\epsilon_1 D_b$ is generated by sections of $L(\epsilon_1,\epsilon_2,n_0)|_{X_b}$ for every $b\in B$; equivalently, the line bundle $L(\epsilon'_1,\epsilon_2,n_0)|_{X_b}$ is globally generated, where $\epsilon_1'\coloneqq (1-\frac{M}{n_0})\epsilon_1\in (0,\epsilon_1)$.
\end{enumerate}
Then, by cohomology and base change,
$\phi_{*}L(\epsilon_1,\epsilon_2,n_0)$ is a locally free sheaf on $B$,
and the relative complete linear system defines a closed immersion
\[
\iota\colon
X\ \hookrightarrow\ \bP_B\ \coloneqq \ 
\bP_B\!\bigl(\phi_*L(\epsilon_1,\epsilon_2,n_0)\bigr),
\]
such that
\[
\iota^*\cO_{\bP_B}(1)
=
L(\epsilon_1,\epsilon_2,n_0).
\] Twisting the exact sequence
\begin{equation}\label{eq:exact-sequence}
   0 \ \longrightarrow \  \cO_{X}(-M\epsilon_1D)
\ \longrightarrow \   \cO_{X}
\ \longrightarrow \   \cO_{M\epsilon_1D}
\ \longrightarrow \   0 
\end{equation}
by $L(\epsilon_1,\epsilon_2,n_0)$, we obtain
\[\textstyle
0
\ \longrightarrow \   L(\epsilon'_1,\epsilon_2,n_0)
\ \longrightarrow \   L(\epsilon_1,\epsilon_2,n_0)
\ \longrightarrow \   L(\epsilon_1,\epsilon_2,n_0)\big|_{M\epsilon_1D}
\ \longrightarrow \   0.
\] Pick a sufficiently divisible integer $n>0$ such that \(L(\epsilon'_1,\epsilon_2,nn_0)\) has no higher direct image and yields a closed embedding \[\textstyle X\ \hookrightarrow \ \bP_B\big(\phi_*L(\epsilon_1',\epsilon_2,nn_0)\big).\] Twisting the exact sequence \Cref{eq:exact-sequence} (with $M$ replaced by $Mn$) by $L(\epsilon_1,\epsilon_2,nn_0)$, we obtain the exact sequence \[\textstyle
0
\ \longrightarrow \   L(\epsilon_1',\epsilon_2,nn_0)
\ \longrightarrow \   L(\epsilon_1,\epsilon_2,nn_0)
\ \longrightarrow \   L(\epsilon_1,\epsilon_2,nn_0)\big|_{Mn\epsilon_1D}
\ \longrightarrow \   0.
\] Together with the fact that $L(\epsilon_1,\epsilon_2,nn_0)|_{Mn\epsilon_1D}$ has vanishing higher direct images over $B$, we obtain the following exact sequence of locally free sheaves on $B$:
\[\textstyle 
0 \ \longrightarrow \ \phi_*L(\epsilon'_1,\epsilon_2,nn_0\big)\ \longrightarrow \  \phi_*L(\epsilon_1,\epsilon_2,nn_0)\ \longrightarrow \ \phi_*\big(L(\epsilon_1,\epsilon_2,nn_0)|_{Mn\epsilon_1D}\big)\ \longrightarrow \  0.
\]
In particular, we have an isomorphism 
\begin{equation}\label{eq:isomorphism}
    \textstyle \det\phi_*L(\epsilon_1,\epsilon_2,nn_0)\ \simeq \
\det \phi_*L(\epsilon_1',\epsilon_2,nn_0)\ \bigotimes \ \det \phi_*\big(L(\epsilon_1,\epsilon_2,nn_0)|_{Mn\epsilon_1D}\big).
\end{equation}
For every $d>0$, the multiplication map induces a surjection 
\begin{equation}\label{eq:surjectivity1}
    \textstyle 
\Sym^d\!
\phi_*L(\epsilon_1,\epsilon_2,nn_0) \ \twoheadrightarrow\ 
\phi_*L(\epsilon_1,\epsilon_2,dnn_0),
\end{equation}
provided that $n$ is sufficiently divisible. It then follows that the composition
\begin{equation}\label{eq:surjectivity2}
\Sym^d\!\phi_*L(\epsilon_1,\epsilon_2,nn_0)\ 
\longrightarrow\ 
\phi_* L(\epsilon_1,\epsilon_2,dnn_0)\ 
\longrightarrow\ 
\phi_*\!\big(
L(\epsilon_1,\epsilon_2,dnn_0)\big|_{Mdn\epsilon_1D}
\big)\end{equation}
is surjective. We will use the domains and codomains of \Cref{eq:surjectivity1} and \Cref{eq:surjectivity2} as two of the direct summands of $\cW$ and $\cQ$, respectively.

We now construct a third summand of $\cW$ and $\cQ$, which will record the maps $X_b\to C_b$.
Consider the closed subscheme $i\colon X\to X\times_B C$, with ideal sheaf $I$, given by the graph of the map $X\rightarrow C$. Let $\pi_1\colon X\times_B C\to X$ and $\pi_2\colon X\times_B C\to C$ be the two projections, and let 
\[
G(\epsilon_1,\epsilon_2,n) \ \coloneqq \ \pi_1^*L(\epsilon_1,\epsilon_2,n)\otimes \pi_2^*(\omega_{C/B}^{\otimes n}\otimes \Lambda_{\Chow}^{\otimes n^2}).
\]
Then we have the exact sequence 
\[
0\ \longrightarrow \  I\otimes G(\epsilon_1,\epsilon_2,n)\ \longrightarrow \  G(\epsilon_1,\epsilon_2,n)\ \longrightarrow \   L(\epsilon_1,\epsilon_2, n)\otimes \psi^*(\omega_{C/B}^{\otimes n}\otimes \Lambda_{\Chow}^{\otimes n^2})\ \longrightarrow \   0.
\]
By the choice of $n$, $\omega_{C/B}\otimes \Lambda_{\Chow}^{\otimes n}$ is relatively ample over $B$. Up to choosing $n$ more divisible, the line bundle $G(\epsilon_1,\epsilon_2,n)$ is relatively very ample over $B$, and both $G(\epsilon_1,\epsilon_2,n)$ and $I\otimes G(\epsilon_1,\epsilon_2,n)$ have vanishing higher direct images to $B$. Let $\mu\colon X\times_B C\to B$ be the composition $\phi\circ \pi_1$, and consider the exact sequence
\[
0\ \longrightarrow\  \mu_*\big(I\otimes G(\epsilon_1,\epsilon_2,n)\big)\ \longrightarrow\ \mu_*G(\epsilon_1,\epsilon_2,n)\ \longrightarrow\ \mu_*\big(L(\epsilon_1,\epsilon_2, n)\otimes \psi^*(\omega_{C/B}^{\otimes n}\otimes \Lambda_{\Chow}^{\otimes n^2})\big)\ \longrightarrow\ 0.
\]
\begin{lem}
   There is an isomorphism \[\mu_*G(\epsilon_1,\epsilon_2,n)\ \simeq\  \phi_*L(\epsilon_1,\epsilon_2,n)\otimes \tau_*\big(\omega_{C/B}^{\otimes n}\otimes\Lambda_{\Chow}^{\otimes n^2}\big).\] 
\end{lem}

\begin{proof}
  Consider the following commutative diagram, in which the square is Cartesian:
\[\begin{tikzcd}[ampersand replacement=\&]
	{X\times_BC} \&\& X \\
	C \&\& B
	\arrow["{\pi_1}", from=1-1, to=1-3]
	\arrow["{\pi_2}"', from=1-1, to=2-1]
	\arrow["\psi"{description}, from=1-3, to=2-1]
	\arrow["\phi"', from=1-3, to=2-3]
	\arrow["\tau", from=2-1, to=2-3]
\end{tikzcd}.\]
By flat base change, we have
\[
\phi^*\tau_*(\omega_{C/B}^{\otimes n}\otimes\Lambda_{\Chow}^{\otimes n^2})\ \simeq \ (\pi_1)_*\pi_2^*(\omega_{C/B}^{\otimes n}\otimes\Lambda_{\Chow}^{\otimes n^2}).
\] 
In particular, since $\phi$ and $\tau$ have connected fibers, we have
\[
\mu_*G(\epsilon_1,\epsilon_2,n)\ \simeq \ \phi_*L(\epsilon_1,\epsilon_2,n)\otimes \tau_*\big(\omega_{C/B}^{\otimes n}\otimes\Lambda_{\Chow}^{\otimes n^2}\big).
\]  
\end{proof}
By \Cref{lemma_nefness_of_vb}, after replacing $n$ by a sufficiently divisible multiple, both $\phi_*L(\epsilon_1,\epsilon_2,n)$ and $\tau_*(\omega_{C/B}^{\otimes n}\otimes\Lambda_{\Chow}^{\otimes n^2})$ are nef, and hence so is their tensor and symmetric products (cf. \cite[Theorem 6.2.12]{Lazarsfeld2}). In particular,
\[
\Sym ^{d_1}\phi_*L(\epsilon_1,\epsilon_2,n)\otimes \Sym^{d_2}\tau_*\big(\omega_{C/B}^{\otimes n}\otimes\Lambda_{\Chow}^{\otimes n^2}\big)
\]
is nef for any $d_1,d_2\in\bN$. Consider the surjection
\begin{equation}\label{eq:surjectivity4}
\Sym^{d_1}\phi_*L(\epsilon_1,\epsilon_2,n)\otimes
\Sym^{d_2}\tau_*\big(\omega_{C/B}^{\otimes n}\otimes\Lambda_{\Chow}^{\otimes n^2}\big)
\ \twoheadrightarrow\ 
\phi_*\!\big(
L(\epsilon_1,\epsilon_2,d_1n)\otimes\psi^*(\omega_{C/B}^{\otimes d_2n}\otimes \Lambda_{\Chow}^{\otimes d_2n^2})
\big)
\end{equation}
obtained by composing the natural multiplication maps. Set $\cW$ to be
\[\textstyle
(\Sym^d \phi_*L(\epsilon_1,\epsilon_2,nn_0))^{\oplus 2}\ \bigoplus\ \Big(
\Sym^{d_1}\phi_*L(\epsilon_1,\epsilon_2,n)\ \bigotimes\ 
\Sym^{d_2}\tau_*\big(\omega_{C/B}^{\otimes n}\otimes\Lambda_{\Chow}^{\otimes n^2}\big)\Big),
\]
and $\cQ$ to be
\[\textstyle
\phi_*L(\epsilon_1,\epsilon_2,dnn_0) \ 
\bigoplus\ 
\phi_*\!\big(
L(\epsilon_1,\epsilon_2,dnn_0)\big|_{Mdn\epsilon_1\cD}
\big) \ \bigoplus\ \Big(\phi_*\!\big(
L(\epsilon_1,\epsilon_2,d_1n)\bigotimes\psi^*(\omega_{C/B}^{\otimes d_2n}\otimes \Lambda_{\Chow}^{\otimes d_2n^2})
\big)\Big).
\]
Let
\[
r_1\coloneqq\rk \phi_*L(\epsilon_1,\epsilon_2,nn_0), \ \ \ 
r_2\coloneqq\rk \phi_*L(\epsilon_1,\epsilon_2,n),\ \ \ 
r_3\coloneqq\rk \tau_*\big(\omega^{\otimes n}_{C/B}\otimes\Lambda_{\Chow}^{\otimes n^2}\big), \ \ \ r_{\cQ}\coloneqq\rk(\cQ).
\]
Then
\[
r_{\cW}\ 
\coloneqq \ \rk\cW \ =\  
2\binom{r_1+d-1}{d}
+
\binom{r_2+d_1-1}{d_1}
\binom{r_3+d_2-1}{d_2}
.
\]
We denote by $\cW_1$, $\cW_2$, and $\cW_3$ (resp. $\cQ_1$, $\cQ_2$, and $\cQ_3$) the three direct summands of $\cW$ (resp. $\cQ$), with corresponding maps $\Lambda_i\colon \cW_i\to \cQ_i$ obtained from \Cref{eq:surjectivity1}, \Cref{eq:surjectivity2}, and \Cref{eq:surjectivity4}. We have a surjective morphism
\[
\Lambda\ \colon \ \cW \ \twoheadrightarrow \ \cQ.
\] Fix a point $b\in B$. The kernel of $(\Lambda_1)_b$ consists of degree-$d$
homogeneous polynomials vanishing along the subvariety
\[
X_b
\ \subseteq\ 
\bP\!  H ^0\bigl(
X_b,
L(\epsilon_1,\epsilon_2,nn_0)|_{X_b}
\bigr),
\]
while the kernel of $(\Lambda_2)_b$ consists of degree-$d$ polynomials vanishing
along $Mdn\epsilon_1D_b$; by the choice of $n_0$, these sections generate the ideal sheaf of
$Mdn\epsilon_1D_b$. Similarly, the kernel of $(\Lambda_3)_b$ consists of bi-homogeneous polynomials of bidegree $(d_1,d_2)$ vanishing along the subvariety $X_b$ in
\[
X_b\ \hookrightarrow \ X_b\times C_b
\ \subseteq \ \bP  H ^0\big(X_b,L(\epsilon_1,\epsilon_2,n)\big)\times \bP  H ^0\big(C_b,\omega_{C_b}^{\otimes n}\otimes \Lambda_{\Cho}^{\otimes n^2}\big).
\] By the construction of $\cW$, the structure group of $\cW$ admits a reduction to
\[
G\ \coloneqq \ \operatorname{Diag}\big(\Sym^d\GL(r_1)\big)\ \times \ \big(\Sym^{d_1}\GL(r_2)\times \Sym^{d_2}\GL(r_3)\big),
\]
where the first factor $\operatorname{Diag}(\Sym^d\GL(r_1))$ is the structure group of $\cW_1\oplus\cW_2$. The quotient map $\Lambda$ therefore determines a set-theoretic map
\[
\Psi\colon
|B|
\longrightarrow
\bigl|\Gr(r_{\cW},r_{\cQ})/G\bigr|,
\]
where a point of $\bigl|\Gr(r_{\cW},r_{\cQ})/G\bigr|$ represents the $G$-orbit of a quotient $\mathbb{C}^{r_{\cW}}\twoheadrightarrow \mathbb{C}^{r_{\cQ}}$. Here, the group $G$ acts via matrices of the form
\[
\begin{pmatrix}
\operatorname{Sym}^d A_1 & 0 & 0 \\
0 & \operatorname{Sym}^d A_1 & 0 \\
0 & 0 & \operatorname{Sym}^{d_1} A_2 \otimes \operatorname{Sym}^{d_2} A_3
\end{pmatrix},
\]
where $A_i$ has size $r_i\times r_i$ for $i=1,2,3$. In particular, for any $b\in |B|$, the point $\Psi(b)$ records:
\begin{enumerate}
    \item the homogeneous ideal of $X_b$, up to a change of coordinates, in
    \[\bP^{r_1-1}=\bP
 H ^0\bigl(
X_b,
L(\epsilon_1,\epsilon_2,nn_0)|_{X_b}
\bigr);\]
    \item the ideal sheaf of the Cartier divisor $Mdn\epsilon_1D_b$ in $X_b$;
    \item the bi-homogeneous ideal of the graph of $X_b\rightarrow C_b$, up to a change of coordinates, in
    \[\bP^{r_2-1}\times \bP^{r_3-1} = \bP  H ^0(X_b,L(\epsilon_1,\epsilon_2,n)|_{X_b})\times \bP  H ^0(C_b,\omega_{C_b}^{\otimes n}\otimes \Lambda_{\Cho}^{\otimes n^2}).\]
\end{enumerate}
Observe that the projection $\bP^{r_2-1}\times \bP^{r_3-1}\to \bP^{r_3-1}$ is equivariant with respect to the action of $\GL(r_2)\times \GL(r_3)$ on the domain, the action of $\GL(r_3)$ on the codomain, and the group homomorphism 
$\GL(r_2)\times \GL(r_3)\to \GL(r_3)$ given by the second projection. In particular, if two closed subschemes $X_b^{(1)}, X_b^{(2)}\subseteq \bP^{r_2-1}\times \bP^{r_3-1}$ lie in the same $\GL(r_2)\times \GL(r_3)$-orbit, their images in $\bP^{r_3-1}$ are projectively isomorphic.

Since $D_b$ is a $\bQ$-Cartier integral divisor, it is uniquely determined by the thickening
$Mdn\epsilon_1D_b$; consequently, for $d\gg 0$, $\Psi(b)$ determines the pair $(X_b,D_b)$ up to automorphisms. Similarly, for $d_1,d_2\gg 0$, the third component recovers the embedded curve $X_b\hookrightarrow \bP^{r_2-1}\times \bP^{r_3-1}$, together with its projection to the second factor, $X_b\to \bP^{r_3-1}$; the Stein factorization of this projection is isomorphic to the map $f_b\colon X_b\to C_b$.

We claim that $\Psi$ has finite fibers. Indeed, suppose that two points of $B$ have the same image under $\Psi$. Then the corresponding coarse space pairs $(X,D)$ are isomorphic, and the maps $X\to C$ are isomorphic as well. By \Cref{lem:uniquely-determined}, there are only finitely many log Calabi--Yau fibrations representing the point $b$; since the map $B\to \overline{\MM}_{\Phi}$ is finite by assumption, it follows that $\Psi$ has finite fibers.

By \Cref{lemma_nefness_of_vb}, the vector bundle $\cW$ is nef. Thus all the assumptions of \Cref{thm:ampleness-lemma} are satisfied, and hence
\begin{equation}\label{eq:det-Q-ample}
\begin{split}
\det\cQ\ =\ &\det \phi_*L(\epsilon_1,\epsilon_2,dnn_0)\ \otimes\ \det \phi_*L(\epsilon_1,\epsilon_2,dnn_0)\big|_{Mdn\epsilon_1D}\\
&\otimes\ \det\Big(\phi_*L(\epsilon_1,\epsilon_2,d_1n)\otimes\psi^*\big(\omega_{C/B}^{\otimes d_2n}\otimes \Lambda_{\Chow}^{\otimes d_2n^2}\big)\Big)
\end{split}
\end{equation}
is ample. Since $\det \phi_*L(\epsilon_1',\epsilon_2,dnn_0)$ is nef, it follows that $\det \cQ \otimes \det \phi_*L(\epsilon_1',\epsilon_2,dnn_0)$ is also ample. By \Cref{eq:isomorphism}, we conclude that
\begin{equation}\label{eq:line-bundle-ample}
\big(\det\phi_*L(\epsilon_1,\epsilon_2,dnn_0)\big)^{\otimes 2}\ \otimes\ \det\Big(\phi_*L(\epsilon_1,\epsilon_2,d_1n)\otimes\psi^*\big(\omega_{C/B}^{\otimes d_2n}\otimes \Lambda_{\Chow}^{\otimes d_2n^2}\big)\Big)
\end{equation}
is ample on $B$.

By cohomology and base change, the line bundle in \Cref{eq:line-bundle-ample} is the pullback of a line bundle on $\ove{\MM}_{\Phi}$ that descends to $\ove{\mtf{M}}_{\Phi}$. As $B\rightarrow \ove{\mtf{M}}_{\Phi}$ is finite and surjective, $\ove{\mtf{M}}_{\Phi}$ is projective.
\end{proof}

\appendix
\section[tocentry={Ampleness lemma}]{Ampleness lemma}\label{appendix:Ampleness Lemma}

In this appendix, we slightly generalize the ampleness lemmas of \cite[Lemma 3.9]{kol90} and \cite[Theorem 5.1]{KP17}, for use in \Cref{sec:projectivity}. More precisely, we introduce an additional argument to remove certain technical assumptions present in \emph{loc.\ cit.}, so that the result applies to a considerably wider range of situations.
\begin{theorem}\label{thm:ampleness-lemma}
    Let $X$ be a proper algebraic space, and for each $i=1,...,r$ let $\cW_i$ be a nef vector bundle of rank $w_{i}$ on $X$ such that each $\cW_i$ has reductive structure group $G_i$. Let $\cQ_i$ be vector bundles of rank $q_i$ on $X$ admitting surjective morphisms $\cW_i\twoheadrightarrow \cQ_i$, and let
    \[u\ :\ \big|X\big| \ \longrightarrow \ \bigtimes_{i=1}^r \big|\Gr(w_i,q_i)/G_i\big|\]
    be the induced classifying map on sets of points. If $u$ has finite fibers, then $\otimes_{i=1}^r\det \cQ_i$ is ample.
\end{theorem}
Compared with \cite[Lemma 3.13]{kol90}, we remove the finiteness assumption in \cite[Definition 3.8(ii)]{kol90}. Compared with \cite[Theorem 5.1]{KP17}, we remove the technical assumption that the closure of the image of the structure group $G_i\subseteq \GL(w_i)$ in the projectivization $\bP(\Mat_{w_i\times w_i})$ is normal.

\begin{proof}
    The proof follows the same strategy as in \cite{kol90,KP17}, with several necessary modifications. We therefore only summarize the main ideas and provide complete details for the parts that differ.
   \begin{enumerate}
      \item 
     We proceed as in \cite[3.1.2]{kol90}: one may assume that $X$ is normal and projective, and it suffices to show that $\bigotimes_{i=1}^r \det \cQ_i$ is big. This comes at the cost of weakening the finiteness assumption on $u$: instead of requiring $u$ to have finite fibers over all of $X$, it suffices to assume that the restriction of $u$ to a dense open subset of $X$ has finite fibers.
       \item Following \cite[Lemma 5.6]{KP17}, one can reduce to the case of a single vector bundle $\cW=\cW_i$ of rank $w$ with quotient $\cQ=\cQ_i$ of rank $q$.
       \item Following the argument of \cite[Theorem 5.5]{KP17}, which is a modification of \cite[Lemma 3.13]{kol90}, one can produce the following commutative diagram: \[\begin{tikzcd}[ampersand replacement=\&]
	{\widetilde{\bf{P}}^\nu} \&\& {\widetilde{\bf{P}}} \&\&\&\& {\Gr(w,q)} \\
	{\bf{P}^{\nu}} \&\& {\bf{P}} \&\& {\bf{P}^{\circ}} \&\& {T\subseteq X\times \Gr(w,q)} \\
	\&\& X
	\arrow["{\widetilde{\nu}}", from=1-1, to=1-3]
	\arrow["{\widetilde{\sigma}}", from=1-1, to=2-1]
	\arrow["{\widetilde{\kappa}}", from=1-3, to=1-7]
	\arrow["\sigma", from=1-3, to=2-3]
	\arrow["\vartheta"{description}, from=1-3, to=2-7]
	\arrow["{\widetilde{\pi}}"{description, pos=0.3}, shift left, bend right = 40pt, from=1-3, to=3-3]
	\arrow["\nu", from=2-1, to=2-3]
	\arrow["\pi", from=2-3, to=3-3]
	\arrow[hook', from=2-5, to=1-3]
	\arrow["\kappa"{description}, shift left, from=2-5, to=1-7]
	\arrow[hook', from=2-5, to=2-3]
	\arrow[from=2-5, to=2-7]
	\arrow["\tau"', from=2-7, to=1-7]
	\arrow["\phi"{description}, from=2-7, to=3-3]
\end{tikzcd}\] see the diagram before \cite[Claim 5.5.3]{KP17}. Compared with \emph{loc.\ cit.}, we add the arrow $\nu:\bf{P}^\nu\to \bf{P}$, which is the normalization, and a corresponding blowup $\wt{\sigma}:\wt{\bf{P}}^\nu\to \bf{P}^\nu$. Here, in the diagram, $\pi:{\bf{P}}\rightarrow X$ is a fiber bundle whose fiber is the closure of the structure group $G\subseteq \GL(w)$ in the projectivization $\bP(\Mat_{w\times w})$, which \cite[Theorem 5.1]{KP17} assumes to be normal, an assumption we do not make. The normalization $\nu:\bf{P}^\nu\to \bf{P}$ is also a simultaneous normalization of the fibers of $\pi:{\bf{P}}\rightarrow X$.

\item (Main difference.) The normality assumption in \cite[Theorem 5.1]{KP17} is used in the last formula on \cite[p.~976]{KP17} to show that $\sigma_*\cO_{\wt{\bf{P}}}=\cO_{\bf{P}}$. This need not hold in our setting, but it is automatic that $(\sigma_\nu)_*\cO_{\widetilde{\bf{P}^\nu}} = \cO_{\bf{P}^\nu}$.
Thus, the same argument yields a morphism 
\[
\big(\pi_*\nu_*\cO_{\bf{P}^\nu}(mq)\big)^*\otimes \cO_X(H)\ \longrightarrow \  (\det\cQ)^{\otimes m}
\]
analogous to \cite[Equation (5.5.5)]{KP17}, where $\cO_{\bf{P}^\nu}(mq)\coloneqq \nu^*\cO_{\bf{P}}(mq)$. However, it is less straightforward to check that $\big(\pi_*\nu_*\cO_{\bf{P}^\nu}(mq)\big)^*$ is nef. 

As $\nu$ is finite, the sheaf $\nu_*\cO_{\textbf{P}^\nu}$ is a coherent $\cO_{\bf{P}}$-module. Thus one can take a very ample divisor on $X$ and a sufficiently large $d\in \bN$ such that \[\nu_*\cO_{\bf{P}^\nu}(d)\otimes \pi^*A \ \simeq \  (\nu_*\cO_{\bf{P}^\nu})\otimes \cO_{\bf{P}}(d)\otimes \pi^*A\] is globally generated. Then there exists a surjection
\[\alpha\ \colon \
\bigoplus_{i=1}^k \cO_{\bf{P}}(mq-d)\otimes \pi^*A^{*} \ \twoheadrightarrow \ \nu_*\cO_{\bf{P}^\nu}(mq).
\]
By taking $mq$ large enough, one has a surjection
\[\bigoplus_{i=1}^k
\pi_* \big(\cO_{\bf{P}}(mq-d)\otimes \pi^*A^{*}\big) \ \twoheadrightarrow\  \pi_*\nu_*\cO_{\bf{P}^\nu}(mq).
\] Moreover, as $mq$ is sufficiently large, $R^i\pi_*\nu_*\cO_{\bf{P}^\nu}(mq)=0$ for any $i>0$. As $\cO_{\bf{P}^\nu}$ is flat over $X$, $\pi_*\nu_*\cO_{\bf{P}^\nu}(mq)$ is a locally free sheaf by Grauert's theorem. Taking the dual, one obtains an injection of vector bundles
\[\big(\pi_*\nu_*\cO_{\bf{P}^\nu}(mq)\big)^*\ \hookrightarrow \ \bigoplus_{i=1}^k 
\big(\pi_* \cO_{\textbf{P}}(mq-d)\big)^*\otimes A \ = \ \bigoplus_{i=1}^k \Sym^{mq-d}(\cW^{\oplus w})\otimes A.
\]
Endowing $A$ with the trivial $G$-action, $\big(\pi_*\nu_*\cO_{\bf{P}^\nu}(mq)\big)^*$ becomes a $G$-invariant subbundle of $\bigoplus_{i=1}^k \Sym^{mq-d}(\cW^{\oplus w})\otimes A$. This ambient bundle is nef, being a direct sum of tensor products of symmetric powers of nef vector bundles, and a $G$-invariant subbundle of a nef vector bundle is again nef by \cite[Lemma 3.6]{kol90}. Hence $\big(\pi_*\nu_*\cO_{\bf{P}^\nu}(mq)\big)^*$ is nef.
\item The rest of the proof is identical to that in \cite{KP17}, using the fact that nef vector bundles are weakly positive (cf.~\cite[Remark 5.4]{KP17}). Alternatively, one can proceed as in \cite{kol90}: blowing up the image of the morphism $\big(\pi_*\nu_*\cO_{\bf{P}^\nu}(mq)\big)^*\otimes \cO_X(H)\ \longrightarrow \  (\det\cQ)^{\otimes m}$ yields a morphism $s\colon X'\to X$, and $s^*(\det\cQ)^{\otimes m}\cong \cO_{X'}(s^*H + N + F)$ where $N$ is nef and $F$ is effective. The desired statement then follows as in \cite[p.~251]{kol90}.
   \end{enumerate}
\end{proof}

\begin{remark}
\label{remark:kollar-clarifications}
We record two points concerning \cite[Lemma~3.13]{kol90}, which were explained to us
by J\'anos Koll\'ar. Although they may well be known to experts, we think they will be
helpful for the reader.
\begin{enumerate}
  \item In \cite[Definition~3.8(ii)]{kol90}, it is assumed that the stabilizer of the
    kernel of $\mathcal{W}_x \to \mathcal{Q}_x$ is a finite subgroup of $G$. In
    practice, however, $G$ is usually taken to be $\mathrm{GL}(r)$, and the scalar
    matrices always fix the kernel. Since scalar matrices act trivially on the
    Grassmannian, one can projectivize the group and disregard the scalars; see
    \cite[Theorem~5.9.2]{AlperModuli} for a projectivized version of Koll\'ar's ampleness
    lemma.
  \item A key step in the proof of \cite[Lemma~3.13, page~250]{kol90} is to produce a
    nonzero global section of
    \begin{equation}\label{eq_section_I_want}
    (\det \mathcal{Q})^{\otimes m} \otimes H^{-1}\otimes p_*\cO_{\textbf{P}}(mk)
    \end{equation}
    by first constructing a section of
    \begin{equation}\label{eq_section_I_have}
    g^*\left(p^*((\det \mathcal{Q})^{\otimes m} \otimes H^{-1})\otimes \cO_{\textbf{P}}(mk)\right),
    \end{equation}
    where $\textbf{P}' \xrightarrow{g} \textbf{P}\xrightarrow{p} X$ are the natural maps arising in the
    construction, and $H$ and $\cO_{\textbf{P}}(mk)$ are certain line bundles on $X$ and $\textbf{P}$,
    respectively. In other words, the line bundle in \eqref{eq_section_I_have} is of the form
    $g^*\cL\otimes g^*p^*\cG$ for a line bundle $\cL$ on $\textbf{P}$ and a line bundle $\cG$ on $X$,
    and one is interested in producing a section of $\cG\otimes p_*\cL$.
    If
    \[
    g_*\cO_{\textbf{P}'} \neq \cO_{\textbf{P}},
    \]
    for example if $\textbf{P}$ is not normal and $\textbf{P}'\to \textbf{P}$ is its normalization, it
    is not guaranteed that a section of \eqref{eq_section_I_have} induces a section of
    \eqref{eq_section_I_want}.
    Koll\'ar explained to us that one can instead require the section of \eqref{eq_section_I_have}
    to vanish along the conductor ideal of $\textbf{P}' \to \textbf{P}$; this directly ensures that it
    descends to the desired section of \eqref{eq_section_I_want}.
\end{enumerate}
\end{remark}

\bibliography{citation}

\end{document}